%% file: main.tex
\documentclass{amsart}
\input{preamble.tex}

\begin{document}

\title[Spectral gaps]{Spectral gaps on thick part of moduli spaces}

\author{Yunhui Wu }
\address[Y. ~W. ]{Department of Mathematical Sciences and Yau Mathematical Sciences Center, Tsinghua University, Haidian District, Beijing 100084, China}
\email{yunhui\_wu@mail.tsinghua.edu.cn}

\author{Haohao Zhang}
\address[H. ~Z. ]{Department of Mathematical Sciences and Yau Mathematical Sciences Center, Tsinghua University, Haidian District, Beijing 100084, China}
\email{zhh21@mails.tsinghua.edu.cn}

\begin{abstract}
    In this paper, we study spectral gaps of closed hyperbolic surfaces for large genus. We show that for any fixed $k\geq 1$, as the genus goes to infinity, the maximum of $\lambda_k-\lambda_{k-1}$ over any thick part of the moduli space of closed Riemann surfaces approaches the limit $\frac 14$. 
\end{abstract}

\maketitle

\section{Introduction} 

Let $X_g$ be a closed hyperbolic surface of genus $g\geq2$. The spectrum of the Laplacian on $X_g$ consists of discrete eigenvalues counted with multiplicity, which can be listed in the following increasing order: 
\[ 0 = \lambda_0(X_g) < \lambda_1(X_g) \leq \lambda_2(X_g) \leq\cdots\to\infty. \] 
For any $\epsilon>0$, recall that the $\epsilon$-thick part $\sM_g^{\geq\epsilon}$ of the moduli space $\sM_g$ of closed Riemann surfaces of genus $g$ consists of hyperbolic surfaces with systole at least $\epsilon$. In this paper, we are interested in studying the spectral gaps of closed hyperbolic surfaces in $\sM_g^{\geq\epsilon}$ for large $g$. We prove 
\begin{theorem}\label{main}
    For any fixed positive integer $k\geq1$ and any fixed $\epsilon>0$, we have \[ \limg \max_{X_g \in \sM_g^{\geq\epsilon}} \big( \lambda_k(X_g)- \lambda_{k-1}(X_g) \big) = \frac 14. \]  
\end{theorem}
\noindent If $\epsilon=0$, Theorem \ref{main} was obtained by the joint work of the authors of this paper with Zhu in \cite{WZZ2024} where the maximum is replaced by the supremum. If $k=1$, Theorem \ref{main} tells that 
\[ \limg \max_{X_g \in \sM_g^{\geq\epsilon}} \lambda_1(X_g) = \frac 14. \] 
Indeed, if $\epsilon=0$, this was a breakthrough due to Hide-Magee \cite{HM2023}. In the same spirit of \cite{HM2023}, they \cite[Appendix A]{LM2022} had also first shown that there exists a sequence of finite covers of a fixed closed arithmetic hyperbolic surface such that $\lambda_1\to \frac{1}{4}$. We remark here that the constant $\epsilon>0$ in Theorem \ref{main} can be arbitrarily large, and our methods are quite different from \cite{LM2022}. Meanwhile, it is \emph{unknown} whether a surface realizing $\max\limits_{X_g \in \sM_g^{\geq\epsilon}} \lambda_1(X_g)$ is always arithmetic.

\subsection{History and related works}
The study of spectral gaps of closed hyperbolic surfaces for large genus has a long history. Buser \cite{Buser1978} incorrectly conjectured $\sup\limits_{X_g\in\sM_g}\lambda_1(X_g)\to0$ as $g\to\infty$. However, he observed in \cite{Buser1984} that Selberg's $\frac{3}{16}$ theorem \cite{Selberg1965} and the Jacquet-Langlands theory \cite{JL1970} imply that there exist closed hyperbolic surfaces with $\lambda_1\geq\frac{3}{16}$ for arbitrarily large genus, and in the same paper Buser conjectured that $\frac{3}{16}$ can possibly be replaced by $\frac{1}{4}$. Since Huber \cite{Huber1974} and Cheng \cite{Cheng1975} had proved that 
\[ \limsupg \sup_{X_g\in\sM_g} \lambda_1(X_g) \leq \frac 14, \] 
the number $\frac14$ is the best possible lower bound of $\sup\limits_{X_g\in \sM_g}\lambda_1(X_g)$ for large $g$. 
Again, based on Selberg's $\frac{3}{16}$ theorem, Buser-Burger-Dodziuk \cite{BBD1988} gave an explicit geometric construction of closed hyperbolic surfaces with genus $\to\infty$ and $\lambda_1\to\frac{3}{16}$, by compactifying the congruence covers $\H / \Gamma(N)$ of the modular surface $\H / \mathrm{SL}(2,\Z)$. 
Brooks-Makover \cite{BM2001} proved that 
\[ \liminfg \sup_{X_g\in\sM_g} \lambda_1(X_g) \geq \inf_{N\geq1} \lambda_1( \H / \Gamma(N) ), \] 
by using a different method of compactification and the ``handle lemma" \cite[Lemma 1.1]{BM2001} extracted from \cite{BBD1988}. In particular, Selberg's famous $\frac{1}{4}$-eigenvalue conjecture \cite{Selberg1965, LRS95, Kim03} implies Buser's conjecture. Recently, Hide-Magee confirmed Buser's conjecture above.
\begin{theorem*}[{\cite[Corollary 1.3]{HM2023}}]
    \[ \limg \sup_{X_g\in\sM_g} \lambda_{1}(X_g) = \frac 14. \] 
\end{theorem*}
\noindent They used deep theories of random surfaces in covering model and random matrix to construct finite-area non-compact hyperbolic surfaces with large spectral gap, then applied the ``handle lemma" to these surfaces to obtain closed surfaces with $\lambda_1\to \frac{1}{4}$. 

The higher spectral gaps were studied by the authors of this paper and Zhu in \cite{WZZ2024}, where they generalized the ``handle lemma" to a degenerating mini-max principle. They proved the following. 
\begin{theorem*}[{\cite[Theorem 4.1]{WZZ2024}}]
    If $\eta(g) \in [1,2g-2]$, then
\[ \liminfg \sup_{X_g\in\sM_g}\left( \lambda_{\eta(g)}(X_g) - \lambda_{\eta(g)-1}(X_g) \right) \geq \frac 14. \] 
\end{theorem*}
\noindent Note that the index $\eta(g)$ can depend on $g$. When $\eta(g)\equiv 1$, this theorem is due to Hide-Magee. With the help of an upper bound for general $\lambda_{\eta(g)}$, we will see in Section \ref{sect-last} that if $\eta(g)=o(g)$, then
\[ \limg \sup_{X_g\in\sM_g}\left( \lambda_{\eta(g)}(X_g) - \lambda_{\eta(g)-1}(X_g) \right)=\frac 14. \] 

However, the closed surfaces constructed in \cite{HM2023,WZZ2024} using the Buser-Burger-Dodziuk compactification procedure have very short closed geodesics, which means that they are located near the boundary $\pa \sM_g$ of the moduli spaces $\sM_g$. For any $\epsilon>0$, the $\epsilon$-thick part $\sM_g^{\geq\epsilon}$ of the moduli space is a compact subset by Mumford \cite{Mumford1971}. It is natural to study the spectral gaps on thick parts of the moduli spaces. As introduced above, Hide-Magee \cite[Appendix A]{LM2022} had first constructed a sequence of closed arithmetic hyperbolic surfaces with $\lambda_1\to\frac14$ and systoles uniformly bounded from below by a positive constant; one may also see quite recent work \cite{HP24} of Hide-Petri for a different approach. 

Our proof of Theorem \ref{main} could be regarded as another method to compactify a finite-area non-compact hyperbolic surface, avoiding the appearance of long thin collars in the compactification procedure of Buser-Burger-Dodziuk. 

\subsection{Proof sketch of Theorem \ref{main}}
The proof of Theorem \ref{main} mainly consists of the following three parts, each of which is of independent interest.
\subsubsection{Comparison between two spectra}
We first introduce a procedure which is the key ingredient in our construction of surfaces with long systoles. Recall (see \eg \cite[Definition 2.1]{Brooks1999}) that for any $l>0$,  a finite-area non-compact hyperbolic surface is said to have \emph{large cusps of length $l$} if the horocycles of length $l$ around each cusp are embedded and pairwise disjoint.
Let $X$ be a finite-area non-compact hyperbolic surface with $n$ large cusps, and let $\epsilon_0$ $(\geq \frac{2\epsilon}{\pi})$ be some positive constant. Consider the open subsurface 
\[ X^{\mathrm{fu}} \df X\setminus \bigsqcup_{i=1}^n\cusp_i(\epsilon_0) \] 
of $X$, where $\bigsqcup_{i=1}^n\cusp_i(\epsilon_0)$ denotes the disjoint union of embedded cusps bounded by horocycles of length $\epsilon_0$. 
By the uniformization theorem and the Ahlfors-Schwarz lemma, $X^{\mathrm{fu}}$ admits a unique Poincar\'e metric associated with the complex structure induced from $X$, which makes $X^{\mathrm{fu}}$ into an infinite-area hyperbolic surface having only funnel ends (see Lemma \ref{lem:infinite_area}). The compact convex core of $X^{\mathrm{fu}}$, \ie the subsurface of $X^{\mathrm{fu}}$ with all funnel ends removed, is a compact hyperbolic surface of geodesic boundary, and we denote it by $Y$. See Figure \ref{compactification} for an illustration. 

\begin{figure} 
    \centering 
    \includegraphics[scale = 0.4]{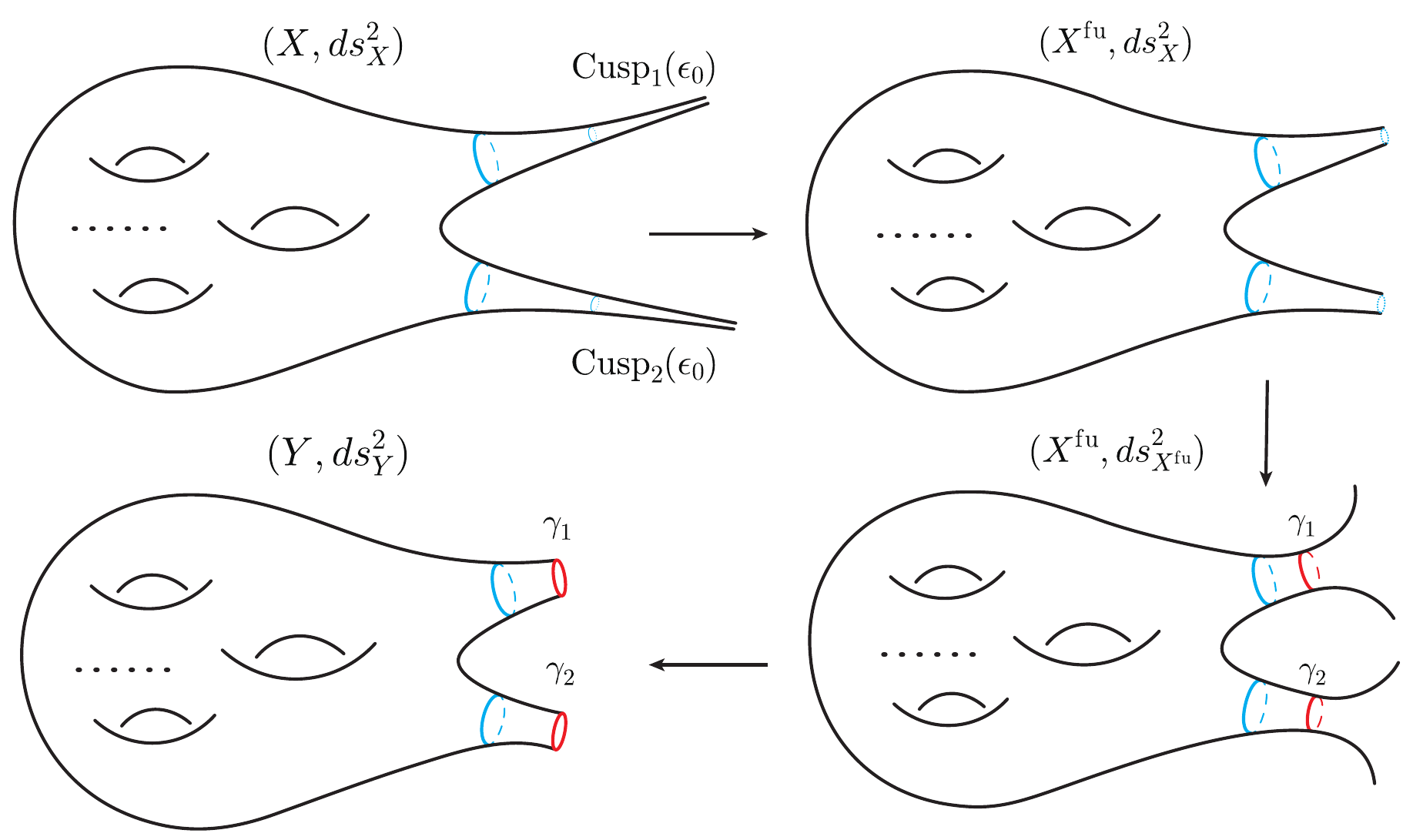}
    \caption{An illustration from $X$ to $Y$.}
    \label{compactification}
\end{figure}

Since $X$ is non-compact of finite area, the spectrum of the Laplacian operator of $X$ consists of absolutely continuous spectrum $\left[\frac{1}{4},\infty\right)$ and possible discrete eigenvalues in $[0,\infty)$. Denote by $\bar\lambda_1(X)$ the minimal non-zero spectrum of $X$, \ie 
\[ \bar\lambda_1(X)=\min\left\{\frac{1}{4},\ \lambda_1(X) \right\} \] 
where we set $\lambda_1(X)=\infty$ if $X$ has no positive discrete eigenvalue. Then we compare the first non-zero Neumann eigenvalue $\sigma_1(Y)$ of $Y$ with the minimal non-zero spectrum $\bar\lambda_1(X)$ of $X$. The below result is the first main result of this paper, and its proof is inspired by the work of Brooks \cite{Brooks1999} and Brooks-Makover \cite{BM2001}. 
\begin{theorem}\label{thm:comparison_eigenvalue}
    Let $X$ be a finite-area non-compact hyperbolic surface, and let $Y$ be constructed as above. For any $0<\delta<1/3$, there exists an $l_{\delta,\epsilon_0}>0$ only depending on $\delta$ and $\epsilon_0$ such that if $X$ has large cusps of length $l\geq l_{\delta,\epsilon_0}$, then 
 \[ \sigma_1(Y) \geq (1-4\delta)\cdot\bar\lambda_1(X) - \delta. \] 
\end{theorem}

\subsubsection{Uniform stability of first Neumann eigenvalues} We will also see that all components of the boundary $\pa Y$ of $Y$ have almost the same length (see Theorem \ref{thm:comparison_geometry}), provided that $X$ has large cusps. Next we make a precise small perturbation of $Y$ into a bordered compact hyperbolic surface $Y'$ with all boundary geodesics having the same length such that we can glue certain copies of $Y'$ into a closed hyperbolic surface. Here a bordered compact hyperbolic surface means that its boundary consists of simple closed geodesics. The following theorem is the second main result of this paper, which roughly says that the first Neumann eigenvalue is uniformly stable under small perturbations of the boundary of a bordered compact hyperbolic surface.
\begin{theorem}\label{thm:stability}
    Let $Y_{g,n}$ be a bordered compact hyperbolic surface with genus $g\geq1$, and let $\{\gamma_i\}_{1\leq i\leq n}$ be its boundary components. Suppose there exists a constant $\ell>0$ such that 
    \[ \Forall 1\leq i\leq n,\ \length(\gamma_i) \leq \ell, \] 
    then for any $(\delta_1, \delta_2,\cdots, \delta_n)\in \R^n$, with 
    \[ \delta \df \max_{1\leq i\leq n}\left\{ \abs{\delta_i} \right\} \] 
    sufficiently small, there exists a bordered compact hyperbolic surface $Y'_{g,n}$ with boundary components $\{\gamma'_i\}_{1\leq i\leq n}$, such that the followings hold: 
    \begin{enumerate}
        \item $\Forall 1\leq i\leq n$, $\ell(\gamma'_i) = (1 + \delta_i)\cdot\ell(\gamma_i)$. 
        \item There exists a bi-Lipschitz map $\mu$ from $Y_{g,n}$ to $Y'_{g,n}$ with Lipschitz constant $K_\mu$ satisfying \[ \abs{K_\mu-1} = O(\sqrt{\delta}), \] where the implied constant depends only on $\ell$.
        \item $\abs{ \sigma_1(Y'_{g,n}) / \sigma_1(Y_{g,n}) - 1 } = O(\sqrt{\delta})$, where the implied constant depends only on $\ell$. 
    \end{enumerate}
\end{theorem}
\noindent The proof of Theorem \ref{thm:stability} is technical. We will compare the geometry of two hyperbolic pairs of pants with two same boundary length (see Proposition \ref{prop:mu_pants}) that contains heavy computations in two-dimensional hyperbolic geometry (see Subsections \ref{subs-l1} and \ref{subs-l2}). We emphasize here that the bound $O(\sqrt{\delta})$ in Part (2) and (3) of Theorem \ref{thm:stability} is independent of $Y_{g,n}$, which is essential in the proof of Theorem \ref{main}.

\subsubsection{Endgame of the proof of Theorem \ref{main} based on random surfaces}

Now one can choose a sequence of finite-area non-compact hyperbolic surfaces $\{\sX_i\}$, all having large $\bar\lambda_1$, large cusps and long systoles, to serve as a starting point. The existence of such a sequence (see Proposition \ref{prop:surfaces}, which is essentially due to Hide-Magee \cite{HM2023}; Magee \cite{Magee2024} and Klukowski-Markovi\'c \cite{KM2024}; Nica \cite{Nica1994}; and Puder-Zimhoni \cite{PZ2024}) is highly non-trivial, and guaranteed by the very recent developments of random surfaces in covering model. Starting with $\{\sX_i\}$, by Theorem \ref{thm:comparison_eigenvalue} and \ref{thm:stability} we prove the following proposition.

\begin{proposition}\label{prop:piece}
    For any fixed $\epsilon>0$, any sufficiently small $\delta>0$, and each sufficiently large $g$, there exists a bordered compact hyperbolic surface $\sY_{g,2}(\epsilon,\epsilon)$ of genus $g$ with two boundary geodesics of same length $\epsilon$ such that 
    \begin{enumerate}
        \item all simple closed geodesics in $\sY_{g,2}$ have length $\geq \epsilon$. 
        \item $\pa\sY_{g,2}$ has long half-collars of width $\epsilon/2$.
        \item $\sigma_1(\sY_{g,2}) \geq \frac 14-O(\delta)$ where the implied constant depends only on $\epsilon$. 
    \end{enumerate} 
\end{proposition}

\noindent The rest of the proof follows the line of the proof of \cite[Theorem 4.1]{WZZ2024}. For each $k\geq1$, and for each large enough genus $g$, we can glue $k$ pieces of surfaces given by Proposition \ref{prop:piece} to obtain a closed hyperbolic surface $\sZ_g\in\sM_g^{\geq\epsilon}$ (see Figure \ref{fig:gluing} for illustration). Finally, by the Mini-max principle, we will see that as $g\to \infty$, it is not hard to see that $\lambda_{k-1}(\sZ_g)\to 0$ (see Corollary \ref{cor:eigenvalue_bound}); and it follows from Part (3) of Proposition \ref{prop:piece} that $\lambda_{k}(\sZ_g)\to \frac{1}{4}$. This tells that $\sZ_g$ has large $k$-th spectral gap, finishing the proof of Theorem \ref{main}. 
\begin{figure}
    \centering
    \includegraphics[scale = 0.30]{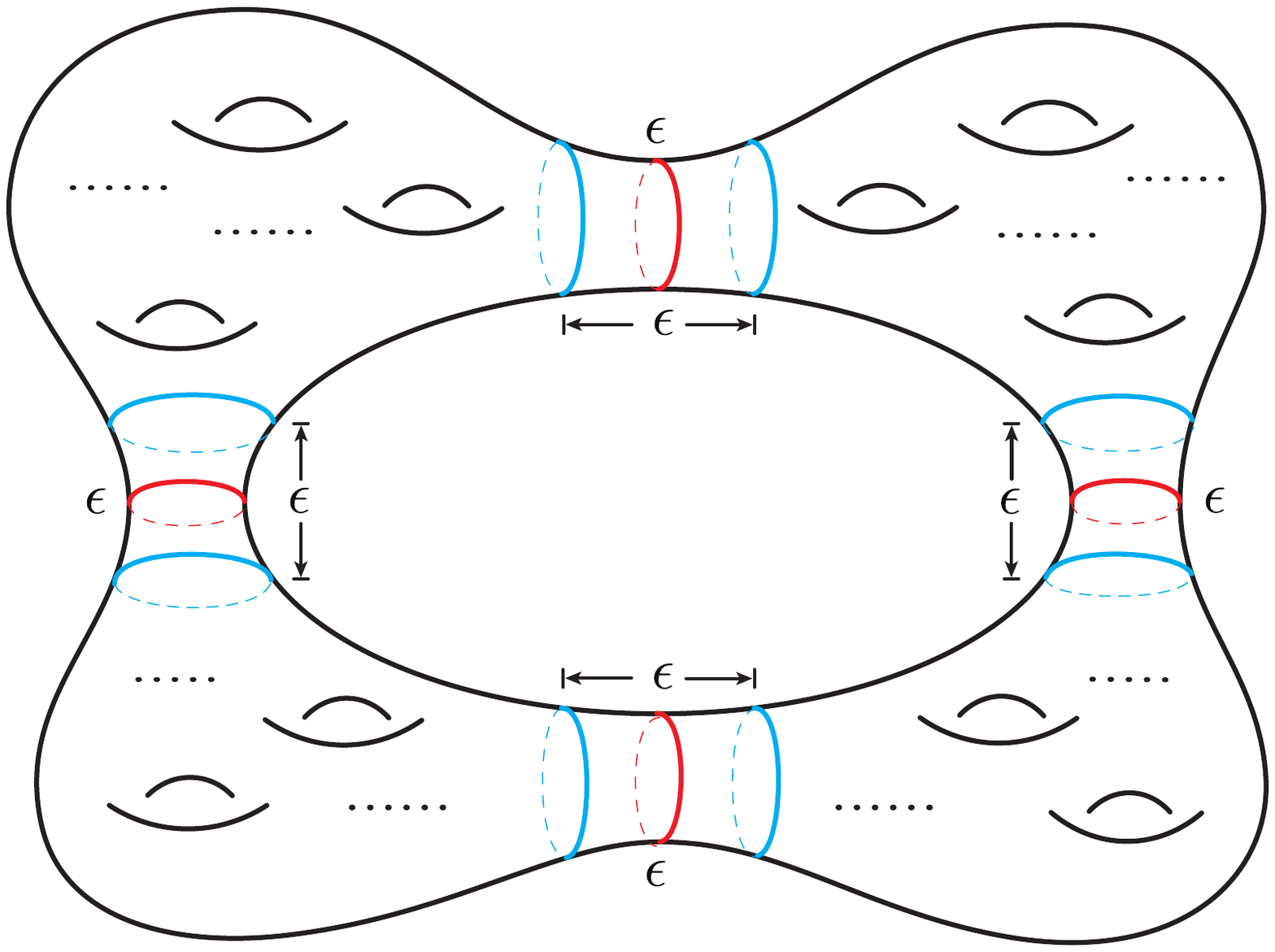}
    \caption{The construction of desired closed hyperbolic surfaces of large genus in Theorem \ref{main} when $k=4$.}
    \label{fig:gluing}
\end{figure}

\subsection*{Plan of the paper}
In Section \ref{section2} we recall and prove some necessary properties on two-dimensional hyperbolic geometry and spectral geometry of hyperbolic surfaces which will be applied later. In Section \ref{section3} we use results from the theory of random surfaces in the covering model to find finite-area non-compact hyperbolic surfaces with large $\bar\lambda_1$, large cusps and long systoles. We prove a comparison theorem between the Poincar\'e metrics of $X^{\mathrm{fu}}$ and $X$ in Section \ref{section4}, and a proposition on mass distribution of eigenfunctions on collars in Section \ref{section5}. In Section \ref{section6} we prove Theorem \ref{thm:comparison_eigenvalue}. In Section \ref{section7} we prove Theorem \ref{main} assuming Theorem \ref{thm:stability}, and in Section \ref{section8} we prove Theorem \ref{thm:stability}. We will introduce several further remarks in Section \ref{sect-last}.

\subsection*{Acknowledgements} We would like to thank Michael Magee and Doron Puder for their interest and comments. The first named author is particularly grateful to Bram Petri for many insightful discussions and sharing his ideas on the possible extension of Theorem \ref{main} to the case of arithmetic hyperbolic surfaces. 
The authors thank Xuwen Zhu for the aforementioned joint work \cite{WZZ2024}. We also thank all the participants in our seminar on Teichm\"uller theory for helpful discussions on this project. The first named author is partially supported by NSFC grants No. 12171263, 12361141813, and 12425107.

\tableofcontents

\section{Preliminaries}\label{section2}
\subsection{Poincar\'e metrics on some domains}
A complete Riemannian metric on a Riemann surface is called a Poincar\'e metric if it has constant Gaussian curvature $-1$. Denote by $\mathbb D$ the unit disc in the complex plane, by $\mathbb D^*$ the punctured unit disc and by $A_{R_1,R_2}$ the annulus 
\[ \{ z\in\C: \ 0 < R_1 < \abs{z} < R_2 < +\infty \}. \] 
By the uniformization theorem and the Ahlfors-Schwarz lemma, each of them admits a unique Poincar\'e metric in its conformal class. 
These metrics are listed below.
\begin{lemma}\label{lem:metrics_r} 
    Let $r = \abs{z}$ and $\theta\in {\R} / {(2\pi\Z)} $. 
    \begin{enumerate}
        \item The Poincar\'e metric $ds_O^2$ on $\mathbb D^*$ is \[ \left( \frac{1}{r \log\frac1r} \right)^2 \abs{\dif z}^2. \] 
        \item The Poincar\'e metric $ds_{R}^2$ on $A_{R_1,R_2}$ is \[ \left( \frac{\pi}{\log(R_2/R_1)} \cdot \frac{1}{r \sin\frac{\pi(\log R_2 - \log r)}{\log (R_2/R_1)}} \right)^2 \abs{\dif z}^2. \] 
    \end{enumerate}
\end{lemma}
\begin{proof}
    The formula of $ds_{O}^2$ is well-known. We sketch a proof of the formula of $ds_{R}^2$ here for completeness (because we cannot find a precise reference).
   The map defined by 
    \[ z\mapsto f(z)\df\exp\frac{\pi(i\log R_2-z)}{\log (R_2/R_1)} \] 
    is a biholomorphic map from the horizontal strip 
    \[ \left\{z\in\C:\ \log R_1 < y < \log R_2 \right\} \] 
    to $\H$, thus the Poincar\'e metric on the horizontal strip above is 
    \begin{align*}
        & \frac{1}{(\im f(z))^2} \abs{\dif f(z)}^2\\ 
        = & \frac{1}{\left( \im \exp\frac{\pi(i\log R_2-z)}{\log (R_2/R_1)} \right)^2} \abs{\frac{\pi}{\log (R_2/R_1)} \exp\frac{\pi(i\log R_2-z)}{\log (R_2/R_1)}}^2 \abs{\dif z}^2\\
        = & \left( \frac{\pi}{\log (R_2/R_1)} \frac{1}{\sin \frac{\pi(\log R_2-y)}{\log(R_2/R_1)}} \right)^2 \abs{\dif z}^2. 
    \end{align*}
    Moreover, since the map $z\mapsto e^{-iz}$ is a holomorphic covering map from this horizontal strip to the annulus $A_{R_1,R_2}$, the metric density function $h_{R}(z)$ of $ds_{R}^2$ satisfies 
    \[ h_{R}(e^{-iz})\cdot\abs{-ie^{-iz}} = \frac{\pi}{\log (R_2/R_1)} \frac{1}{\sin \frac{\pi(\log R_2-y)}{\log(R_2/R_1)}}. \] 
    Since $y=\log\abs{e^{-iz}}$, one has 
    \[ h_{R}(z) = \frac{\pi}{\log (R_2/R_1)}\frac{1}{\abs{z}\sin\frac{\pi(\log R_2-\log\abs{z})}{\log(R_2/R_1)}} = \frac{\pi}{\log (R_2/R_1)}\frac{1}{r\sin\frac{\pi(\log R_2-\log r)}{\log(R_2/R_1)}}. \] 
    The proof is complete. 
\end{proof}

Let $(\rho,\theta)$ be the pair of hyperbolic polar coordinates of $\mathbb D$ centered at the origin, where $\rho = 2\arctanh |z|$. The hyperbolic metrics on $\mathbb D^*$ and $A_{R_1,R_2}$ in hyperbolic polar coordinates can be given as follows.
\begin{lemma}\label{lem:metrics_rho}
    Let $\rho = 2\arctanh \abs{z}$ and $\theta\in {\R} / {(2\pi\Z)} $. 
    \begin{enumerate}
        \item The Poincar\'e metric $ds_O^2$ on $\mathbb D^*$ is \[ \left( \frac{1}{\sinh\rho\cdot\log\coth(\rho/2)} \right)^2 \left( \dif\rho^2 + \sinh^2\!\rho \dif\theta^2 \right). \] 
        \item The Poincar\'e metric $ds_{R}^2$ on $A_{R_1,R_2}$ is \[ \left( \frac{\pi}{\log(R_2/R_1)} \cdot \frac{1}{\sinh \rho\sin\frac{\pi(\log R_2 - \log \tanh(\rho/2))}{\log (R_2/R_1)}} \right)^2 \left( \dif\rho^2 + \sinh^2\!\rho \dif\theta^2 \right). \] 
    \end{enumerate}
\end{lemma}
\begin{proof}
The conclusion follows from a direct computation by taking substitutions of $r=\tanh\frac{\rho}{2}$ and $dr=\frac{d\rho}{2\cosh^2\frac{\rho}{2}}$ into Lemma \ref{lem:metrics_r}. We omit the details here. 
\end{proof}

\begin{remark*}\label{rmk:cylinder_type}
    The hyperbolic punctured disk $(\D^*,ds_O^2)$ is also called a parabolic cylinder. 
    And $(A_{R_1,R_2},ds_{R}^2)$ is called a hyperbolic cylinder. 
    They are known as elementary hyperbolic surfaces. 
\end{remark*}

Denote by $h_{O}(\rho)$ and $h_{R_1,R_2}(\rho)$ the metric density functions of $ds_{O}^2$ and $ds_{R}^2$, respectively. 
Note that these hyperbolic metrics are pairwise conformal to each other (on the annulus $A_{R_1,R_2}$). Since $A_{R_1,R_2}\subsetneq\mathbb{D}^*\subsetneq\mathbb{D}$, by the domain monotonicity of Poincar\'e metrics, one has 
\[ h_{R_1,R_2}(\rho) > h_O(\rho) > 1. \] 
Moreover, we have the following lemmas, which will be used to prove Theorem \ref{thm:comparison_geometry}.

\begin{lemma}\label{lem:calculus1}
    The function $h_O$ satisfies the following. 
    \begin{align*}
        \lim_{\rho\to\infty} &h_{O}(\rho) = 1,\\
        \lim_{\rho\to\infty} &\frac{\dif h_{O}}{\dif\rho}(\rho) = 0,\\
        \textrm{and\quad} \lim_{\rho\to\infty} &\frac{\dif^2 h_{O}}{\dif\rho^2}(\rho) = 0.
    \end{align*}
\end{lemma}
\begin{proof}
    The results are implicitly contained in Brooks \cite[Lemma 1.2 \& Lemma 1.3]{Brooks1999}. A short proof is given here for completeness. Let 
    \[ H_O \df \frac{1}{h_O} = \sinh\rho \cdot \log\coth\frac{\rho}{2}, \] 
    then 
    \begin{equation*}
        \frac{\dif h_{O}}{\dif\rho} = -\frac{\dif H_{O}/\!\dif\rho}{H_{O}^2}\ \textrm{ \ and\ \ }
        \frac{\dif^2 h_{O}}{\dif\rho^2} = \frac{2(\dif H_{O}/\!\dif\rho)^2 - H_{O}\cdot (\dif^2 H_{O}/\!\dif\rho^2)}{H_{O}^3}. 
    \end{equation*}
    Therefore it suffices to prove that as $\rho$ tends to infinity, one has 
    \[ H_O(\rho)\to1, \ \frac{\dif H_O}{\dif\rho}(\rho)\to0, \ \textrm{and\ } \frac{\dif^2 H_O}{\dif\rho^2}(\rho)\to0. \] 
    Since as $\rho$ tends to infinity, one has 
    \begin{equation*}
        \log\coth\frac{\rho}{2} \sim 1-\tanh\frac{\rho}{2} \sim \frac{1}{\sinh\rho}, 
    \end{equation*}
    thus 
    \begin{equation*}
        \begin{aligned}
            H_{O}(\rho) 
            & = \sinh\rho \cdot \log\coth\frac{\rho}{2} \to 1,\\ 
            \frac{\dif H_O}{\dif\rho}(\rho)
            & = \cosh\rho \cdot \log\coth\frac{\rho}{2} - 1 \to 0,\\ 
            \textrm{and\quad} \frac{\dif^2 H_O}{\dif\rho^2}(\rho) 
            & = \sinh\rho \cdot\log\coth\frac{\rho}{2} - \coth\rho \to 0. 
        \end{aligned}
    \end{equation*}
    Then the lemma follows. 
\end{proof}

\begin{lemma}\label{lem:calculus2}
    Let $R_1\in(0,\, 1)$ and $\delta>0$. Then for each sufficiently large $\rho_0$, there exists an $R(\rho_0)$ satisfying 
\[ \tanh\frac{\rho_0+1}{2} < R(\rho_0), \] 
    such that for any $R_2\in[R(\rho_0),\,1]$, 
    the following inequalities hold: 
    \begin{align*}
        \max_{\rho_0 \leq \rho \leq \rho_0+1} &\abs{h_{R_1,R_2}(\rho)-1} \leq \delta,\\
        \max_{\rho_0 \leq \rho \leq \rho_0+1} &\abs{\frac{\dif h_{R_1,R_2}}{\dif\rho}(\rho)} \leq \delta,\\
        \textrm{and\quad} \max_{\rho_0 \leq \rho \leq \rho_0+1} &\abs{\frac{\dif^2 h_{R_1,R_2}}{\dif\rho^2}(\rho)} \leq \delta. 
    \end{align*}
\end{lemma}
\begin{proof}
    Let $t_0 \df 1-\tanh(\rho_0/2)$ and let $R(\rho_0) \df 1-t_0^3$. We are going to prove that, when $t_0$ is sufficiently small (hence $\rho_0$ will be sufficiently large), this $R(\rho_0)$ will satisfy our desired properties. By the Mean Value Theorem, 
    \begin{equation}\label{eqn:calculus1}
        \begin{aligned}
            \tanh\frac{\rho_0+1}{2} 
            &  =  \tanh\frac{\rho_0+1}{2}-\tanh\frac{\rho_0}{2}+\tanh\frac{\rho_0}{2}\\ 
            &\leq \frac{1}{2\cosh^2(\rho_0/2)}+\tanh\frac{\rho_0}{2}\\ 
            &  =  \frac{1-(1-t_0)^2}{2} + 1-t_0 = 1-\frac{t_0^2}{2}, 
        \end{aligned}
    \end{equation}
    thus, for $t_0$ sufficiently small, we have $\tanh((\rho_0+1)/2) < 1-t_0^3 = R(\rho_0)$. 
    
    Now let $t \df 1-\tanh(\rho/2)$. If $\rho\in[\rho_0,\, \rho_0+1]$, then by \eqref{eqn:calculus1}, $t\in[t_0^2/2,\, t_0]$. Hence $t_0$ tends to zero if and only if $t$ tends to zero. For any $R_2\in[R(\rho_0),\,1]$, similar to the proof of Lemma \ref{lem:calculus1}, let 
    \begin{equation}\label{eqn:calculus2}
        H_{R}\df \frac{1}{h_{R_1,R_2}} = \frac{\log(R_2/R_1)}{\pi}\sinh \rho\sin\frac{\pi(\log R_2 - \log \tanh(\rho/2))}{\log (R_2/R_1)}. 
    \end{equation}
    then it suffices to prove that as $t_0$ tends to zero, we have 
    \[ \max_{\rho_0 \leq \rho \leq \rho_0+1}\abs{H_{R}(\rho)-1}\to 0,\, \max_{\rho_0 \leq \rho \leq \rho_0+1}\abs{\frac{\dif H_{R}}{\dif\rho}(\rho)}\to 0,\, \max_{\rho_0 \leq \rho \leq \rho_0+1}\abs{\frac{\dif^2 H_{R}}{\dif\rho^2}(\rho)}\to 0. \] 
    Since $0\leq1-R_2\leq1-R(\rho_0)= t_0^3=o(t)$ as $t\to0$, we have 
    \begin{equation}\label{eqn:calculus3}
        \begin{aligned}
            \log R_2 - \log \tanh\frac{\rho}{2} 
            & = R_2 - 1 + o(1-R_2) + 1-\tanh\frac{\rho}{2} + o(1-\tanh\frac{\rho}{2})\\
            & = -t_0^3 + o(t_0^3) + t + o(t)\\
            & = t+o(t). 
        \end{aligned}
    \end{equation}
    Moreover, we also have that as $t\to 0$, 
     \begin{equation}\label{eqn:calculus4}
       \begin{aligned}
          \sinh \rho 
            &= \frac{2\tanh(\rho/2)}{1-\tanh^2(\rho/2)}
            = \frac{1}{t} + O(1),\\ 
            \cosh \rho &\sim \sinh \rho=\frac{1}{t} + O(1).
        \end{aligned}
    \end{equation}
    From \eqref{eqn:calculus2}, \eqref{eqn:calculus3} and \eqref{eqn:calculus4}, we deduce that as $t_0$ tends to zero, for $H_{R}$ we have 
    \begin{align*}
        \max_{\rho_0 \leq \rho \leq \rho_0+1}\abs{H_{R}(\rho)-1} 
        & = \max_{t_0^2/2\leq t\leq t_0}\abs{\left(\frac1t+O(1)\right)\left(t+o(t)\right) - 1}\\ 
        & = \max_{t_0^2/2\leq t\leq t_0}\abs{1+o(1)-1}\\ 
        & = o(1); 
    \end{align*}
    since the derivative of $\left( \log\tanh(\rho/2) \right)$ is $1/\sinh\rho$, for $\dif H_{R}/\!\dif\rho$ we have 
    \begin{align*}
        \max_{\rho_0 \leq \rho \leq \rho_0+1}\abs{\frac{\dif H_{R}}{\dif\rho}(\rho)} 
        &= \max_{t_0^2/2\leq t\leq t_0}\abs{\cosh \rho\cdot\left(t+o(t)\right) - \cos\frac{\pi(\log R_2 - \log \tanh(\rho/2))}{\log (R_2 / R_1)}}\\
        &= \max_{t_0^2/2\leq t\leq t_0}\abs{1+o(1)-1+o(1)}\\
        &= o(1); 
    \end{align*}
    and for $\dif^2 H_{R}/\!\dif\rho^2$ we have 
    \begin{align*}
        \max_{\rho_0 \leq \rho \leq \rho_0+1}\abs{\frac{\dif^2 H_{R}}{\dif\rho^2}(\rho)} 
        &= \max_{t_0^2/2\leq t\leq t_0}\abs{\sinh \rho \cdot (t+o(t)) - \frac{1+o(1)}{\tanh \rho}+\frac{t+o(t)}{\sinh \rho}}\\
        &= \max_{t_0^2/2\leq t\leq t_0}\abs{1+o(1)-1+o(1)}\\
        &= o(1). 
    \end{align*}
    Then the lemma follows. 
\end{proof}

\subsection{Geometry of hyperbolic surfaces}
Let $S$ be a complete hyperbolic surface with possibly non-empty geodesic boundaries and cusps. For each homotopy class of a closed curve in $S$, if it is not homotopic to a point or a puncture, there exists a unique closed geodesic representative $\gamma$, having the minimal length among all of its homotopy classes. In particular, $\gamma$ is smooth, hence 
\begin{equation}\label{eqn:geodesic_smooth}
    \length_{S}(\gamma) = \inf\left\{ \length_{S}(\gamma'):\ \gamma'\simeq\gamma, \textrm{\,and $\gamma'$ is smooth} \right\}, 
\end{equation}
where $\length_S$ is the length function with respect to the hyperbolic metric on $S$. 

A closed geodesic is called simple if it has no self-intersection points. 
Let $\gamma$ be a simple closed geodesic in $S$, then the collar $\collar(\gamma, w)$ of $\gamma$ with width $w$ is defined by 
\begin{equation*}
    \collar(\gamma, w) \df \{ p\in S:\ \dist(p,\gamma) \leq w \}. 
\end{equation*}
Recall that the collar lemma asserts that 
\begin{theorem}[{Collar lemma, see \eg \cite[Theorem 4.4.6]{Buser1992}}]\label{thm:collar_lemma}
    Let $S$ be a hyperbolic surface and let $\gamma_1,\gamma_2,\cdots,\gamma_m$ be disjoint simple closed geodesics in the interior of $S$. Then the followings hold: 
    \begin{enumerate}
        \item the collars $\collar(\gamma_1,w(\gamma_1)),\cdots,\collar(\gamma_m,w(\gamma_m))$ are embedded and pairwise disjoint, where 
        \begin{equation}\label{eqn:collar_width}
            w(\gamma) \df \arcsinh\frac{1}{\sinh\frac{\length_S(\gamma)}{2}}.
        \end{equation}
        \item Each $\collar(\gamma_i, w(\gamma_i))$ is isometric to the cylinder 
        \[ \left\{ (\rho,t):\ -w(\gamma_i)\leq \rho\leq w(\gamma_i),\, t\in\R/\Z \right\} \] 
        with the Riemannian metric 
        \begin{equation}\label{eqn:collar_metric}
            \dif\rho^2+\length_S(\gamma_i)^2\cosh^2\!\rho \dif t^2. 
        \end{equation}
        Here, for each point $(\rho,t)$, its projection to the geodesic $\gamma_i$ is $(0,t)$, and the distance between them is $\abs{\rho}$. In fact, any collar $\collar(\gamma, w)$ will isometric to such a cylinder with the metric \eqref{eqn:collar_metric} as long as it is embedded in $S$. 
    \end{enumerate}
\end{theorem}
In particular, the area of a collar $\collar(\gamma, w(\gamma))$ satisfies
\begin{equation}\label{eqn:collar_area}
    \begin{aligned}
        \area(\collar(\gamma, w(\gamma))) 
    & = \int_0^1\int_{-w(\gamma)}^{w(\gamma)} \length(\gamma)\cosh \rho \dif \rho\dif t\\
    & = 2\length(\gamma)\sinh(w(\gamma))\\
    & = \frac{2\length(\gamma)}{\sinh\frac{\length(\gamma)}{2}} \leq 4. 
    \end{aligned}
\end{equation}
Let $\gamma$ and $\gamma'$ be two different closed geodesics having non-empty intersections in $S$, then it follows from the collar lemma that 
\begin{equation}\label{eqn:cross_collar}
    \length_S(\gamma')\geq 2w(\gamma). 
\end{equation}
If $\gamma$ is a component of $\pa S$, then $\collar(\gamma, w(\gamma))$ is only a half-collar isometric to 
\[ \left\{ (\rho,t):\ 0\leq\rho\leq w(\gamma),\, t\in\R / \Z \right\} \] 
with the same metric \eqref{eqn:collar_metric}. 

If $S$ is non-compact, then it has ends. There are only two kinds of ends: a finite-area cusp end or an infinite-area funnel end. (see Figure \ref{fig:CollarCuspFunnel}). 
Denote by $\cusp(l)$ an embedded cusp bounded by a horocycle of length $l$, then it is biholomorphic to a punctured disc 
\[ \cusp(l) \simeq \left\{ z\in\D: \ 0 < \abs{z} \leq r(l) \right\} \] 
with the metric $ds_O^2$ (the Poincar\'e metric on $\D^*$). The function $r(\cdot)$ is given by 
\begin{equation}\label{eqn:cusp_disc}
    r(l) \df \exp(-2\pi/l). 
\end{equation}
The distance between two horocycles of length $l_1$ and $l_2$ is 
\begin{equation}\label{eqn:dist_horo}
    \abs{\log(l_1/l_2)}. 
\end{equation}
\begin{figure}
    \centering
    \includegraphics[scale=0.38]{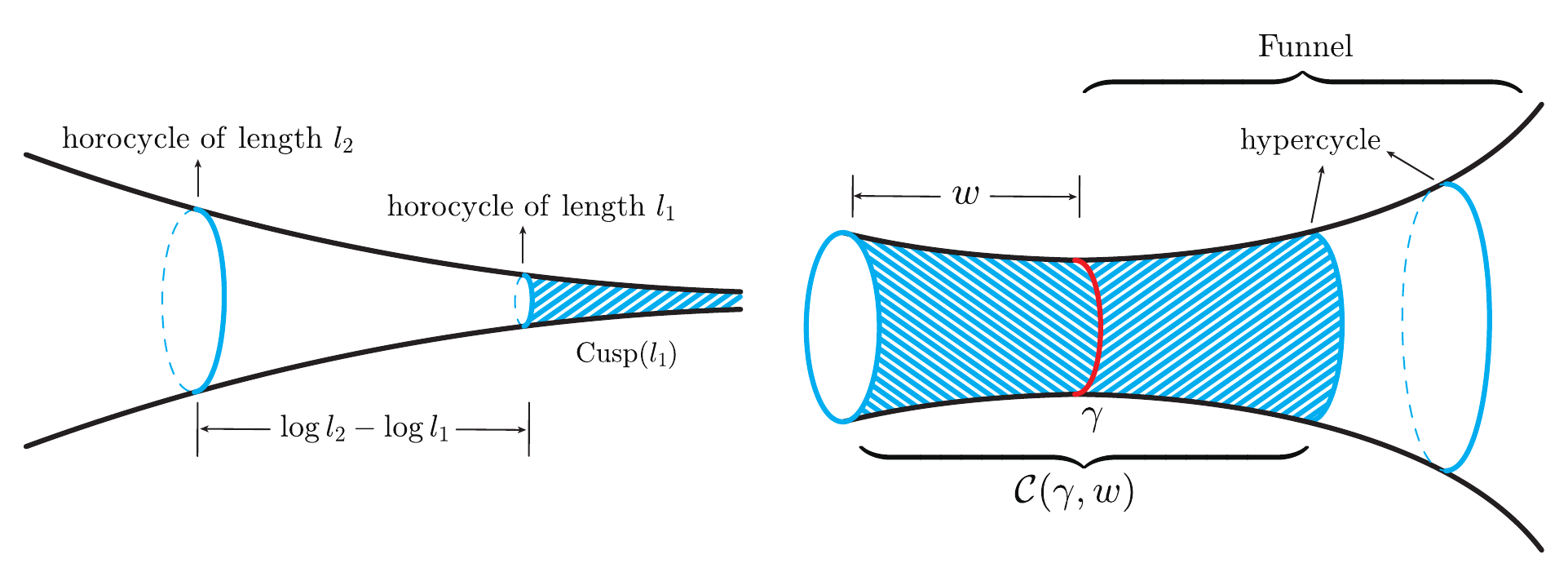}
    \caption{Cusp, collar and funnel}
    \label{fig:CollarCuspFunnel}
\end{figure}

The following lemma will be used later: 
\begin{lemma}\label{lem:horocycle}
    Let $\cusp(l_0)$ be an embedded cusp in $S$ bounded by a horocycle $\beta_0$. For any smooth simple closed curve $\beta_1\subset S\setminus\cusp(l_0)$, if $\beta_1$ is homotopic to $\beta_0$, then $\length_S(\beta_1)\geq\length_S(\beta_0)$. 
\end{lemma}
\begin{proof}
    Consider the covering surface $\widetilde{S}$ with respect to the fundamental group of $\cusp(l_0)$. Since $\pi_1(\cusp(l_0))\simeq \Z$ is generated by the parabolic element $\langle \beta_0 \rangle = \langle \beta_1 \rangle$, $\widetilde{S}$ is isometric to the parabolic cylinder $(\D^*,ds_O^2)$. Denote by $\tilde{\beta}_0$ and $\tilde{\beta}_1$ the isometrical lifts of $\beta_0$ and $\beta_1$, respectively. Suppose that they are parametrized by 
    \[ \tilde{\beta}_0: \theta\mapsto(\rho_0(\theta), \theta) \textrm{\quad and\quad } \tilde{\beta}_1: \theta\mapsto(\rho_1(\theta), \theta) \] 
    respectively in the pair of coordinates $(\rho,\theta)$, then $\rho_0(\theta)$ is a constant function, and $\rho_1(\theta)\geq \rho_0(\theta)$. Hence by Lemma \ref{lem:metrics_rho} one has 
    \begin{align*}
        \length_S(\beta_1) = \length_{\widetilde{S}}(\tilde{\beta}_1)
        & = \int_0^{2\pi} h_O(\rho_1)\sqrt{ \sinh^2\! \rho_1 + \Big(\frac{\dif\rho_1}{\dif\theta}\Big)^2 } \dif\theta\\
        &\geq \int_0^{2\pi} h_O(\rho_1)\sinh\rho_1 \dif\theta\\
        &\geq \int_0^{2\pi} h_O(\rho_0)\sinh\rho_0 \dif\theta = \length_{\widetilde{S}}(\tilde{\beta}_0) = \length_S(\beta_0) 
    \end{align*}
    as desired. 
\end{proof}
The proof of Lemma \ref{lem:horocycle} also leads to the following.
\begin{lemma}\label{lem:hypercycle}
    Let $\gamma$ be a simple closed geodesic in $S$, and let $\beta\subset\pa\collar(\gamma, w)$ be a hypercycle (equidistant curve) with respect to $\gamma$. For each smooth simple closed curve $\beta'\subset S\setminus\collar(\gamma, w)$, if $\beta'$ is homotopic to $\beta$, then we have 
    \[ \length_{S}(\beta') \geq \length_{S}(\beta). \] 
\end{lemma}

Now we can prove the following fact stated earlier in the introduction: 
\begin{lemma}\label{lem:infinite_area}
    Let $X$ be a finite-area non-compact hyperbolic surface with $n$ large cusps, then the surface 
    \[ X^{\mathrm{fu}} \df X\setminus \bigsqcup_{i=1}^n\cusp_i(\epsilon_0) \] 
    with its Poincar\'e metric has only funnel ends. 
\end{lemma}
\begin{proof}
    Suppose $X^{\mathrm{fu}}$ has a cusp end. Let $\{\beta_i\}$ be an arbitrary sequence of horocycles in this cusp end with length going to zero. For all $i$, by the domain monotonicity of Poincar\'e metrics, we have 
    \[ \length_{X^{\mathrm{fu}}}(\beta_i) \geq \length_{X}(\beta_i); \] 
    while by Lemma \ref{lem:horocycle}, we have 
    \[ \length_{X}(\beta_i) \geq \epsilon_0 > 0. \] 
    Let $i\to\infty$, we have $0\geq\epsilon_0>0$, leading to a contradiction. Thus $X^{\mathrm{fu}}$ has no cusp ends. In particular, $X^{\mathrm{fu}}$ has only funnel ends. 
\end{proof}

A funnel end is a half-collar of infinite width $\collar(\infty)$. In particular, the hyperbolic cylinder $(A_{R_1,R_2},ds_{R}^2)$ is the union of its two funnel ends 
\[ \left\{ z\in\D:\ R_1 < \abs{z} \leq \sqrt{R_1R_2} \right\} \textrm{\ and\ } \left\{ z\in\D:\ \sqrt{R_1R_2} < \abs{z} \leq R_2 \right\}. \] 
There is a unique (primitive) simple closed geodesic in $A_{R_1,R_2}$, which is given by 
\begin{equation*}
    \gamma_{R} \df \left\{ \abs{z} = \sqrt{R_1R_2} \right\}, 
\end{equation*}
and the length of $\gamma_{R}$ in the metric $ds_{R}^2$ is 
\begin{equation}\label{eqn:core_length}
    \frac{2\pi^2}{\log R_2 - \log R_1}. 
\end{equation}

\subsection{Eigenvalues of hyperbolic surfaces}
In this subsection we review some basic spectral theories of hyperbolic surfaces; one may refer to \cite{Chavel1984,Bergeron2016,Borthwick2016} for more details. Let $S$ be a complete finite-area hyperbolic surface and let $\Delta$ be its geometric Laplacian operator. If $S$ has no boundaries, then $\Delta$ is a self-adjoint operator on the Sobolev space $H^{1,2}(S)$. A real number $\lambda\geq0$ is called an eigenvalue of $\Delta$ if there exists a function $f\neq0\in H^{1,2}(S)$ such that 
\begin{equation*}
    \Delta f=\lambda\cdot f. 
\end{equation*}
Then $f$ is called an eigenfunction of $\Delta$ with eigenvalue $\lambda$. 

If $S$ is closed, then the spectrum of $\Delta$ consists only of discrete eigenvalues: 
\[ 0=\lambda_0(S)<\lambda_1(S)\leq\lambda_2(S)\leq\cdots\to\infty. \] 
Let $\{\phi_k\}\subset C^\infty(S)$ be a sequence of orthonormal eigenfunctions of $S$ with respect to these eigenvalues, then one has 
\begin{equation*}
    \lambda_k(S) = \inf \left\{\frac{\int_S \abs{\nabla f}^2}{\int_S f^2}: \ \int_S f\phi_i=0,\,\Forall i<k \right\}, 
\end{equation*}
where $f\neq0$ runs over $H^{1,2}(S)$. In fact, this infimum can be realized by $\phi_k$, hence it is enough to let $f$ run only over the space of smooth functions. Note that in particular $\phi_0 = \frac{1}{\sqrt{\area(S)}}$ is a constant function, hence one has 
\begin{equation}\label{eqn:RQ1}
    \begin{aligned}
        \lambda_1(S) = \RQ(S) 
        & \df \inf \left\{\frac{\int_S\abs{\nabla f}^2}{\int_S f^2 - \frac{1}{\area(S)}(\int_S f)^2}:\ f\neq0\in H^{1,2}(S) \right\}. 
    \end{aligned}
\end{equation}

If $S$ is not closed, then its spectrum is no longer discrete. The spectrum of $\Delta$ consists of a continuous spectrum $\left[\frac14,\infty\right)$ with multiplicity the number of cusps of $S$, a trivial simple eigenvalue $0$ and possibly discrete eigenvalues in $(0,\infty)$. Hence the minimal non-zero spectrum $\bar\lambda_1(S)$ of $S$ is always less than or equal to $\frac14$. In this case, one still has 
\begin{lemma}[{see \eg \cite[Theorem XIII.1]{RS1978}}]\label{lem:RQ2}
    If $S$ is non-compact and has finite area, then 
    \[ \bar\lambda_1(S) \leq \RQ(S). \] 
\end{lemma}

It is still \emph{unknown} whether every hyperbolic surface in $\sM_{g,n}$, the moduli space of Riemann surfaces of genus $g$ with $n$ punctures, always has infinitely many eigenvalues. 
Selberg had shown that certain arithmetic hyperbolic surfaces have infinitely many cuspidal eigenvalues by using his celebrated Selberg's trace formula. 
Nevertheless, Phillips-Sarnak \cite{PS1985} conjectured that a ``generic" surface in $\sM_{g,n}$ (other than $(g,n) = (1,1)$) has only finitely many eigenvalues. We refer the readers to \cite[Section 2]{Sarnak2003} for more discussions on this topic.

If $S$ is compact and has non-empty boundaries, then the Neumann eigenvalue problem asks the existence of eigenfunctions $f$ satisfying the \emph{Neumann boundary condition}, that is, the partial derivative of $f$ in the normal direction vanishes on $\pa S$. The corresponding eigenvalues are called Neumann eigenvalues, which can also be listed as follows: 
\[ 0 = \sigma_0(S) < \sigma_1(S) \leq \sigma_2(S) \leq\cdots\to\infty. \] 
Similarly one still has 
\begin{equation}\label{eqn:RQ3}
    \sigma_1(S) = \RQ(S). 
\end{equation}

The following mini-max principle is a useful tool to prove a hyperbolic surface has large spectral gaps. One may see \eg \cite[Theorem 3.34]{Bergeron2016} for more details. 
\begin{theorem}[Mini-max Principle]\label{thm:mini-max}
    Let $S$ be a closed hyperbolic surface. 
    \begin{enumerate}
        \item Let $f_1,f_2,\cdots,f_k$ be $k$ $(k\geq 2)$ smooth functions on $S$ with pairwise disjoint supports, then 
            \begin{equation}\label{eqn:minimax1}
                \lambda_{k-1}(S) \leq \max_{1\leq i \leq k} \frac{\int_S \abs{\nabla f_i}^2}{\int_S f_i^2}. 
            \end{equation}
        \item Let $S = N_1 \cup N_2 \cup\cdots\cup N_k$ $(k\geq 1)$ be a partition of $S$ into relatively compact subsets such that for any $i\neq j$, 
            \[ \area(N_i)>0 \textrm{ \ and \ }\area(N_i\cap N_j)=0, \] 
        then 
            \begin{equation}\label{eqn:minimax2}
                \lambda_k(S) \geq \min_{1\leq i\leq k} \sigma_1(N_i). 
            \end{equation}
    \end{enumerate}
\end{theorem}

Now we introduce two consequences of Theorem \ref{thm:mini-max} that will be applied later. 
\begin{lemma}\label{lem:eigenvalue_monotonicity}
    Let $S$ be bordered compact hyperbolic surface. Suppose we can find a certain pair of components of $\pa S$ and glue them together (with any twist parameter) to obtain another bordered hyperbolic surface $\widetilde{S}$, then we have 
    \[ \sigma_1(\widetilde{S}) \geq \sigma_1(S); \] 
    In particular if $\widetilde{S}$ is closed, then 
    \[ \lambda_1(\widetilde{S}) \geq \sigma_1(S). \] 
\end{lemma}
\begin{proof}
    The second part is the case of $k=1$ of \eqref{eqn:minimax2}. The proof of the first part is exactly the same as the proof of Part (2) of Theorem \ref{thm:mini-max}. We leave it as an exercise to interested readers. 
\end{proof}

\begin{corollary}\label{cor:eigenvalue_bound}
    Let $S=S_g$ be a closed hyperbolic surface. Suppose there exists $k$ $(k\geq 2)$ disjoint simple closed geodesics $\{\gamma_1,\cdots,\gamma_k\}$ in $S$ such that 
    \begin{equation*}
        S\setminus\bigsqcup_{i=1}^k \gamma_i = \bigsqcup_{i=1}^k S_{g_i,2}(\gamma_i,\gamma_{i+1}), 
    \end{equation*}
    where $\gamma_{k+1}\df\gamma_1$, $g_i\geq1$, and $S_{g_i,2}(\gamma_i,\gamma_{i+1})$ is a compact hyperbolic surface of genus $g_i$ with $2$ geodesic boundaries $\gamma_i$ and $\gamma_{i+1}$. If for any $i\in[1,\,k]$, we have $\length_{S}(\gamma_i) \leq C$, then 
    \begin{equation*}
        \lambda_{k-1}(S) = O\left(\frac{1}{\min\{g_i\}}\right), 
    \end{equation*}
    where the implied constant depends only on $C$. 
\end{corollary}
\begin{proof}
    For simplicity of notation, we write $\ell_i\df\length_{S}(\gamma_i)$, $S_i\df S_{g_i,2}$, and 
    \[ \collar_{i1}\df \collar(\gamma_{i}, w_{i})\cap S_i,\quad \collar_{i2}\df \collar(\gamma_{i+1}, w_{i+1})\cap S_i, \] 
    where $w_{i}=w(\gamma_i)$ is the width of the standard collar of $\gamma_{i}$ defined as in \eqref{eqn:collar_width}. Then by the collar lemma, the two half-collars $\collar_{i1}$ and $\collar_{i2}$ are disjoint. For each $i\in[1,\,k]$, we define a function $f_i$ on $S_i$ by 
    \begin{equation*}
        f_i(p) \df 
        \begin{cases}
            \frac{1}{\sqrt{\area(S_{i})}}, & p\notin \collar_{i1} \cup \collar_{i2},\\
            \frac{1}{w_{i}\sqrt{\area(S_{i})}}\dist(p,\gamma_{i}), & p\in \collar_{i1},\\
            \frac{1}{w_{i+1}\sqrt{\area(S_{i})}}\dist(p,\gamma_{i+1}), & p\in \collar_{i2}.
        \end{cases}
    \end{equation*}
    By \eqref{eqn:collar_area}, we have $\area(\collar_{i1})\leq 2$ and $\area(\collar_{i2}) \leq 2$ (since they are only half-collars), thus 
    \begin{equation}\label{eqn:eigenvalue_bound1}
        \begin{aligned}
            \int_{S_i} f_i^2 \dvol 
            \geq \int_{S_i\setminus(\collar_{i1}\cup\collar_{i2})} f_i^2 \dvol 
            &  =  \frac{\area(S_i) - \area(\collar_{i1}) - \area(\collar_{i2})}{\area(S_i)}\\ 
            &\geq \frac{\area(S_i)-4}{\area(S_i)}. 
        \end{aligned}
    \end{equation}
    Since in the pair of coordinates $(\rho,t)$ one has $ \dist(p,\gamma_i)=\abs{\rho}$, it follows that 
    \begin{equation*}
        \abs{\nabla\dist(p,\gamma_i)} = \abs{\frac{\pa}{\pa \rho}\abs{\rho}} = 1. 
    \end{equation*}
    By definition \eqref{eqn:collar_width}, $w_{i} \geq \arcsinh(e^{-C})$, thus 
    \begin{equation}\label{eqn:eigenvalue_bound2}
        \begin{aligned}
            \int_{S_i} \abs{\nabla f_i}^2 \dvol 
             = \int_{\collar_{i1}\cup\collar_{i2}} \abs{\nabla f_i}^2 \dvol
            & = \frac{\area(\collar_{i1})}{w_{i}^2\area(S_i)} + \frac{\area(\collar_{i2})}{w_{i+1}^2\area(S_i)}\\
            & = O\left(\frac{1}{\area(S_i)}\right). 
        \end{aligned}
    \end{equation}
    From \eqref{eqn:eigenvalue_bound1}, \eqref{eqn:eigenvalue_bound2} and the Gauss-Bonnet formula, we deduce that 
    \begin{equation}\label{eqn:eigenvalue_bound3}
        \frac{\int_{S_i} \abs{\nabla f_i}^2 \dvol}{\int_{S_i} f_i^2 \dvol} = O\left(\frac{1}{\area(S_i)-4}\right) = O\left(\frac{1}{g_i}\right), 
    \end{equation}
    thus the conclusion follows from \eqref{eqn:eigenvalue_bound3} and the Mini-max principle \eqref{eqn:minimax1}. 
\end{proof}

\section{Some results from random surfaces in covering model}\label{section3}
In this section we show that one can use recent developments of random surfaces in covering model to find finite-area non-compact hyperbolic surfaces with large $\bar\lambda_1$, large cusps and long systoles. We note here that the proposition below is essentially due to Hide-Magee \cite{HM2023}; Magee \cite{Magee2024} and Klukowski-Markovi\'c \cite{KM2024}; Nica \cite{Nica1994}; and Puder-Zimhoni \cite{PZ2024}. 
\begin{proposition}\label{prop:surfaces}
    For any $\epsilon>0$, $l>0$, $\delta>0$ and each large enough integer $i$, there exists a finite-area non-compact hyperbolic surface $\sX_i$ such that 
    \begin{enumerate}
        \item the Euler characteristic satisfies $\abs{\chi(\sX_i)} = i$.
        \item Each simple closed geodesic in $\sX_i$ has length $\geq \epsilon$.
        \item $\sX_i$ has large cusps of length $l$.
        \item The minimal non-zero spectrum of $\sX_i$ satisfies $\bar{\lambda}_1(\sX_i)>\frac 14-\delta$. 
    \end{enumerate} 
\end{proposition}
 
Random surfaces in covering model have been systematically studied by Magee-Naud \cite{MN2020}, Magee-Puder \cite{MP2023}, and Magee-Naud-Puder \cite{MNP2022}. We will briefly introduce the basic notions of this theory below. Readers may refer to the papers mentioned above for a comprehensive introduction, and see \eg \cite{LM2022,Naud2022, Naud2023, HM2023} for further developments in this area. 

Let $X = \H / \Gamma$ be the ``base" hyperbolic surface, where $\Gamma$ is a torsion-free discrete subgroup of $\mathrm{PSL}(2,\R)$. For the purpose of this paper, we only consider the case of $X$ being finite-area and non-compact. For any homomorphism $\phi\in\Hom(\Gamma, S_n)$, where $S_n$ is the permutation group of $[n] \df \{1,\cdots,n\}$, let $\Gamma$ act on $\H\times [n]$ by 
\[ \gamma:(z,i) \longmapsto (\gamma z, \phi(\gamma)i), \] 
then the quotient space 
\[ X_\phi \df (\H\times [n]) \Big/ \Gamma, \] 
is a degree $n$ covering surface (possibly not connected) of $X$. Since $\Gamma$ is a finitely generated free group, $\Hom(\Gamma, S_n)$ is a finite set. If the homomorphism $\phi$ is chosen uniformly at random from $\Hom(\Gamma, S_n)$, then the surface $X_\phi$ is called a uniformly random degree $n$ covering surface of $X$. 
A property of hyperbolic surfaces is said to hold \emph{asymptotically almost surely (\textit{a.a.s.})}, if the probability that $X_\phi$ satisfies this property tends to one as $n\to\infty$. 

The proof of Proposition \ref{prop:surfaces} is just a combination of several theorems. The first theorem we need is due to Hide-Magee \cite{HM2023}: 
\begin{theorem}[{\cite[Theorem 1.1]{HM2023}}]\label{thm:HM}
    For any $\delta>0$, a.a.s. 
    \[ \spec (X_\phi)\cap \Big[0, \frac 14 - \delta \Big) = \spec (X)\cap \Big[ 0, \frac 14 - \delta \Big) \] 
    and the multiplicities on both sides are the same. 
\end{theorem}

Recall that the notion of large cusps is about the existence of disjoint embedded large cusps. The second theorem we need is due to Klukowski-Markovi\'c \cite{KM2024} and independently Magee \cite{Magee2024}: 
    
\begin{theorem}[{\cite[Theorem 3.5]{KM2024} or \cite[Section 5]{Magee2024}}]\label{thm:KM_Magee}
    For any fixed positive constant $l$, a.a.s. $X_\phi$ has large cusps of length $l$. 
\end{theorem}
\begin{remark*}
    Magee had actually proved a stronger result which said that $l$ can be chosen to be $K\log n$ for some uniform positive constant $K$. But for the current purpose of this paper, we do not need this stronger result. 
\end{remark*}
The last result we need is a direct corollary of two theorems due to Nica \cite{Nica1994} and Puder-Zimhoni \cite{PZ2024}. Let $X=S_{0,3}$ be the hyperbolic three-punctured sphere. Recall that $\pi_1(X)$ is a free group on two generators, and we fix an isomorphism $\Gamma\simeq\pi_1(X)\simeq F_2$. For any $\phi\in\Hom(\Gamma, S_n)$, let $\pi:X_\phi\to X$ be the covering map. For any primitive closed geodesic $\gamma$ in $X$, let $\widetilde{\gamma}\in\pi_1(X)$ be one of its representative, and let $\pi^{-1}(\gamma)$ be the set of primitive closed geodesics in $X_\phi$ that covers $\gamma$ (that is, for any $\eta\in\pi^{-1}(\gamma)$, its image under $\pi$ coincides with $\gamma$ as a subset of $X$). Then we have 
\begin{equation*}
    \abs{\pi^{-1}(\gamma)} = \abs{\left\{ \textrm{cycles in the cycle decomposition of $\phi(\widetilde{\gamma})$} \right\}}.
\end{equation*}
Note that fixed points of $\phi(\widetilde{\gamma})$ are viewed as $1-$cycles. Moreover, $\eta\in\pi^{-1}(\gamma)$ is a degree $k$ cover of $\gamma$ if and only if it corresponds to a cycle of length $k$ in the cycle decomposition of $\phi(\widetilde{\gamma})$. Since any primitive closed geodesic in $X_\phi$ covers a primitive closed geodesic in $X$, we deduce that a shortest closed geodesic of $X_\phi$ has length $>\epsilon$, if and only if for any primitive closed geodesic $\gamma$ in $X$ with $\length_X(\gamma) \leq \epsilon$, the cycle decomposition of $\phi(\widetilde{\gamma})$ has no cycles of length $<\epsilon / \ell_X(\gamma)$. 
    
Let $\sP(X)$ be the set of primitive closed geodesics in $X$. For any $\epsilon>0$, let 
\[ \{\gamma_1,\cdots,\gamma_{k}\}\subset \sP(X) \] 
be the set of all primitive closed geodesics in $X$ with length $<\epsilon$, and be listed such that $\ell_X(\gamma_1)\leq\cdots\leq\ell_X(\gamma_k) < \epsilon$. Denote by $\sys(X_\phi)$ the length of a shortest closed geodesic in $X_\phi$, then 
\begin{equation}\label{eqn:systole1}
    \begin{aligned}
        &\prob_n\left(X_\phi: \ \sys(X_\phi) \geq \epsilon \right)\\
        =  &\prob_n\left(X_\phi: \ \textrm{$\Forall 1\leq i\leq k$, $\phi(\widetilde{\gamma}_i)$ has no cycles of length $< \epsilon / \ell_X(\gamma_i)$} \right)\\
        \geq &\prob_n\left(X_\phi: \ \textrm{$\Forall 1\leq i\leq k$, $\phi(\widetilde{\gamma}_i)$ has no cycles of length $< \epsilon / \ell_X(\gamma_1)$} \right)\\
        \geq &\prob_n\left(X_\phi: \ \textrm{$\Forall 1\leq i\leq k$, $(\phi(\widetilde{\gamma}_i))^{d}$ has no fixed points} \right)\\
    \end{aligned}
\end{equation}
where $d$ is the factorial of $\lfloor\epsilon / \ell_X(\gamma_1)\rfloor$. Since $\pi_1(X)\simeq F_2$ and $(\phi(\widetilde{\gamma}_i))^{d} = \phi((\widetilde{\gamma}_i)^{d})$, the problem is reduced to studying the distribution of the number of fixed points of a random homomorphism in $\Hom(F_2,S_n)$, endowed with the uniform probability measure. Denote by $\mathrm{Fix}_n(\phi)$ the number of fixed points of a permutation $\phi$ in $S_n$. Nica had proved that (see also \cite{LP2010} for a new proof) 
\begin{theorem}[{\cite[Corollary 1.2]{Nica1994}}]\label{thm:Nica}
    Let $F$ be a finitely generated free group, and let $w=v^d\in F$ with $v$ non-power and $d\geq 1$. Then 
    \[ \lim_{n\to\infty} \prob_n\left(\phi\in\Hom(F,S_n): \ \mathrm{Fix}_n(\phi(w))=0 \right) = e^{-\sum_{h|d}h/d} > 0, \] 
    where $\Hom(F,S_n)$ is endowed with the uniform probability measure. 
\end{theorem}
Applying Theorem \ref{thm:Nica} to $w=(\widetilde{\gamma}_i)^d$, it follows that the probability that $(\widetilde{\gamma}_i)^d$ has no fixed points is positive as $n\to\infty$. It is also necessary to consider the limit of the joint distribution of the number of fixed points of finitely many $\phi(\widetilde{\gamma})$'s. Puder-Zimhoni had proved that 
\begin{theorem}[{\cite[Theorem 1.14]{PZ2024}}]\label{thm:PZ}
    Let $F$ be a finitely generated free group, and let $\{w_i=v_i^{d_i}\}$ be a finite set of words with $v_i$ non-power and $d_i\geq 1$. If for any $i\neq j$, $v_i$ is not conjugate to $v_j$ and $v_j^{-1}$, then $\mathrm{Fix}_n(\phi(w_1)),\mathrm{Fix}_n(\phi(w_2)),\cdots$ are asymptotically independent, \ie they have a joint limit distribution of independent random variables. 
\end{theorem}
\noindent Since different primitive closed geodesics belong to different conjugacy classes in $\pi_1(X)\simeq F_2$, it follows from Theorem \ref{thm:PZ} that 
\begin{equation}\label{eqn:systole2}
    \begin{aligned}
        &\lim_{n\to\infty}\prob_n\left(X_\phi: \ \Forall 1\leq i\leq k,\ \mathrm{Fix}_n((\phi(\widetilde{\gamma}_i))^{d}) = 0 \right)\\
        = &\prod_{1\leq i\leq k}\lim_{n\to\infty}\prob_n\left(X_\phi: \ \mathrm{Fix}_n(\phi((\widetilde{\gamma}_i)^{d})) = 0 \right).
    \end{aligned}
\end{equation}
By Theorem \ref{thm:Nica}, the RHS of the equation above depends only on $d$ and $k$, hence depends only on $\epsilon$. Therefore, one may conclude from \eqref{eqn:systole1} and \eqref{eqn:systole2} that 
\begin{corollary}\label{cor:systole}
    For any fixed $\epsilon>0$, $\liminf\limits_{n\to\infty} \prob_n\left(X_\phi: \ \sys(X_\phi) \geq \epsilon\right) > 0$. 
\end{corollary}

\begin{proof}[Proof of Proposition \ref{prop:surfaces}]
    Let $X=S_{0,3}$ be the hyperbolic three-punctured sphere, then a degree $i$ covering surface of $X$ has Euler characteristic exactly $-i$. $X_\phi$ is connected \textit{a.a.s.} due to a theorem of Dixon \cite{Dixon1969}. It is known that $\lambda_1(X)>\frac{1}{4}$. Thus, by Theorem \ref{thm:HM}, Theorem \ref{thm:KM_Magee} and Corollary \ref{cor:systole}, for $i$ sufficiently large, the set of degree $i$ covering surfaces of $X$ satisfying our desired properties has positive probability, hence is non-empty, from which the proposition follows. 
\end{proof}

\section{A comparison theorem}\label{section4}
Let $\epsilon_0>0$ be a fixed constant. Let $X$ be a finite-area non-compact hyperbolic surface with $n$ cusps, and suppose that $X$ has large cusps of length $l\geq\epsilon_0$. Let 
\[ X^{\mathrm{fu}} \df X\setminus\bigsqcup_{i=1}^n\cusp_i(\epsilon_0) \] 
be the subsurface of $X$ equipped with its own complete Poincar\'e metric. Denote the complete Poincar\'e metrics on $X$ and $X^{\mathrm{fu}}$ by $ds_X^2$ and $ds_{X^{\mathrm{fu}}}^2$ respectively, then there is a natural conformal inclusion: 
\[ (X^{\mathrm{fu}},ds_{X^{\mathrm{fu}}}^2) \hookrightarrow (X,ds_X^2). \] 
Inspired by the work of Brooks \cite{Brooks1999}, in this section we will prove that if $l$ is sufficiently large, then these two hyperbolic metrics $ds_X^2$ and $ds_{X^{\mathrm{fu}}}^2$ are close to each other outside those large cusps of $X$. 

\subsection{Construction of the intermediate metric}
In this subsection we prove the following lemma which is inspired by \cite[Lemma 1.1]{Brooks1999}.
\begin{lemma}\label{lem:intermediate}
    Let $R_1\in(0,\, 1)$ and $\delta\in(0,\,1)$. Then for each sufficiently large $\rho_0$, there exists an $R(\rho_0)$ satisfying
    \[ \tanh\frac{\rho_0+1}{2} < R(\rho_0), \] 
    such that for any $R_2\in[R(\rho_0),\,1]$, 
    there exists a smooth Riemannian metric $ds_{\delta}^2$ on $A_{R_1,R_2}$ such that the followings hold: 
    \begin{enumerate}
        \item $ds_{\delta}^2$ is conformal to $ds_{R}^2$ on $A_{R_1,R_2}$, and $ds_{\delta}^2\geq ds_O^2$. 
        \item $ds_{\delta}^2 = ds_{R}^2$ on the annulus $\left\{ z\in\D:\ R_1 < \abs{z} \leq \tanh(\rho_0/2) \right\}$.
        \item $ds_{\delta}^2 = ds_O^2$ on the annulus $\left\{ z\in\D:\ \tanh((\rho_0+1)/2) \leq \abs{z}<R_2 \right\}$.
        \item The Gaussian curvature of $ds_{\delta}^2$ is everywhere between $-1-\delta$ and $-1+\delta$. 
    \end{enumerate}
\end{lemma}
\begin{proof}
    Let $0\leq\chi_0(\rho)\leq1$ be a universal smooth cut-off function such that $\chi_0(\rho) \equiv 0$ when $\rho\leq 0$, and $\chi_0(\rho) \equiv 1$ when $\rho\geq 1$. By Lemma \ref{lem:calculus2}, for each sufficiently large $\rho_0$, there exists an $R(\rho_0)$ satisfying
    \[ \tanh\frac{\rho_0+1}{2} < R(\rho_0). \] 
    For any $R_2\in[R(\rho_0),\, 1]$, let 
    \begin{equation}\label{eqn:function_intermediate}
        h_\delta(\rho) \df 
        \begin{cases}
            h_{R_1,R_2}(\rho), & \rho\leq \rho_0,\\
            (1-\chi_0(\rho-\rho_0))h_{R_1,R_2}(\rho) + \chi_0(\rho-\rho_0)h_O(\rho), & \rho_0\leq \rho\leq \rho_0+1,\\
            h_O(\rho), & \rho_0+1\leq \rho, 
        \end{cases}
    \end{equation}
    where the functions $h_{R_1,R_2}(\rho)$, $h_O(\rho)$ are given in Lemma \ref{lem:metrics_rho}. Thus $h_\delta$ is a function smoothly connecting $h_{R_1,R_2}$ and $h_{O}$ such that 
    \begin{equation*}
        h_{R_1,R_2} \geq h_{\delta} \geq h_{O}. 
    \end{equation*}

    \noindent Define the Riemannian metric $ds_{\delta}^2$ on $A_{R_1,R_2}$ by 
    \begin{equation}\label{eqn:metric_intermediate}
        ds_{\delta}^2 \df h_\delta^2(\rho)\cdot\left( \dif\rho^2 + \sinh^2\!\rho \dif\theta^2 \right), 
    \end{equation} 
    where $\rho = 2\arctanh\abs{z}$, $\theta\in {\R} / {(2\pi\Z)} $. It follows from \eqref{eqn:function_intermediate} and \eqref{eqn:metric_intermediate} that $ds_{\delta}^2$ satisfies our desired properties (1), (2) and (3). Recall that the Gaussian curvature $K_h$ of a metric in the form 
\[ h^2(\rho)\cdot\left( \dif\rho^2 + \sinh^2\!\rho \dif\theta^2 \right) \] 
can be calculated by the formula (\cf \cite[Page 160]{Brooks1999}) 
\begin{equation}\label{eqn:curvature}
     K_h = -\frac{1}{h^2}\cdot \left( \frac{\dif^2(\log h)}{\dif\rho^2} + \frac{\dif(\log h)}{\dif\rho} \coth\rho + 1 \right). 
\end{equation}
\noindent Denote by $K_{\delta}$ the Gaussian curvature of $ds_{\delta}^2$. When $\rho\notin[\rho_0,\rho_0+1]$, $K_{\delta}$ is exactly $-1$; when $\rho\in[\rho_0,\rho_0+1]$, by the curvature formula \eqref{eqn:curvature} we have 
    \begin{align*}
        \abs{ K_{\delta}+1 } 
        &  =  \frac{1}{h_\delta^2}\cdot \abs{ h_\delta^2 - 1 - \frac{\dif^2(\log h_\delta)}{\dif\rho^2} - \frac{\dif(\log h_\delta)}{\dif\rho}\coth\rho }\\
        &\leq \abs{h_\delta^2 - 1} + \abs{\frac{(\dif^2 h_\delta/\!\dif\rho^2)h_\delta - (\dif h_\delta/\!\dif\rho)^2}{h_\delta^2}} + \abs{\frac{\dif h_\delta/\!\dif\rho}{h_\delta}\coth\rho}\\
        &\leq \abs{h_\delta^2 - 1} + \abs{\frac{\dif^2 h_\delta}{\dif\rho^2}} + \abs{\frac{\dif h_\delta}{\dif\rho}}^2\!\! + \abs{\frac{\dif h_\delta}{\dif\rho}\coth \rho}, 
    \end{align*}
    where we have used the fact that $h_\delta\geq h_O > 1$. Thus to prove $\abs{ K_{\delta}+1 }\leq\delta$ it suffices to verify that, for $\rho_0$ sufficiently large we have 
    \begin{equation}\label{eqn:h_epsilon}
        \begin{aligned}
            1 \leq h_\delta &\leq 1+\delta/6,\\
            \abs{\frac{\dif h_\delta}{\dif\rho}\coth \rho} &\leq \delta/6,\\
            \textrm{and\quad} \abs{\frac{\dif ^2 h_\delta}{\dif\rho^2}} &\leq \delta/6, 
        \end{aligned}
    \end{equation}
    when $\rho_0\leq\rho\leq \rho_0+1$. By \eqref{eqn:function_intermediate} we have 
    \begin{align*}
        \abs{\frac{\dif h_\delta}{\dif\rho}} &\leq c_1\cdot \left( \abs{h_{R_1,R_2}-h_O} + \abs{\frac{\dif h_{R_1,R_2}}{\dif\rho}} + \abs{\frac{\dif h_{O}}{\dif\rho}} \right),\\
        \abs{\frac{\dif^2 h_\delta}{\dif\rho^2}} &\leq c_2\cdot \left( \abs{h_{R_1,R_2}-h_O} + \abs{\frac{\dif h_{R_1,R_2}}{\dif\rho}} + \abs{\frac{\dif h_{O}}{\dif\rho}} + \abs{\frac{\dif^2 h_{R_1,R_2}}{\dif\rho^2}} + \abs{\frac{\dif^2 h_{O}}{\dif\rho^2}} \right),
    \end{align*}
    where $c_1,c_2>0$ are universal constants depending only on $\chi_0$. Therefore by Lemma \ref{lem:calculus1} and Lemma \ref{lem:calculus2}, when $\rho\in[\rho_0, \rho_0+1]$, the inequalities \eqref{eqn:h_epsilon} can be satisfied by choosing $\rho_0$ sufficiently large. The proof is complete.
\end{proof}

\subsection[short]{Comparison between hyperbolic metrics}
We introduce the following definition which is similar to the notion of large cusps: 
\begin{definition*}\label{def:longcollar}
    Let $S$ be a bordered compact hyperbolic surface. We say $\pa S$ has \emph{long half-collars of width $w$}, if the collars of width $w$ around each component of $\pa S$ are embedded and pairwise disjoint. 
\end{definition*}

Recall that $X^{\mathrm{fu}}$ has only funnel ends. Denote the central simple closed geodesics of these ends by $\gamma_1,\gamma_2,\cdots,\gamma_n$, and let $Y$ be the compact convex core of $X^{\mathrm{fu}}$. In this subsection, we prove the following theorem which is inspired by \cite[Theorem 2.1]{Brooks1999}.
\begin{theorem}\label{thm:comparison_geometry} 
    For any $0<\delta<1/3$, there exists an $l_{\delta,\epsilon_0}>0$ and a $w_{\delta,\epsilon_0}>0$ such that if $X$ has large cusps of length $l\geq l_{\delta,\epsilon_0}$, then the followings hold: 
    \begin{enumerate}
        \item $\Forall 1\leq i \leq n$, $(1-\delta)\pi\epsilon_0 \leq \length_{X^{\mathrm{fu}}}(\gamma_i) \leq (1+\delta)\pi\epsilon_0$. 
        \item $\pa Y$ has long half-collars of width 
              \[ w(l) \df (1-\delta) \log l - c_{\delta,\epsilon_0}, \] 
              where $c_{\delta,\epsilon_0}>0$ is a constant depending only on $\delta$ and $\epsilon_0$. 
        \item We have 
        \[ ds_X^2 < ds_{X^{\mathrm{fu}}}^2 \leq (1+\delta)ds_{X}^2 \] 
        on the subsurface 
        \[ Y\setminus\bigsqcup_{i=1}^n\collar(\gamma_i, w_{\delta,\epsilon_0}). \] 
    \end{enumerate}
\end{theorem}

As a direct corollary of Part (3) of Theorem \ref{thm:comparison_geometry}, we have the following. 
\begin{corollary}\label{cor:comparison_volume}
    The volume forms $\dvol_X$ and $\dvol_Y$ of $X$ and $Y$ satisfy that 
    \[ \dvol_{X} < \dvol_{Y} \leq (1+\delta)\dvol_{X} \] 
    on the subsurface 
    \[ Y\setminus\bigsqcup_{i=1}^n\collar(\gamma_i, w_{\delta,\epsilon_0}). \] 
\end{corollary}

The proof of Theorem \ref{thm:comparison_geometry} will be divided into several lemmas. To simplify notations, for each $i\in[1,n]$ and any $0<l_1<l_2$, let 
\[ \cusp_i(l_1,l_2) \df \cusp_i(l_2)\setminus\cusp_i(l_1). \] 

We first determine $l_{\delta,\epsilon_0}$. 
Suppose $X$ has large cusps of length $l$. By definition, each end of $X^{\mathrm{fu}}$ has the form 
$\cusp_i(\epsilon_0,\,l)$, 
hence is biholomorphic to the annulus 
\begin{equation*}
    \left\{ z\in\D:\ r(\epsilon_0) < \abs{z} \leq r(l) \right\}, 
\end{equation*} 
where the function $r(\cdot)$ is given by \eqref{eqn:cusp_disc}. Let 
\begin{equation}\label{eqn:R_1}
    R_1 \df r(\epsilon_0) = \exp(-2\pi/\epsilon_0). 
\end{equation}
Choose $\rho_0$ to be sufficiently large such that 
\begin{equation}\label{eqn:rho}
    r(2\epsilon_0  ) \leq \tanh\frac{\rho_0}{2},  
\end{equation}
and moreover, for any $0<\delta<1/3$, we can apply Lemma \ref{lem:intermediate} to $\rho_0$ to find an $R(\rho_0)$. Hence $\rho_0$ depends only on $\delta$ and $\epsilon_0$. Finally let 
\begin{equation}\label{eqn:R_2}
    r(l_{\delta,\epsilon_0}) = R_2 \df \max \left\{ r(4\epsilon_0/\delta), \ R(\rho_0) \right\}. 
\end{equation}
The reason we choose $\rho_0$ and $R_2$ as above will be seen in the proof of Lemma \ref{lem:comparison_geometry_lem3}. 

Recall that $ds_{R}^2$ denotes the Poincar\'e metric on the annulus $A_{R_1,R_2}$. 
\begin{lemma}\label{lem:comparison_geometry_lem1}
    If $X$ has large cusps of length $l\geq l_{\delta,\epsilon_0}$, then there exists a smooth Riemannian metric $ds_{X^\delta}^2$ on $X^{\mathrm{fu}}$ such that the followings hold (see Figure \ref{fig:IntermediateMetric}): 
    \begin{enumerate}
        \item $ds_{X^\delta}^2$ is conformal to $ds_{X}^2$, hence is conformal to $ds_{X^{\mathrm{fu}}}^2$. 
        \item $ds_{X^\delta}^2 = ds_{R}^2$ on each component of the disjoint union of annuli 
        \[ \bigsqcup_{i=1}^n \cusp_i\!\left( \epsilon_0, 2\epsilon_0   \right). \] 
        \item $ds_{X^\delta}^2 \geq ds_{X}^2$ on $X^{\mathrm{fu}}$, and $ds_{X^\delta}^2 = ds_X^2$ on the subsurface 
        \[ X\setminus\bigsqcup_{i=1}^n\cusp_i(l_{\delta,\epsilon_0}).\] 
        \item The Gaussian curvature of $ds_{X^\delta}^2$ is everywhere between $-1-\delta$ and $-1+\delta$. 
    \end{enumerate}
\end{lemma}
\begin{figure}
    \centering
    \includegraphics[scale = 0.36]{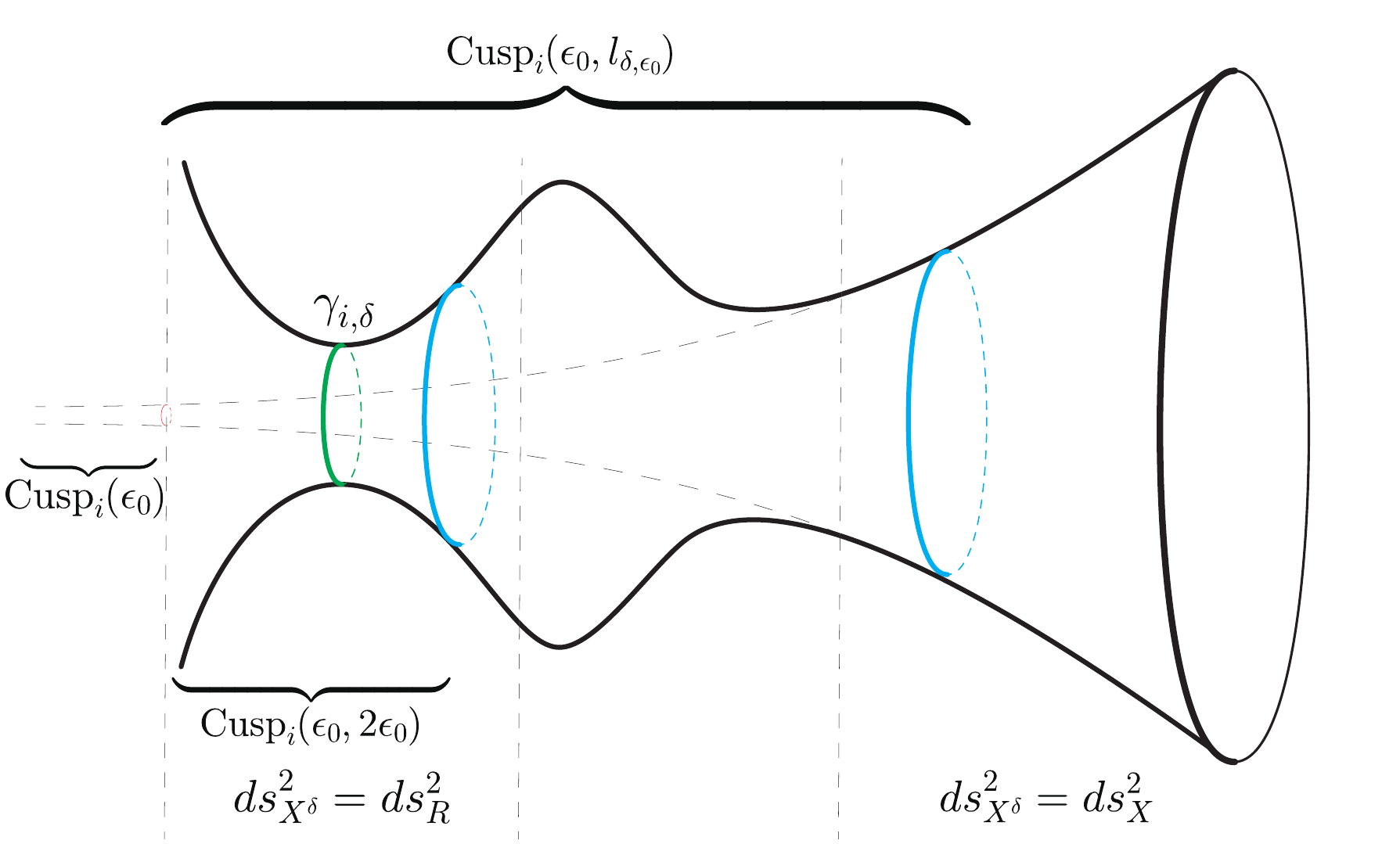}
    \caption{An illustration for the metric  $ds_{X^\delta}^2$ on $\cusp_i(l)$.}
    \label{fig:IntermediateMetric}
\end{figure}

\begin{proof}
    By Lemma \ref{lem:intermediate}, for the $R_1$, $\rho_0$ and $R_2$ defined by \eqref{eqn:R_1}, \eqref{eqn:rho} and \eqref{eqn:R_2}, there exists a smooth Riemannian metric $ds_{\delta}^2$ on the annulus $A_{R_1,R_2}$. 
    By \eqref{eqn:rho} we have 
    \[ r\!\left( 2\epsilon_0   \right) \leq \tanh\frac {\rho_0}{2}, \] 
    hence by Lemma \ref{lem:intermediate}, 
    \[ ds_{\delta}^2 = ds_{R}^2 \textrm{\ \ when $\abs{z}\leq r\!\left( 2\epsilon_0   \right)$}, \] which proves Part $(2)$. 
    By \eqref{eqn:R_2} we have 
    \[ \tanh\frac{\rho_0+1}{2} < R(\rho_0) \leq r(l_{\delta,\epsilon_0}), \] 
    hence by Lemma \ref{lem:intermediate}, 
    \[ ds_{\delta}^2 = ds_O^2 \textrm{\ \ when $\abs{z}\geq r(l_{\delta,\epsilon_0})$}. \] 
    Now we replace the metric $\eval{ds_{X}^2}_{X^{\mathrm{fu}}}$ on each end of $X^{\mathrm{fu}}$ by the metric $ds_{\delta}^2$. 
    Since $l\geq l_{\delta,\epsilon_0}$ and
    \[ ds_{X}^2\big|_{\cusp_i(l)} = ds_{O}^2, \] 
     the resulting metric, denoted by $ds_{X^\delta}^2$, is a smooth Riemannian metric on $X^{\mathrm{fu}}$. The conclusion then follows from Lemma \ref{lem:intermediate}. The proof is complete. 
\end{proof}

\begin{lemma}\label{lem:comparison_geometry_lem2}
    Let $ds_{X^\delta}^2$ be the metric on $X^{\mathrm{fu}}$ given by Lemma \ref{lem:comparison_geometry_lem1}, then 
    \begin{equation*}
        (1-\delta)ds_{X^\delta}^2 \leq ds_{X^{\mathrm{fu}}}^2 \leq (1+\delta)ds_{X^\delta}^2. 
    \end{equation*}
\end{lemma}
\begin{proof}
    Suppose in local conformal coordinate $z=x+iy$ we have 
    \[ ds_{X^{\mathrm{fu}}}^2 = h^2_{X^{\mathrm{fu}}}\cdot \abs{\dif z}^2, \quad ds_{X^\delta}^2 = h^2_{X^\delta}\cdot \abs{\dif z}^2, \] 
    then by Lemma \ref{lem:comparison_geometry_lem1} we know that their Gaussian curvatures satisfy  
    \begin{equation}\label{eqn:comparison_geometry_lem2-1}
        \begin{aligned}
            K_{X^{\mathrm{fu}}} &= - h_{X^{\mathrm{fu}}}^{-2}\,\Delta\log h_{X^{\mathrm{fu}}} = -1,\\ 
            -1-\delta \leq K_{X^\delta} &= - h_{X^\delta}^{-2}\,\Delta\log h_{X^\delta} \leq -1+\delta. 
        \end{aligned}
    \end{equation}
    Here $\Delta=\frac{\pa^2}{\pa x^2}+\frac{\pa^2}{\pa y^2}$ is the Euclidean Laplacian. By the Ahlfors-Schwarz lemma (see \eg \cite[Theorem 2.6.1]{Jost2006}), from \eqref{eqn:comparison_geometry_lem2-1} and Part (1) of Lemma \ref{lem:comparison_geometry_lem1} we have 
    \begin{equation}\label{eqn:AS_lemma1}
        (1-\delta)h_{X^\delta}^2 \leq h_{X^{\mathrm{fu}}}^2. 
    \end{equation}
    Moreover, since for each $1\leq i\leq n$ we have 
    \[ \cusp_i(\epsilon_0,\, l_{\delta,\epsilon_0}) \subsetneq X^{\mathrm{fu}} \subsetneq X, \textrm{\quad and\quad} \cusp_i(\epsilon_0,\, l_{\delta,\epsilon_0})\simeq A_{R_1,R_2}, \] 
    by domain monotonicity of Poincar\'e metrics and Part (2) and (3) of Lemma \ref{lem:comparison_geometry_lem1}, we have 
    \begin{equation*}
        f(p) \df \big( h_{X^\delta}^2 \big/ h_{X^{\mathrm{fu}}}^2 \big)(p) = 
        \begin{cases}
            \big( ds_{R}^2\, / ds_{X^{\mathrm{fu}}}^2 \big)(p) > 1, & p\in\bigsqcup_{i=1}^n \cusp_i\!\left( \epsilon_0, 2\epsilon_0   \right);\\
            \big( ds_{X}^2 / ds_{X^{\mathrm{fu}}}^2 \big)(p) < 1, & p\in X\setminus\bigsqcup_{i=1}^n\cusp_i(l_{\delta,\epsilon_0}). 
        \end{cases}
    \end{equation*} 
    It follows that $f$, and hence $\log f$, achieves an minimum at some $p_0\in X^{\mathrm{fu}}$. Thus 
    \begin{equation}\label{eqn:AS_lemma2}
        \begin{aligned}
            \left( (1+\delta)\frac{h_{X^\delta}^2(p)}{h_{X^{\mathrm{fu}}}^2(p)} - 1 \right)\cdot h_{X^{\mathrm{fu}}}^2(p_0)
            &\geq (1+\delta)h_{X^\delta}^2(p_0) - h_{X^{\mathrm{fu}}}^2(p_0)\\
            \textrm{(by \eqref{eqn:comparison_geometry_lem2-1})\qquad} 
            &\geq (- K_{X^\delta}\cdot h_{X^\delta}^2 + K_{X^{\mathrm{fu}}}\cdot h_{X^{\mathrm{fu}}}^2 )(p_0)\\
            \textrm{(by \eqref{eqn:comparison_geometry_lem2-1})\qquad} 
            & = (\Delta\log h_{X^\delta} - \Delta\log h_{X^{\mathrm{fu}}})(p_0)\\
            \textrm{(by definition of $f$)\qquad} 
            & = (\Delta\log f)(p_0)/2\\
            \textrm{($\log f$ has a minimum at $p_0$)\qquad} 
            &\geq 0. 
        \end{aligned}
    \end{equation}
    Then the conclusion follows from \eqref{eqn:AS_lemma1} and \eqref{eqn:AS_lemma2}. 
\end{proof}

For each $1\leq i\leq n$, let $\gamma_{i,\delta}$ be the circle (see Figure \ref{fig:IntermediateMetric}) 
\[ \gamma_{i,\delta} \df \left\{ \abs{z} = \sqrt{R_1R_2} \right\} \] 
contained in the $i$-th end $\cusp_i( \epsilon_0,\,l )$ of $X^{\mathrm{fu}}$. 
We have the following lemma: 
\begin{lemma}\label{lem:comparison_geometry_lem3}
    For each $1\leq i\leq n$, the circle $\gamma_{i,\delta}$ is a simple closed geodesic in $ds_{X^\delta}^2$ (and is unique since $ds_{X^\delta}^2$ is complete and has negative curvature), and we have 
    \begin{equation*}
        \pi\epsilon_0 \leq \length_{X^\delta}(\gamma_{i,\delta}) \leq (1+\delta/3)\pi\epsilon_0, 
    \end{equation*}
    where $\length_{X^\delta}$ denotes the length function with respect to $ds_{X^\delta}^2$. 
\end{lemma}
\begin{proof}
    By definition, $R_1 = r(\epsilon_0) = \exp(-2\pi/\epsilon_0)$, so we have (note that $R_2 <1$) 
    \[ \sqrt{R_1R_2} < \sqrt{R_1} = r(2\epsilon_0), \] 
    which means that $\gamma_{i,\delta}$ is contained in the annulus $\cusp_i(\epsilon_0, 2\epsilon_0 )$. Thus, by Part (2) of Lemma \ref{lem:comparison_geometry_lem1}, $\gamma_{i,\delta}$ is a simple closed geodesic in $ds_{X^\delta}^2$. By \eqref{eqn:core_length} we have 
    \[ \length_{X^\delta}(\gamma_{i,\delta}) = \frac{2\pi^2}{\log R_2 - \log R_1}, \] 
    it follows that 
    \begin{equation}\label{eqn:comparison_geometry_lem3-1}
        \length_{X^\delta}(\gamma_{i,\delta}) \geq \frac{2\pi^2}{-\log r(\epsilon_0)} = \frac{2\pi^2}{2\pi/\epsilon_0} = \pi\epsilon_0. 
    \end{equation}
    Moreover by \eqref{eqn:R_2}, we have $R_2\geq r(4\epsilon_0/\delta)$, hence 
    \begin{equation}\label{eqn:comparison_geometry_lem3-2}
        \length_{X^\delta}(\gamma_{i,\delta}) \leq \frac{2\pi^2}{\log r(4\epsilon_0/\delta) - \log r(\epsilon_0)} = \frac{2\pi^2}{(1 - \delta/4)\cdot2\pi/\epsilon_0} \leq (1+\delta/3)\pi\epsilon_0. 
    \end{equation}
    The proof is complete. 
\end{proof}

Recall that $\gamma_1,\gamma_2,\cdots,\gamma_n$ denote the central simple closed geodesics of the funnel ends of $X^{\mathrm{fu}}$. 
For any $l>0$, denote by $\horo_i(l)$ the embedded horocycle of length $l$ in the $i$-th cusp. 
We have the following lemma: 
\begin{figure}
    \centering
    \includegraphics[scale=0.32]{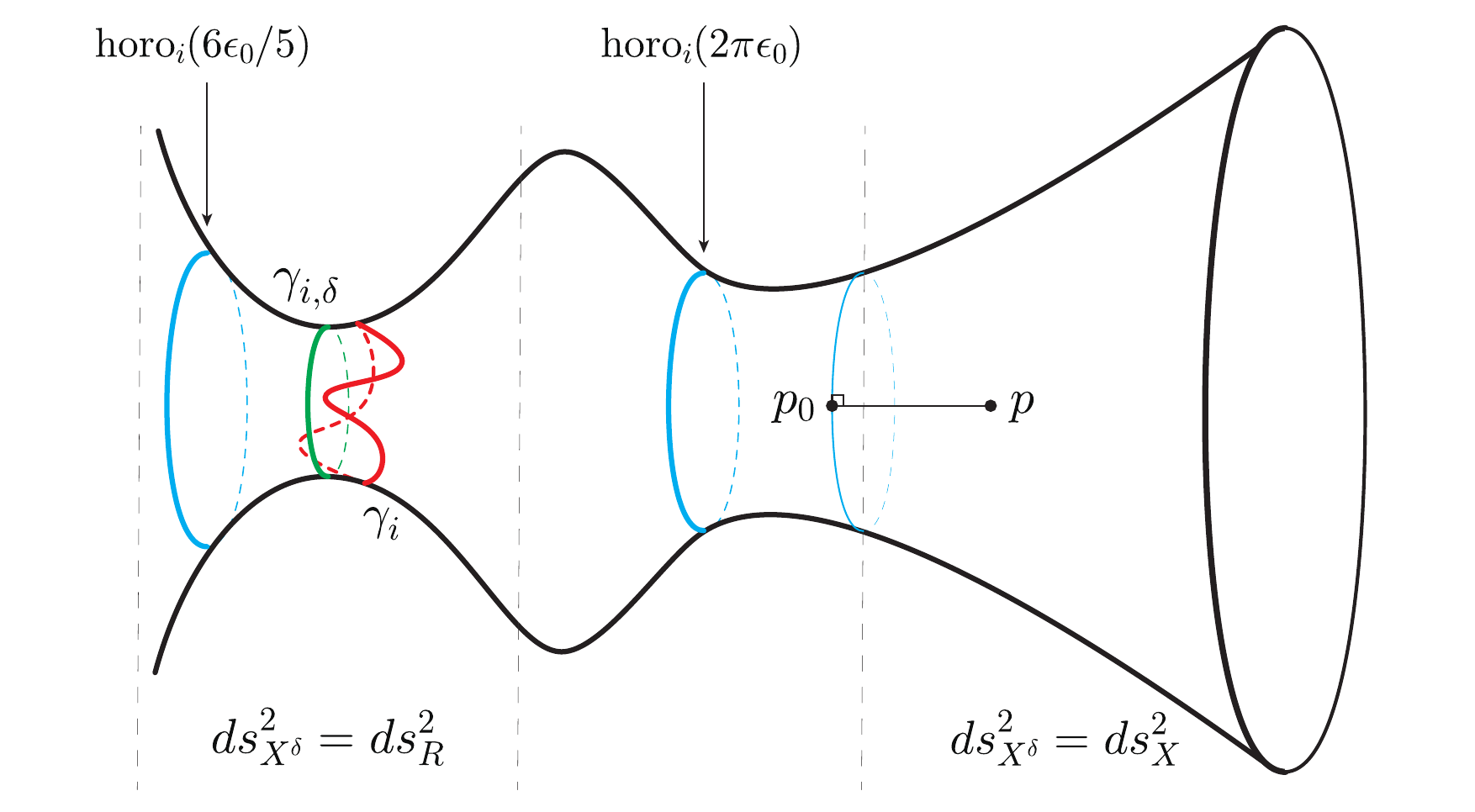}
    \caption{The location of $\gamma_i$ in the proof of Lemma \ref{lem:comparison_geometry_lem4}.}
    \label{fig:position}
\end{figure}

\begin{lemma}\label{lem:comparison_geometry_lem4}
    For each $1\leq i\leq n$, we have 
    \[ \gamma_i\cap\cusp_i(2\pi\epsilon_0)\neq\emptyset, \] 
    and 
    \[ \gamma_i\cap\big(X\setminus\cusp_i(6\epsilon_0 / 5)\big) \neq \emptyset. \] 
\end{lemma}
\begin{proof}
    Denote by $\length_{X}$ and $\length_{X^{\mathrm{fu}}}$ the length functions with respect to $ds_{X}^2$ and $ds_{X^{\mathrm{fu}}}^2$, respectively. Since $\gamma_{i}$ is homotopic to $\gamma_{i,\delta}$, we have 
    \begin{equation}\label{eqn:comparison_geometry_lem4-1}
        \begin{aligned}
            \length_{X^{\mathrm{fu}}}(\gamma_i) 
            &= \inf\left\{ \length_{X^{\mathrm{fu}}}(\gamma'):\ \textrm{$\gamma'\simeq\gamma_{i}$ and $\gamma'$ is smooth} \right\}\\
            \textrm{(by Lemma \ref{lem:comparison_geometry_lem2})\qquad} &\leq (1+\delta)^{\frac12}\cdot \inf\left\{ \length_{X^\delta}(\gamma'):\ \textrm{$\gamma'\simeq\gamma_{i,\delta}$ and $\gamma'$ is smooth} \right\}\\
            \textrm{(by Lemma \ref{lem:comparison_geometry_lem3})\qquad} &= (1+\delta)^{\frac12}\cdot \length_{X^\delta}(\gamma_{i,\delta})\\
            \textrm{(by \eqref{eqn:comparison_geometry_lem3-2})\qquad} &\leq (1+\delta/2)(1+\delta/3)\pi\epsilon_0 \leq (1+\delta)\pi\epsilon_0. 
        \end{aligned}
    \end{equation}
    Therefore by Lemma \ref{lem:comparison_geometry_lem2}, for $\delta<1/3$ we have 
    \begin{equation}\label{eqn:comparison_geometry_lem4-2}
        \length_{X^\delta}(\gamma_i) \leq (1-\delta)^{-\frac12}\cdot\length_{X^{\mathrm{fu}}}(\gamma_i) \leq (1+\delta)\cdot\length_{X^{\mathrm{fu}}}(\gamma_i) < 2\pi\epsilon_0. 
    \end{equation}
    
    Suppose for contradiction that $\gamma_i\cap\cusp_i(2\pi\epsilon_0)=\emptyset$. Recall that Part (3) of Lemma \ref{lem:comparison_geometry_lem1} tells that $ds_{X^\delta}^2 \geq ds_{X}^2$, thus by Lemma \ref{lem:horocycle} we have 
    \[ \length_{X^\delta}(\gamma_i) \geq \length_{X}(\gamma_i) \geq \length_{X}(\horo_i(2\pi\epsilon_0)) = 2\pi\epsilon_0, \] 
    leading to a contradiction to \eqref{eqn:comparison_geometry_lem4-2}, hence $\gamma_i\cap\cusp_i(2\pi\epsilon_0)$ is not empty. 
    
    Suppose for contradiction that $\gamma_i$ is contained in $\cusp_i(\epsilon_0,\, 6\epsilon_0 / 5)$. Recall that $R_1=r(\epsilon_0)$ and $R_2\geq r(4\epsilon_0/\delta)>r(12\epsilon_0)$, by definition of $r(\cdot)$ we have 
    \[ R_1 < r(6\epsilon_0/5) < r\Big(\frac{2\epsilon_0}{1+1/12}\Big) < \sqrt{R_1R_2} < r(2\epsilon_0) < R_2. \] 
    Since $ds_{X^\delta}^2 = ds_{R}^2$ on the annulus 
    \[ \cusp_i\! \left( \epsilon_0, 2\epsilon_0  \right), \] 
    combining with Lemma \ref{lem:hypercycle} we have (see Figure \ref{fig:position} for an illustration) 
    \begin{equation}\label{eqn:comparison_geometry_lem4-3}
        \length_{X^\delta}(\gamma_i) = \length_{R}(\gamma_i) \geq \length_{R}(\horo_i(6\epsilon_0/5)), 
    \end{equation}
    where $\length_{R}$ denotes the length function with respect to $ds_{R}^2$. 
    The RHS of \eqref{eqn:comparison_geometry_lem4-3} is equal to 
    \begin{equation}\label{eqn:comparison_geometry_lem4-4}
        \int_{0}^{2\pi} r(6\epsilon_0/5)\cdot h_{R_1,R_2}\!\left( r(6\epsilon_0/5) \right) \dif\theta = 2\pi\cdot r(6\epsilon_0/5)\cdot h_{R_1,R_2}\!\left( r(6\epsilon_0/5) \right). 
    \end{equation}
    Recall that $h_{R_1,R_2}$ is the density function of the Poincar\'e metric on the annulus $A_{R_1,R_2}$, we thus have 
    \[ h_{R_1,R_2} > h_{R_1,1}, \] 
    where $h_{R_1,1}$ is the metric density function of the Poincar\'e metric on the annulus $A_{R_1,1}$. 
    It follows that 
    \begin{equation}\label{eqn:comparison_geometry_lem4-5}
        \begin{aligned}
            2\pi\cdot r(6\epsilon_0/5)\cdot h_{R_1,R_2}\!\left( r(6\epsilon_0/5) \right) 
            &\geq 2\pi\cdot r(6\epsilon_0/5)\cdot h_{R_1,1}\!\left( r(6\epsilon_0/5) \right)\\
            \textrm{(by Lemma \ref{lem:metrics_r})\qquad} 
            & = 2\pi\cdot \frac{\pi}{- \log r(\epsilon_0)}\cdot \frac{1}{\sin\frac{\pi(- \log r(6\epsilon_0/5))}{- \log r(\epsilon_0)}}\\
            \textrm{(by definition of $r(\cdot)$)\qquad} 
            & = 2\pi\cdot \frac{\pi}{2\pi/\epsilon_0}\cdot \frac{1}{\sin\frac{\pi\cdot 5\pi/3\epsilon_0}{2\pi/\epsilon_0}}\\
            & = 2\pi\cdot \frac{\epsilon_0}{2}\cdot \frac{1}{\sin\frac{5\pi}{6}} = 2\pi\epsilon_0. 
        \end{aligned}
    \end{equation}
    By \eqref{eqn:comparison_geometry_lem4-3}, \eqref{eqn:comparison_geometry_lem4-4} and \eqref{eqn:comparison_geometry_lem4-5}, we have 
    \[ \length_{X^\delta}(\gamma_i) \geq 2\pi\epsilon_0, \]     
    also leading to a contradiction to \eqref{eqn:comparison_geometry_lem4-2}, hence $\gamma_i$ is not contained in $\cusp_i(\epsilon_0,\, 6\epsilon_0/5)$. The proof is complete. 
\end{proof}

Denote by $\dist_{X}$, $\dist_{X^{\mathrm{fu}}}$, and $\dist_{X^\delta}$ the distance functions with respect to $ds_{X}^2$, $ds_{X^{\mathrm{fu}}}^2$, and $ds_{X^\delta}^2$, respectively. We have the following lemma: 
\begin{lemma}\label{lem:comparison_geometry_lem5}
    Suppose $l\geq 2\pi\epsilon_0$. For any point $p\in\horo_i(l)$, we have 
    \begin{equation}\label{eqn:comparison_geometry_lem5-1}
        \dist_{X^{\mathrm{fu}}}(p,\gamma_i) \geq (1-\delta) \log l - c_{\delta,\epsilon_0}, 
    \end{equation}
    and 
    \begin{equation}\label{eqn:comparison_geometry_lem5-2}
        \dist_{X^{\mathrm{fu}}}(p,\gamma_i) \leq (1+\delta) \log l + C_{\delta,\epsilon_0}, 
    \end{equation}
    where $c_{\delta,\epsilon_0}, C_{\delta,\epsilon_0}>0$ are constants depending only on $\delta$ and $\epsilon_0$. 
\end{lemma}
\begin{proof}
    See Figure \ref{fig:position} for an illustration. By Lemma \ref{lem:comparison_geometry_lem2} we have 
    \begin{equation}\label{eqn:comparison_geometry_lem5-3}
        \dist_{X^{\mathrm{fu}}}( p, \gamma_i ) \geq (1-\delta)^{\frac12}\cdot \dist_{X^\delta}( p, \gamma_i ), 
    \end{equation}
    thus by Lemma \ref{lem:comparison_geometry_lem4} and \eqref{eqn:comparison_geometry_lem4-2} we have 
    \begin{equation}\label{eqn:comparison_geometry_lem5-4}
        \begin{aligned}
            \dist_{X^\delta}( p, \gamma_i )
            &\geq \dist_{X^\delta}\!\big( p,\, \horo_i(2\pi\epsilon_0) \big) - 2\pi\epsilon_0\\
            \textrm{(since $ds_{X^\delta}^2\geq ds_{X}^2$)\qquad} 
            &\geq \dist_{X}\!\big( p,\, \horo_i(2\pi\epsilon_0) \big) - 2\pi\epsilon_0\\
            \textrm{(by \eqref{eqn:dist_horo})\qquad} 
            &= \log l - \log(2\pi\epsilon_0) - 2\pi\epsilon_0. 
        \end{aligned}
    \end{equation}
    From \eqref{eqn:comparison_geometry_lem5-3} and \eqref{eqn:comparison_geometry_lem5-4} it follows that 
    \begin{equation*}
        \begin{aligned}
            \dist_{X^{\mathrm{fu}}}( p, \gamma_i ) 
            &\geq (1-\delta)\cdot \left( \log l - \log(2\pi\epsilon_0) - 2\pi\epsilon_0 \right)\\
            & = (1-\delta) \log l - c_{\delta,\epsilon_0}, 
        \end{aligned}
    \end{equation*}
    which proves \eqref{eqn:comparison_geometry_lem5-1} by suitable choosing of $c_{\delta,\epsilon_0}$. 

    Similarly, by Lemma \ref{lem:comparison_geometry_lem2}, Lemma \ref{lem:comparison_geometry_lem4} and \eqref{eqn:comparison_geometry_lem4-2} we have 
    \begin{equation}\label{eqn:comparison_geometry_lem5-5}
        \dist_{X^{\mathrm{fu}}}( p,\, \gamma_i ) 
        \leq (1+\delta)\cdot \dist_{X^\delta}\!\big(p,\, \horo_i(6\epsilon_0/5) \big) + 4\pi\epsilon_0. 
    \end{equation}
    Let $p_0$ be the projection of $p$ to the horocycle 
    \[ \Big\{ \abs{z} = \tanh\frac{\rho_0+1}{2} \Big\}, \] 
    where $\rho_0$ is given by \eqref{eqn:rho}. Let $\dist_{R}$ be the distance function with respect to $ds_{R}^2$. It follows from the construction of $ds_{X^\delta}^2$ that 
    \begin{equation}\label{eqn:comparison_geometry_lem5-6}
        \begin{aligned}
            \dist_{X^\delta}\!\big(p,\, \horo_i(6\epsilon_0/5) \big) 
            & = \dist_{X^\delta} (p,p_0) + \dist_{X^\delta}\!\big( p_0,\, \horo_i(6\epsilon_0/5) \big)\\
            &\leq \dist_{X}(p,p_0) + \dist_{R}\!\big( p_0,\horo_i(6\epsilon_0/5) \big)\\
            &\leq \log l - \log(2\epsilon_0) + \dist_{R}\!\big( p_0,\horo_i(6\epsilon_0/5) \big). 
        \end{aligned}
    \end{equation}
    Since $R_1$, $R_2$ and $\rho_0$ depend only on $\delta$ and $\epsilon_0$, we find that $\dist_{R}( p_0,\horo_i(6\epsilon_0/5) )$ also depends only on $\delta$ and $\epsilon_0$. Thus \eqref{eqn:comparison_geometry_lem5-2} follows from \eqref{eqn:comparison_geometry_lem5-5} and \eqref{eqn:comparison_geometry_lem5-6} by suitable choosing of $C_{\delta,\epsilon_0}$. The proof is complete. 
\end{proof}

Now we can go back to the proof of Theorem \ref{thm:comparison_geometry}. 
\begin{proof}[Proof of Theorem \ref{thm:comparison_geometry}]
    Suppose $X$ has large cusps of length $l\geq l_{\delta,\epsilon_0}$. 

    \underline{Proof of statement (1)}. 
    Recall \eqref{eqn:comparison_geometry_lem4-1} tells that, 
    \begin{equation}\label{eqn:comparison_geometry_pf1}
        \length_{X^{\mathrm{fu}}}(\gamma_i) \leq (1+\delta)\pi\epsilon_0, 
    \end{equation}
    similar to its proof, by \eqref{eqn:comparison_geometry_lem3-1} and Lemma \ref{lem:comparison_geometry_lem2} and \ref{lem:comparison_geometry_lem3} we have 
    \begin{equation}\label{eqn:comparison_geometry_pf2}
        \length_{X^{\mathrm{fu}}}(\gamma_i) \geq (1-\delta)^{\frac12}\cdot \length_{X^\delta}(\gamma_{i,\delta}) \geq (1-\delta)\pi\epsilon_0, 
    \end{equation}
    thus the first statement of the theorem follows from \eqref{eqn:comparison_geometry_pf1} and \eqref{eqn:comparison_geometry_pf2}. 
    
    \underline{Proof of statement (2)}. 
    For each $1\leq i\leq n$, consider the embedded horocycle $\horo_i(l)$ of length $l$. By Lemma \ref{lem:comparison_geometry_lem5} we have 
    \begin{equation*}
        \dist_{X^{\mathrm{fu}}}\!\big( \horo_i(l),\, \gamma_i \big) \geq (1-\delta)\log l - c_{\delta,\epsilon_0} \equalscolon w(l), 
    \end{equation*}
    which means that $\gamma_i$ has an embedded half-collar of width $w(l)$ in $X^{\mathrm{fu}}$. Since $X$ has large cusps of length $l$, these horocycles $\{\horo_i(l)\}$ are pairwise disjoint, so these half-collars are also disjoint. The second statement of the theorem follows. 
    
    \underline{Proof of statement (3)}. 
    By Lemma \ref{lem:comparison_geometry_lem5}, for any $p\in \horo_i(l_{\delta,\epsilon_0})$ we have 
    \begin{equation}\label{eqn:comparison_geometry_pf3}
        \dist_{X^{\mathrm{fu}}}(p, \gamma_i ) \leq (1+\delta) \log l_{\delta,\epsilon_0} + C_{\delta,\epsilon_0}. 
    \end{equation}
    Let $w_{\delta,\epsilon_0}$ be the RHS of \eqref{eqn:comparison_geometry_pf3}, then it depends only on $\delta$ and $\epsilon_0$. We see that 
    \[ \cusp_i(l_{\delta,\epsilon_0}) \quad\textrm{is contained in}\quad \collar_i(\gamma_i,w_{\delta,\epsilon_0}), \] 
    which in turn implies that 
    \[ Y\setminus\bigsqcup_{i=1}^n \collar_i(\gamma_i, w_{\delta,\epsilon_0})\quad\textrm{is contained in}\quad X\setminus\bigsqcup_{i=1}^n\cusp_i(l_{\delta,\epsilon_0}). \] 
    Combining above with Part (3) of Lemma \ref{lem:comparison_geometry_lem1} and Lemma \ref{lem:comparison_geometry_lem2}, the third statement of the theorem follows. 
\end{proof}

\section{Mass distribution of eigenfunctions on collars}\label{section5}
In this section, let $f$ be a Laplacian eigenfunction with eigenvalue $\lambda$ on some hyperbolic surface $S$. If $S$ is bordered, then we further require that $f$ satisfies the \emph{Dirichlet} or \emph{Neumann boundary condition} on each component of $\pa S$. We show that under the condition that $\lambda$ is small, the $L^2$-norm of $f$ is concentrated far away from the central closed geodesics of the collars. The main proposition of this section is as follows. 
\begin{proposition}\label{prop:distribution}
    Let $0< w_1\leq w_2$. Let $\gamma$ be a simple closed geodesic in $S$ (which may be a component of $\pa S$). Suppose $\collar(\gamma,w_1)$ and $\collar(\gamma,w_2)$ are embedded in $S$. 
    \begin{enumerate}
        \item If $\lambda\leq\frac14$, then 
        \begin{align}\label{distribution_1} 
            \frac{\norm{f}_{L^2(\collar(\gamma, w_1))}^2}{\norm{f}_{L^2(\collar(\gamma, w_2))}^2} \leq \frac{w_1}{w_2}; 
        \end{align}
        \item If $\lambda\leq\frac14 - \delta^2$ for some $\delta>0$, then, 
        \begin{equation}\label{distribution_2} 
            \frac{\norm{f}_{L^2(\collar(\gamma, w_1))}^2}{\norm{f}_{L^2(\collar(\gamma, w_2))}^2}\leq \frac{4\delta w_1+\sinh(2\delta w_1)}{4\delta w_2+\sinh(2\delta w_2)}. 
        \end{equation}
        \noindent In particular, as $w_1\to \infty$, 
          \begin{equation}\label{distribution_3} 
            \frac{\norm{f}_{L^2(\collar(\gamma, w_1))}^2}{\norm{f}_{L^2(\collar(\gamma, w_2))}^2} \leq \frac{2}{e^{2\delta\cdot(w_2-w_1)}}. 
        \end{equation}
    \end{enumerate}
\end{proposition}

Before proving the proposition, we need a useful lemma: 
\begin{lemma}\label{lem:integral}
    Let $u_1(\rho),u_2(\rho)$ be smooth positive functions defined for $\rho\geq0$. If 
    \[ \frac{\dif^2 u_2}{\dif \rho^2}u_1 - u_2\frac{\dif^2 u_1}{\dif \rho^2} \geq 0 \quad\textrm{and}\quad \bigg( \frac{\dif u_2}{\dif\rho}u_1 - u_2\frac{\dif u_1}{\dif\rho} \bigg)(0) \geq 0, \] 
    then for any $w_2\geq w_1>0$, one has 
    \begin{equation*}
        \left(\int_0^{w_1}u_2^2 \middle/ \int_0^{w_2}u_2^2\right) \leq \left(\int_0^{w_1}u_1^2 \middle/ \int_0^{w_2}u_1^2\right). 
    \end{equation*} 
\end{lemma}
\begin{proof}
    Let $F(w) \df \left(\int_0^{w}u_2^2 \middle/ \int_0^{w}u_1^2\right)$ for $w>0$. Then it suffices to show that 
    \[ \frac{\dif F}{\dif w}(w) = \frac{u_2^2(w)\int_0^wu_1^2 - u_1^2(w)\int_0^wu_2^2}{\left(\int_0^wu_1^2\right)^2} \geq 0, \] 
    which is equivalent to showing that $\left(u_2^2(w)\middle/u_1^2(w)\right) \geq \left(\int_0^{w}u_2^2 \middle/ \int_0^{w}u_1^2\right)$. By Cauchy's mean value theorem, 
    \[ \frac{\int_0^{w}u_2^2}{\int_0^{w}u_1^2} = \frac{u_2^2(\rho_0)}{u_1^2(\rho_0)} \] 
    for some $\rho_0\in(0,w)$, thus, it suffices to show that $u_2/u_1$ is non-decreasing, which is equivalent to proving that  
    \[ \frac{\dif}{\dif \rho}\Big(\frac{u_2}{u_1}\Big) = \frac{\frac{\dif u_2}{\dif\rho}u_1 - u_2\frac{\dif u_1}{\dif\rho}}{u_1^2} \geq 0. \] 
    Since the derivative of the numerator is 
    \[ \frac{\dif^2 u_2}{\dif \rho^2}u_1 - u_2\frac{\dif^2 u_1}{\dif \rho^2}, \] 
    which is non-negative from our assumption, it follows that the numerator is non-decreasing, hence 
    \[ \frac{\dif u_2}{\dif\rho}u_1 - u_2\frac{\dif u_1}{\dif\rho} \geq \bigg(\frac{\dif u_2}{\dif\rho}u_1 - u_2\frac{\dif u_1}{\dif\rho}\bigg)(0) \geq 0, \] 
    from which the lemma follows. 
\end{proof}

\begin{proof}[Proof of Proposition \ref{prop:distribution}]
    In the pair of coordinates $(\rho,t)$, the hyperbolic metric is given by (see the Collar lemma, Theorem \ref{thm:collar_lemma})
    \[ \dif\rho^2 + \ell^2\cosh^2\!\rho\dif t^2, \] 
    where $\ell$ is the length of $\gamma$, hence the geometric Laplacian $\Delta$ is given by  
    \begin{equation}\label{eqn:distribution1}
        -\Delta = \frac{\pa^2}{\pa\rho^2} + \tanh\rho\frac{\pa}{\pa\rho} + \frac{1}{\ell^2\cosh^2\!\rho}\frac{\pa^2}{\pa t^2}. 
    \end{equation}
    Therefore if $\Delta f = \lambda f$, then by \eqref{eqn:distribution1} one has 
    \begin{equation}\label{eqn:distribution2}
        \left( \frac{\pa^2}{\pa\rho^2} + \tanh\rho\frac{\pa}{\pa\rho} + \lambda + \frac{1}{\ell^2\cosh^2\!\rho}\frac{\pa^2}{\pa t^2} \right) f = 0. 
    \end{equation}
    As an eigenfunction of the Laplace operator, $f$ is smooth, and we may assume the Fourier series of $f$ with respect to $t\in\R/\Z$ is 
    \begin{equation}\label{eqn:distribution3}
        f(\rho,t) = \alpha_0(\rho) + \sum_{j=1}^{\infty}\big(\alpha_j(\rho)\cos(2\pi jt) + \beta_j(\rho)\sin(2\pi jt)\big), 
    \end{equation}
    then it follows from \eqref{eqn:distribution2} and \eqref{eqn:distribution3} that for each $j\geq0$, the coefficient functions $\alpha_j(\rho)$ and $\beta_j(\rho)$ satisfy the following ordinary differential equation: 
    \begin{equation}\label{eqn:distribution4}
        \frac{\dif^2\phi}{\dif\rho^2} + \tanh\rho\frac{\dif\phi}{\dif\rho} + \left(\lambda - \frac{4\pi^2j^2}{\ell^2\cosh^2\!\rho}\right) \phi = 0. 
    \end{equation}
    Let $u(\rho) = \sqrt{\cosh\rho}\cdot\phi(\rho)$. Then \eqref{eqn:distribution4} is reduced to 
    \begin{equation}\label{eqn:distribution5}
        \frac{\dif^2 u}{\dif\rho^2} = \left( \frac 14 - \lambda + \left(\frac14+\frac{4\pi^2j^2}{\ell^2}\right)\cdot\frac{1}{\cosh^2\!\rho} \right)u. 
    \end{equation}
    Equation \eqref{eqn:distribution5} has two linearly independent solutions $\varphi_j$, $\psi_j$ satisfying
    \[ \varphi_j(0) = \frac{\dif\psi_j}{\dif\rho}(0) = 0, \textrm{\quad and\quad} \frac{\dif\varphi_j}{\dif\rho}(0) = \psi_j(0) = 1. \] 
    It can be easily checked that $-\varphi_j(-\rho)$ and $\psi_j(-\rho)$ also satisfy \eqref{eqn:distribution5}. Therefore by the uniqueness of solutions of ordinary differential equations, one has 
    \[ \varphi_j(\rho) = -\varphi_j(-\rho), \quad \psi_j(\rho) = \psi_j(-\rho). \] 
    Moreover one has $\varphi_j(\rho),\psi_j(\rho)>0$ when $\rho>0$; otherwise suppose $\varphi_j(\rho_0)=0$ for some $\rho_0>0$ and $\varphi_j(\rho)>0$ for $0<\rho<\rho_0$. By the Mean Value Theorem, there exists a $\rho_1\in(0,\rho_0)$ with $(\dif\varphi_j/\dif\rho)(\rho_1)=0$, hence there exists a $\rho_2\in(0,\rho_1)$ with $(\dif^2\varphi_j/\dif\rho^2)(\rho_2)<0$, which is in contradiction with \eqref{eqn:distribution5}. The case of $\psi_j$ can be similarly proved. 
    
    Since both $\sqrt{\cosh\rho}\cdot\alpha_j$ and $\sqrt{\cosh\rho}\cdot\beta_j$ satisfy \eqref{eqn:distribution5}, one can write them as 
    \begin{equation}\label{eqn:distribution6}
        \begin{aligned}
            \sqrt{\cosh\rho}\cdot\alpha_j(\rho) &= a_{j1}\cdot\varphi_j(\rho) + a_{j2}\cdot\psi_j(\rho),\\
            \sqrt{\cosh\rho}\cdot\beta_j(\rho) &= b_{j1}\cdot\varphi_j(\rho) + b_{j2}\cdot\psi_j(\rho),
        \end{aligned}
    \end{equation}
    where $a_{ji},b_{ji}$ are constants depending on $f$. In particular, by \eqref{eqn:distribution6} we have 
    \begin{equation}\label{eqn:distribution7}
        \begin{aligned}
            \alpha_j(0) &= a_{j2}, \\
            \beta_j(0)  &= b_{j2}, 
        \end{aligned}
        \quad
        \begin{aligned}
            (\dif\alpha_j/\dif\rho)(0) &= a_{j1}; \\
            (\dif\beta_j/\dif\rho)(0)  &= b_{j1}. 
        \end{aligned}
    \end{equation}

    If $\gamma$ is not a component of $\pa S$, then we have 
    \begin{equation}\label{eqn:norm_collar_1}
        \begin{aligned}
            \norm{f}_{L^2\left(\collar(\gamma, w)\right)}^2 
            &= \int_{0}^{1}\int_{-w}^{w}f(\rho,t)^2\ell\cosh\rho \dif\rho \dif t\\
            \textrm{(by \eqref{eqn:distribution3})\qquad} 
            &= \int_{-w}^{w}\Big(\alpha_0^2 + \frac12\sum_{j=1}^\infty(\alpha_j^2+\beta_j^2)\Big)\ell\cosh\rho \dif\rho\\
            \textrm{(by \eqref{eqn:distribution6})\qquad} 
            &= \int_{0}^{w} 2\ell \Big(a_{01}^2\varphi_0^2+a_{02}^2\psi_0^2 + \frac12\sum_{j=1}^\infty(c_{j1}^2\varphi_j^2+c_{j2}^2\psi_j^2)\Big) \dif\rho, 
        \end{aligned}
    \end{equation}
    where $c_{ji}^2=a_{ji}^2+b_{ji}^2$. In the second equality we have used that $\int_{0}^{1}\cos^2(2\pi jt)\dif t = \int_{0}^{1}\sin^2(2\pi jt)\dif t = 1/2$, and in the third equality we have used that $\int_{-w}^w \varphi_j\psi_j = 0$. 
    
    If $\gamma\subset\pa S$, and $f$ satisfies the Neumann boundary condition on $\gamma$, \ie 
    \begin{align*}
        0=\frac{\pa f}{\pa\rho}(0,t) 
        &= \frac{\dif\alpha_0}{\dif\rho}(0) + \sum_{j=1}^\infty\big(\frac{\dif\alpha_j}{\dif\rho}(0)\cos(2\pi jt) + \frac{\dif\beta_j}{\dif\rho}(0)\sin(2\pi jt)\big)\\
        \textrm{(by \eqref{eqn:distribution7})\qquad}
        &= a_{01} + \sum_{j=1}^\infty\big(a_{j1}\cos(2\pi jt) + b_{j1}\sin(2\pi jt)\big), 
    \end{align*}
    then we have 
    \[ a_{01}=a_{j1} = b_{j1} = 0, \ \Forall j\geq 1. \] 
    It follows that 
    \begin{equation}\label{eqn:norm_collar_2}
        \begin{aligned}
            \norm{f}_{L^2\left(\collar(\gamma, w)\right)}^2 
            &= \int_{0}^{1}\int_{0}^{w}f(\rho,t)^2\ell\cosh\rho \dif\rho \dif t\\
            \textrm{(by \eqref{eqn:distribution3})\qquad} 
            &= \int_{0}^{w}\Big(\alpha_0^2 + \frac12\sum_{j=1}^\infty(\alpha_j^2+\beta_j^2)\Big)\ell\cosh\rho \dif\rho\\
            \textrm{(by \eqref{eqn:distribution6})\qquad} 
            &= \int_{0}^{w} \ell \Big(a_{02}^2\psi_0^2 + \frac12\sum_{j=1}^\infty c_{j2}^2\psi_j^2\Big) \dif\rho. 
        \end{aligned}
    \end{equation}
    Similarly if $f$ satisfies the Dirichlet boundary condition, \ie $f$ vanishes on $\gamma$, then 
    \begin{equation}\label{eqn:norm_collar_3}
        \begin{aligned}
            \norm{f}_{L^2\left(\collar(\gamma, w)\right)}^2 = \int_{0}^{w} \ell \Big(a_{01}^2\varphi_0^2 + \frac12\sum_{j=1}^\infty c_{j1}^2\varphi_j^2\Big) \dif\rho. 
        \end{aligned}
    \end{equation}

    For each $j\geq0$, and for all $\rho>0$, since $\varphi_j(\rho),\psi_j(\rho)>0$, by \eqref{eqn:distribution5} we have 
    \begin{align*}
        \frac{\dif^2\varphi_j}{\dif\rho^2} \geq \Big(\frac14-\lambda\Big)\varphi_j, \quad \frac{\dif^2\psi_j}{\dif\rho^2} \geq \Big(\frac14-\lambda\Big)\psi_j. 
    \end{align*}
    Suppose $\lambda\leq\frac14-\delta^2$ for some $\delta\geq0$. Let $u_1=\cosh(\delta\rho)$, $u_2=\varphi_j$ or $\psi_j$. One can easily check that $u_1$ and $u_2$ satisfy the conditions of Lemma \ref{lem:integral}. Thus by \eqref{eqn:norm_collar_1}, \eqref{eqn:norm_collar_2} and \eqref{eqn:norm_collar_3}, for any simple closed geodesic $\gamma$ in $S$ we have 
    \begin{align*}
        \frac{\norm{f}_{L^2(\collar(\gamma, w_1))}^2}{\norm{f}_{L^2(\collar(\gamma, w_2))}^2}
        &\leq \sup_{j\geq0} \bigg\{ \frac{\int_{0}^{w_1}\varphi_j^2\dif\rho}{\int_{0}^{w_2}\varphi_j^2\dif\rho},\ \frac{\int_{0}^{w_1}\psi_j^2\dif\rho}{\int_{0}^{w_2}\psi_j^2\dif\rho} \bigg\}\\
        \textrm{(by Lemma \ref{lem:integral})\qquad}
        &\leq \frac{\int_{0}^{w_1}\cosh^2(\delta\rho)\dif\rho}{\int_{0}^{w_2}\cosh^2(\delta\rho)\dif\rho}. 
    \end{align*}
    If $\delta=0$, then 
    \[ \frac{\norm{f}_{L^2(\collar(\gamma, w_1))}^2}{\norm{f}_{L^2(\collar(\gamma, w_2))}^2} \leq \frac{\int_{0}^{w_1}1\dif\rho}{\int_{0}^{w_2}1\dif\rho} = \frac{w_1}{w_2}, \] 
    which proves \eqref{distribution_1}. If $\delta>0$, then 
    \[ \frac{\norm{f}_{L^2(\collar(\gamma, w_1))}^2}{\norm{f}_{L^2(\collar(\gamma, w_2))}^2} \leq \frac{4\delta w_1+\sinh(2\delta w_1)}{4\delta w_2+\sinh(2\delta w_2)}, \] 
    which proves \eqref{distribution_2}. 
\end{proof}
\begin{remark*}
    Mondal \cite[Lemma 3.17]{Mondal2015} had proved that, for any $\ell< \epsilon\leq \epsilon_0< 1$, let $w=\arccosh(\epsilon/\ell)$, $w_0=\arccosh(\epsilon_0/\ell)$ and $w_1=\arccosh(1/\ell)$, then the ratio 
    \[ \frac{\norm{f}^2_{L^2( \collar(w_0)\setminus\collar(w) )}}{\norm{f}^2_{L^2( \collar(w_1)\setminus\collar(w_0) )}} \] 
    has an upper bound depending only on $\epsilon$ and $\epsilon_0$. Gamburd \cite[Lemma 4.2]{Gamburd2002} had also proved a similar statement on mass distribution of eigenfunctions on funnels, which asserted that for any $w>0$, the ratio 
    \[ \frac{\norm{f}^2_{L^2( \collar(w+\log2)\setminus\collar(w) )}}{\norm{f}^2_{L^2( \collar(\infty)\setminus\collar(w+\log2) )}} \] 
    has a uniform lower bound depending only on the eigenvalue $\lambda$. One may also see \cite[Lemma 4.1]{BM2001}, \cite[Lemma 4.1]{Gamburd2002} and \cite[Lemma 3.1]{Mondal2015} for similar results on mass distribution of eigenfunctions on cusps. 
\end{remark*}

\section{Proof of Theorem \ref{thm:comparison_eigenvalue}}\label{section6}
In this section we prove Theorem \ref{thm:comparison_eigenvalue}, which is stated earlier in the introduction. Let $X$ be a finite-area non-compact hyperbolic surface with $n$ large cusps, and recall that 
\[ X^{\mathrm{fu}} = X\setminus\bigsqcup_{i=1}^n\cusp_i(\epsilon_0) \] 
is an infinite-area hyperbolic surface with $n$ funnel ends, where $\epsilon_0$ is a positive constant. Let $Y$ be the compact convex core of $X^{\mathrm{fu}}$, and let $\gamma_1,\gamma_2,\cdots,\gamma_n$ be the set of all components of $\pa Y$. Note that the Poincar\'e metric $ds_{Y}^2$ of $Y$ is just the restriction of $ds_{X^{\mathrm{fu}}}^2$ on $Y$. 

Let $f$ be a normalized first Neumann eigenfunction of $Y$. If $\sigma_1(Y)\geq\frac14$, then Theorem \ref{thm:comparison_eigenvalue} is obvious; from now on we always assume that 
\[ \sigma_1(Y)\leq\frac14. \]  
By Lemma \ref{lem:RQ2}, to prove Theorem \ref{thm:comparison_eigenvalue} it suffices to show that for $l$ sufficiently large we have 
\begin{equation}\label{comparison:RQ}
    \sigma_1(Y) = \int_Y\abs{\nabla_{Y} f}^2 \geq (1-4\delta)\cdot\frac{\int_X\abs{\nabla_{X} h}^2}{\int_X h^2 - \frac{1}{\area(X)}(\int_X h)^2} - \delta 
\end{equation}
for some test function $h$ on $X$. 

The proof of \eqref{comparison:RQ} will be divided into four parts. 
We first construct the test function $h$. 
For any $\delta>0$, let $l'_\delta=l_{\delta,\epsilon_0}$ be the constant given by Theorem \ref{thm:comparison_geometry}. Suppose $X$ has large cusps of length $l\geq l'_\delta$, then $\pa Y$ has long half-collars of width $w=w(l)$, and moreover, 
\[ w\to\infty \textrm{\quad as\quad} l\to\infty. \] 
Let $\chi_1=\chi_1(\rho)$ be a universal smooth cut-off function defined for $\rho\geq0$ such that 
\begin{align*}
    \chi_1(\rho) &= 0 \text{\quad for } 0\leq\rho\leq 1,\\
    0 \leq \chi_1(\rho) &\leq 1 \text{\quad for } 1\leq\rho\leq 2,\\
    \chi_1(\rho) &= 1 \text{\quad for } \rho\geq 2,\\
    \text{and\quad} \abs{\chi'_1(\rho)} &\leq 2 \text{\quad for } \rho\geq 0. 
\end{align*}
Suppose $w\geq1$. Let $\chi$ be a smooth function defined on $Y$ such that $\chi$ equals to $\chi_1(2\rho/\sqrt{w})$ in the pair of coordinates $(\rho,t)$ on each half-collar $\collar(\gamma_i,w)$, and equals to $1$ elsewhere. Then we have 
\begin{equation}\label{bound_chi}
    0\leq\chi\leq 1, \quad \abs{\nabla\chi}\leq \frac{4}{\sqrt{w}}. 
\end{equation} 
Multiply $f$ by $\chi$ and extend $\chi f$ by zero to the whole $X$. Denote the extension function by $h$, then it will be our desired test function.

\begin{lemma}\label{lem:comparison_eigenvalue_lem1}
    For $l$ sufficiently large we have 
    \begin{equation*}
        \int_Y \abs{\nabla_{Y} f}^2 \geq \int_Y \abs{\nabla_{Y} h}^2 - \delta. 
    \end{equation*}
\end{lemma}
\begin{proof}
    To simplify notations, we omit the subscript $Y$. By definition of $h$ we have 
    \[ \int_Y \abs{\nabla h}^2 
    = \int_Y \abs{\nabla \chi}^2f^2 + \int_Y 2\chi f\nabla\chi\cdot\nabla f + \int_Y \chi^2\abs{\nabla f}^2. \] 
    By \eqref{bound_chi} we have 
    \begin{align*}
        \abs{ \int_Y 2\chi f\nabla\chi\cdot\nabla f} 
        \leq \int_Y 2\abs{f}\abs{\nabla\chi}\abs{\nabla f}
        \leq \frac{8}{\sqrt{w}}\int_Y \abs{f}\abs{\nabla f}, 
    \end{align*} 
    since $f$ is normalized, 
    \[ \int_Y f^2 = 1, \quad \int_Y \abs{\nabla f}^2 = \sigma_1(Y), \] 
    hence by Cauchy's inequality we find that 
    \begin{align*}
        \int_Y \abs{\nabla h}^2 
        &\leq \frac{16}{w}\int_Y f^2 + \frac{8}{\sqrt{w}}\int_Y \abs{f}\abs{\nabla f} + \int_Y \abs{\nabla f}^2\\
        &\leq \frac{16}{w} + \frac{8}{\sqrt{w}}\sqrt{\sigma_1(Y)} + \int_Y \abs{\nabla f}^2\\
        &\leq \frac{20}{\sqrt{w}} + \int_Y \abs{\nabla f}^2 \leq \delta+\int_Y \abs{\nabla f}^2, 
    \end{align*}
    when $w$ is sufficiently large. The lemma follows. 
\end{proof}

\begin{lemma}\label{lem:comparison_eigenvalue_lem2}
    We have 
    \begin{equation*}
        \int_Y \abs{\nabla_Y h}^2 \dvol_{Y} = \int_X\abs{\nabla_X h}^2 \dvol_{X}. 
    \end{equation*}
\end{lemma}
\begin{proof}
    Since $h$ is the zero-extension of $\chi f$, 
    \[ \int_X \abs{\nabla_X h}^2 \dvol_{X} = \int_Y \abs{\nabla_X h}^2 \dvol_{X}. \] 
    By the conformal invariance of two-dimensional Dirichlet energy, 
    \[ \int_Y \abs{\nabla_X h}^2 \dvol_{X} = \int_Y \abs{\nabla_Y h}^2 \dvol_{Y}. \]
    The lemma follows. 
\end{proof}

By Lemma \ref{lem:comparison_eigenvalue_lem1} and \ref{lem:comparison_eigenvalue_lem2}, to prove the inequality \eqref{comparison:RQ} it suffices to prove 
\begin{equation}\label{eqn:denominator}
    \int_X h^2 \dvol_{X} - \frac{1}{\area(X)}\Big(\int_X h \dvol_{X}\Big)^2
    \geq 1-4\delta. 
\end{equation}
The proof of \eqref{eqn:denominator} will be divided into two parts, following the strategy of the proof of \cite[Lemma 1.2]{BM2001}. 
\begin{lemma}\label{lem:comparison_eigenvalue_lem3}
    For $l$ sufficiently large we have 
    \begin{equation*}
        \int_X h^2 \dvol_{X} \geq 1 - 2\delta. 
    \end{equation*}
\end{lemma}
\begin{proof}
    To simplify notations, let 
    \[ \collar(\cdot) \df \bigsqcup_{i=1}^n\collar(\cdot,\gamma_i). \] 
    Let $l$ be sufficiently large such that $\sqrt{w}\geq w_{\delta,\epsilon_0}$ that is given by Theorem \ref{thm:comparison_geometry}, then we have 
    \begin{equation}\label{eqn:comparison_eigenvalue_lem3-1}
        \begin{aligned}
            \int_X h^2 \dvol_{X} = \int_Y h^2 \dvol_{X}
            &\geq \int_{Y\setminus\collar(w_{\delta,\epsilon_0})} h^2 \dvol_{X}\\
            \textrm{(by Corollary \ref{cor:comparison_volume})\qquad} 
            &\geq \left(1-\delta\right)\cdot \int_{Y\setminus\collar(w_{\delta,\epsilon_0})} (\chi f)^2 \dvol_{Y}\\
            \textrm{(since $\sqrt{w}\geq w_{\delta,\epsilon_0}$)\qquad} 
            &\geq \left(1-\delta\right)\cdot \int_{Y\setminus\collar(\sqrt{w})} (\chi f)^2 \dvol_{Y}.
        \end{aligned}
    \end{equation}
    By the definition of $\chi$ we have 
    \begin{equation}\label{eqn:comparison_eigenvalue_lem3-2}
        \int_{Y\setminus\collar(\sqrt{w})} (\chi f)^2 \dvol_{Y} = \int_{Y} f^2 \dvol_{Y} - \int_{\collar(\sqrt{w})} f^2 \dvol_{Y}. 
    \end{equation}
    By Part (1) of Proposition \ref{prop:distribution} we have 
    \begin{equation}\label{eqn:comparison_eigenvalue_lem3-3}
        \int_{\collar(\sqrt{w})} f^2 \dvol_{Y} \leq \frac{1}{\sqrt{w}}\cdot \int_{\collar(w)} f^2 \dvol_{Y} \leq \frac{1}{\sqrt{w}}.
    \end{equation}
    It follows from \eqref{eqn:comparison_eigenvalue_lem3-1}, \eqref{eqn:comparison_eigenvalue_lem3-2} and \eqref{eqn:comparison_eigenvalue_lem3-3} that 
    \begin{align*}
        \int_X h^2 \dvol_{X} 
        &\geq (1-\delta)\cdot(1-\frac{1}{\sqrt{w}})\\
        &\geq 1-\delta-\frac{1}{\sqrt{w}} \geq 1-2\delta, 
    \end{align*}
when $w$ is sufficiently large. The lemma follows. 
\end{proof}

\begin{lemma}\label{lem:comparison_eigenvalue_lem4}
    For $l$ sufficiently large we have 
    \begin{equation*}
        \frac{1}{\area(X)}\Big(\int_X h \dvol_{X}\Big)^2 \leq 2\delta. 
    \end{equation*}
\end{lemma}
\begin{proof}
    By definition of $h$ we have 
    \begin{align}\label{eqn:comparison_eigenvalue_lem4-1}
        \int_X h \dvol_{X} 
        = \int_Y h \dvol_{X}
        = \int_{Y\setminus\collar(\sqrt{w})} f \dvol_{X} + \int_{\collar(\sqrt{w})} h \dvol_{X}. 
    \end{align}

    \noindent We first estimate the first term of the RHS of \eqref{eqn:comparison_eigenvalue_lem4-1}. Let $l$ be sufficiently large such that $\sqrt{w}\geq w_{\delta,\epsilon_0}$, then by Corollary \ref{cor:comparison_volume} we have 
    \begin{equation}\label{eqn:comparison_eigenvalue_lem4-2}
        \begin{aligned}
            \abs{ \int_{Y\setminus\collar(\sqrt{w})} f \dvol_{X} - \int_{Y\setminus\collar(\sqrt{w})} f \dvol_{Y} }
            &\leq \delta\cdot \int_{Y\setminus\collar(\sqrt{w})} \abs{f}\dvol_{Y}\\
            &\leq \delta\cdot \sqrt{\area\big(Y\setminus\collar(\sqrt{w})\big)}. 
        \end{aligned}
    \end{equation}
    Since $f$ is a first Neumann eigenfunction of $Y$, we have $\int_Y f \dvol_{Y} = 0$, thus 
    \begin{equation}\label{eqn:comparison_eigenvalue_lem4-3}
        \begin{aligned}
            \abs{ \int_{Y\setminus\collar(\sqrt{w})} f \dvol_{Y} }
            & = \abs{ \int_{\collar(\sqrt{w})} f \dvol_{Y} }\\
            &\leq \sqrt{\int_{\collar(\sqrt{w})} f^2 \dvol_{Y}} \cdot \sqrt{\area\big(\collar(\sqrt{w})\big)}, 
        \end{aligned}
    \end{equation}
    then it follows from \eqref{eqn:comparison_eigenvalue_lem4-3} and \eqref{eqn:comparison_eigenvalue_lem3-3} that 
    \begin{align}\label{eqn:comparison_eigenvalue_lem4-4}
        \abs{ \int_{Y\setminus\collar(\sqrt{w})} f \dvol_{Y} }
        \leq \sqrt{\frac{1}{\sqrt{w}}}\cdot \sqrt{\area\big(\collar(\sqrt{w})\big)}. 
    \end{align}
    Combining \eqref{eqn:comparison_eigenvalue_lem4-2} and \eqref{eqn:comparison_eigenvalue_lem4-4}, we find that 
    \begin{equation}\label{eqn:comparison_eigenvalue_lem4-5}
        \begin{aligned}
            &\abs{ \int_{Y\setminus\collar(\sqrt{w})} f \dvol_{X}}\\
            \leq & \delta\cdot\sqrt{\area\big(Y\setminus\collar(\sqrt{w})\big)} + \sqrt{\frac{1}{\sqrt{w}}}\cdot \sqrt{\area\big(\collar(\sqrt{w})\big)}. 
        \end{aligned}
    \end{equation}
    We then estimate the second term in the RHS of \eqref{eqn:comparison_eigenvalue_lem4-1}. We have 
    \begin{equation}\label{eqn:comparison_eigenvalue_lem4-6}
        \begin{aligned}
            \abs{ \int_{\collar(\sqrt{w})} h \dvol_{X} } 
            &\leq \int_{\collar(\sqrt{w})} \abs{h} \dvol_{X}\\
            \textrm{(by definition of $h$)\qquad} 
            &\leq \int_{\collar(\sqrt{w})} \abs{f} \dvol_{X}\\
            \textrm{(since $\dvol_X\leq \dvol_Y$)\qquad} 
            &\leq \int_{\collar(\sqrt{w})} \abs{f} \dvol_{Y}\\
            \textrm{(By Cauchy's inequality)\qquad} 
            &\leq \sqrt{\frac{1}{\sqrt{w}}}\cdot \sqrt{\area\big(\collar(\sqrt{w})\big)}
        \end{aligned}
    \end{equation}
    Now combining \eqref{eqn:comparison_eigenvalue_lem4-1}, \eqref{eqn:comparison_eigenvalue_lem4-5} and \eqref{eqn:comparison_eigenvalue_lem4-6}, by Cauchy's inequality we have 
    \begin{equation*}
        \begin{aligned}
            \Big(\int_X h \dvol_{X}\Big)^2 
            &\leq \Big(\delta^2 + \frac{4}{\sqrt{w}}\Big) \cdot \area(Y)\\
            &  =  \Big(\delta^2 + \frac{4}{\sqrt{w}}\Big) \cdot \area(X).
        \end{aligned}
    \end{equation*}
    Thus for $w$ sufficiently large, 
    \[ \frac{1}{\area(X)}\Big(\int_X h \dvol_{X}\Big)^2 \leq \delta^2 + \frac{4}{\sqrt{w}} \leq 2\delta. \] 
    The lemma follows. 
\end{proof}

\begin{proof}[Proof of Theorem \ref{thm:comparison_eigenvalue}]
    By Lemma \ref{lem:comparison_eigenvalue_lem1}, \ref{lem:comparison_eigenvalue_lem2}, \ref{lem:comparison_eigenvalue_lem3} and \ref{lem:comparison_eigenvalue_lem4}, the inequality \eqref{comparison:RQ} holds when $\sigma_1(Y)\leq 1/4$ and $l$ is sufficiently large, thus the theorem follows. 
\end{proof}

\section{Proof of Theorem \ref{main}}\label{section7}
In this section, assuming Theorem \ref{thm:stability}, we first prove Proposition \ref{prop:piece}, then use Proposition \ref{prop:piece} to prove Theorem \ref{main}. We postpone the proof of Theorem \ref{thm:stability} until Section \ref{section8}.

\begin{proof}[Proof of Proposition \ref{prop:piece} by assuming Theorem \ref{thm:stability}]
    Fix $\epsilon>0$, and let $\epsilon_0=\epsilon/\pi$. For any $\delta>0$, there exists an $l_{\delta,\epsilon_0}$ by Theorem \ref{thm:comparison_eigenvalue}, and an $l'_{\delta,\epsilon_0}$ by Theorem \ref{thm:comparison_geometry}. Let $l$ be sufficiently large such that 
    \[ l\geq\max\{l_{\delta,\epsilon_0},l'_{\delta,\epsilon_0}\} \textrm{\quad and\quad} w(l)\geq\epsilon, \] 
    where $w(l)$ is defined in Part (2) of Theorem \ref{thm:comparison_geometry}. 
    
    \underline{Step-1}. 
    By Proposition \ref{prop:surfaces}, for each sufficiently large $g$, there exists a finite-area non-compact hyperbolic surface $\sX_{2i}$, such that 
    \begin{enumerate}
        \item $\abs{\chi(\sX_{2i})} = 2i = 2g$; 
        \item All simple closed geodesics in $\sX_{2i}$ have length $\geq 2\epsilon$; 
        \item $\sX_{2i}$ has large cusps of length $l$; 
        \item $\bar\lambda_1(\sX_{2i}) \geq \frac14 - \delta$. 
    \end{enumerate}

    \underline{Step-2}. 
    For each $i$, we apply the compactification procedure in the introduction to $\sX_{2i}$ to obtain a bordered compact hyperbolic surface $\sY_{2i}$ which is a compact convex core of an infinite-area hyperbolic surface with funnel ends. By Theorem \ref{thm:comparison_geometry}, each component of $\pa\sY_{2i}$ has length between $(1-\delta)\epsilon$ and $(1+\delta)\epsilon$; and for each simple closed geodesic $\gamma$ in $\sY_{2i}$ which is not a component of $\pa\sY_{2i}$, we have 
    \[ \length_{\sY_{2i}}(\gamma) > \length_{\sX_{2i}}(\gamma) \geq 2\epsilon. \] 
    Moreover by Theorem \ref{thm:comparison_eigenvalue}, we have 
    \begin{align*}
        \sigma_1(\sY_{2i}) 
        &\geq  (1-4\delta)\cdot\bar\lambda_1(\sX_{2i}) - \delta \\
        &\geq \frac 14 - 3\delta. 
    \end{align*}
    Thus the bordered compact hyperbolic surface $\sY_{2i}$ satisfies that 
    \begin{enumerate}
        \item $\abs{\chi(\sY_{2i})} = 2g$; 
        \item All simple closed geodesic in $\sY_{2i}$ have length $\geq 2\epsilon$; 
        \item $\pa\sY_{2i}$ has long half-collars of width $\epsilon$ (by Part (2) of Theorem \ref{thm:comparison_geometry}); 
        \item $\sigma_1(\sY_{2i}) \geq \frac14 - 3\delta$; 
        \item Each component of $\pa\sY_{2i}$ has length between $(1-\delta)\epsilon$ and $(1+\delta)\epsilon$. 
    \end{enumerate}

    \underline{Step-3}. 
    It was proved in \cite[Theorem 2]{Zograf1987} (see also \cite[Theorem 13]{SW2022}) that 
    \[ \lim_{n\to\infty} \sup_{X\in\sM_{0,n}} \lambda_1(X) = 0, \] 
    we deduce that for $g$ sufficiently large, $\sX_{2i}$ has genus $\geq1$. So $\sY_{2i}$ also has genus $\geq1$. Let $\{\gamma_j\}$ be set of all components of $\pa\sY_{2i}$. Therefore for sufficiently small $\delta$ we can apply Theorem \ref{thm:stability} to $\sY_{2i}$ with 
    \[ \delta_j = \frac{\epsilon}{\ell(\gamma_j)}-1, \] 
    to obtain a bordered compact hyperbolic surface $\sY'_{2i}$. Moreover, for $\delta$ sufficiently small the map $\mu:\sY_{2i} \to \sY_{2i}'$ is a bi-Lipschitz map with Lipschitz constant $2$, it follows that 
    \begin{enumerate}
        \item $\abs{\chi(\sY'_{2i})} = 2g$; 
        \item All simple closed geodesic in $\sY'_{2i}$ have length $\geq \epsilon$; 
        \item $\pa\sY'_{2i}$ has long half-collars of width $\epsilon / 2$; 
        \item $\sigma_1(\sY'_{2i}) \geq \frac14 - O(\delta)$, where the implied constant depends only on $\epsilon$; 
        \item Each component of $\pa\sY'_{2i}$ has the same length $\epsilon$. 
    \end{enumerate}

    \underline{Step-4}. 
    For each $i$, the Euler characteristic of $\sY'_{2i}$ is even, thus $\pa\sY'_{2i}$ has an even number of components $\{\gamma'_1,\cdots,\gamma'_{2n}\}$ for some $n\geq 1$. Gluing $\gamma'_{2j-1}$ with $\gamma'_{2j}$ with arbitrary twist parameters for $2\leq j\leq n$, we obtain a bordered compact hyperbolic surface $\sY_{g_i,2}(\epsilon,\epsilon)$ with two boundary components. 
    Since 
    \[ 2g = \abs{\chi(\sY'_{2i})} = \abs{\chi(\sY_{g_i,2}(\epsilon,\epsilon))} = 2g_i-2+2, \] 
    we deduce that $g_i=g$. 
    For each simple closed geodesic in $\sY_{g,2}$, since it is either contained in $\sY'_{2i}$, a component of $\pa\sY'_{2i}$, or transitively intersecting with some $\gamma'_j$ for $2\leq j\leq n$, by \eqref{eqn:cross_collar} we find that it always has length $\geq\epsilon$. Moreover, by Lemma \ref{lem:eigenvalue_monotonicity} we have 
    \[ \sigma_1(\sY_{g,2}) \geq \sigma_1(\sY'_{2i}) \geq \frac14-O(\delta). \] 
    The proposition follows. 
\end{proof}

Now we are ready to prove Theorem \ref{main}.

\begin{proof}[Proof of Theorem \ref{main}]
    For any sufficiently large integer $i$, let $\sY_{i,2} \df \sY_{i,2}(\epsilon,\epsilon)$ be the bordered compact hyperbolic surface given by Proposition \ref{prop:piece}. 
    For any given sufficiently small $\delta>0$, we are going to construct a sequence of closed hyperbolic surfaces $\{\sZ_g\}$, one in each sufficiently large genus $g$, such that $\sZ_g\in\sM_g^{\geq \epsilon}$,  
    \begin{align*}
        \lambda_{k-1}(\sZ_g) \leq \delta \quad \textrm{and} \quad \lambda_k(\sZ_g) \geq \frac 14 -\delta. 
    \end{align*}

    If $k=1$, then we just glue the two boundary components of $\sY_{g-1,2}$ together with an arbitrary twist parameter to obtain a closed hyperbolic surface $\sZ_g$ of genus $g$. By Lemma \ref{lem:eigenvalue_monotonicity}, we have 
    \[ \lambda_1(\sZ_g)\geq\sigma_1(\sY_{g-1,2})\geq\frac14-O(\delta), \] 
    and it can be easily seen from Part (1) and Part (2) of Proposition \ref{prop:piece} that $$\sZ_g\in\sM_g^{\geq\epsilon}.$$ 

    If $k\geq2$, then for each sufficiently large $g$, let 
    \begin{equation}\label{eqn:main_pf1}
        g-1 = i_g\cdot k + r_g, \textrm{\quad where $0\leq r_g < k$}. 
    \end{equation}
    Then we can glue $k-r_g$ copies of $\sY_{i_g,2}$ and $r_g$ copies of $\sY_{i_g+1,2}$ together with arbitrary twist parameters, as illustrated by Figure \ref{fig:gluing}, to obtain a connected closed hyperbolic surface $\sZ_g$ of genus $g$. Similarly $\sZ_g\in\sM_g^{\geq\epsilon}$. Since by \eqref{eqn:main_pf1} we have 
    \[ i_g\to\infty \] 
    as $g\to\infty$, by Corollary \ref{cor:eigenvalue_bound}, for $g$ sufficiently large we have 
    \[ \lambda_{k-1}(\sZ_g) \leq \delta; \] 
    and by the Mini-max principle, \ie Theorem \ref{thm:mini-max}, 
    \[ \lambda_k(\sZ_g) \geq \min\{\sigma_1(\sY_{i_g,2}), \sigma_1(\sY_{i_g+1,2})\} \geq \frac{1}{4}-O(\delta). \] 

    Combining the two above cases, it follows that for any $\delta>0$ and any fixed $k\geq1$, we have 
    \[ \liminfg \sup_{X_g\in\sM_g^{\geq\epsilon}} \left(\lambda_k(X_g) - \lambda_{k-1}(X_g) \right) \geq \frac14-O(\delta), \]
    which implies that  
    \begin{equation}\label{eqn:main_pf2}
        \liminfg \sup_{X_g\in\sM_g^{\geq\epsilon}} \left(\lambda_k(X_g) - \lambda_{k-1}(X_g) \right) \geq \frac14. 
    \end{equation}
    On the other hand,  it follows from \eg Cheng \cite[Corollary 2.3]{Cheng1975} that for any fixed $k$, 
    \begin{equation}\label{eqn:main_pf3}
        \limsupg \sup_{X_g\in\sM_g^{\geq\epsilon}} \left(\lambda_k(X_g) - \lambda_{k-1}(X_g) \right) \leq \limsupg \sup_{X_g\in\sM_g} \lambda_k(X_g) \leq \frac 14. 
    \end{equation}
    Combining \eqref{eqn:main_pf2} and \eqref{eqn:main_pf3}, we have finished the proof of Theorem \ref{main}. 
\end{proof}

\section{Stability of Neumann eigenvalues}\label{section8}
In this section we prove Theorem \ref{thm:stability}. We first recall several basic facts and formulas on two-dimensional hyperbolic geometry.

\subsection{Pants decompositions}
A \emph{pair of pants} is a hyperbolic surface of genus zero with three totally geodesic boundaries (a cusp is viewed as a geodesic boundary component of zero length). We denote a pair of pants by $P= P(\alpha,\beta,\gamma)$, where $\alpha$, $\beta$, and $\gamma$ denote the boundary geodesics of $P$, and by abuse of notations, also denote their lengths. Each pair of pants is uniquely determined by the lengths of its three boundary geodesics. 

A \emph{pants decomposition} of a hyperbolic surface $S_{g,n}$ is a maximal collection of pairwise disjoint simple closed geodesics in $S_{g,n}$. Each pants decomposition of $S_{g,n}$ contains exactly $3g-3+n$ geodesics, and decomposes $S_{g,n}$ into $-\abs{\chi(S)} = 2g-2+n$ pairs of pants, \ie 
\[ S_{g,n} = \bigcup_{k=1}^{2g-2+n} S_{0,3}(\ell_{k1},\ell_{k2},\ell_{k3}). \] 
Here $\ell_{k1}$, $\ell_{k2}$ and $\ell_{k3}$ are lengths (which could be zero) of boundary geodesics of the corresponding pair of pants. Given a pants decomposition $\{\gamma_i\}$ of $S_{g,n}$, then the \emph{Fenchel-Nielsen coordinates} of the Teichm\"uller space of $S_{g,n}$, consists of $3g-3+n$ lengths of these $\gamma_i$'s, and $3g-3+n$ \emph{twist parameters} along these geodesics. 
See \cite[Chapter 3]{Buser1992} for more details. 
\begin{figure}
    \centering
    \includegraphics[scale=0.38]{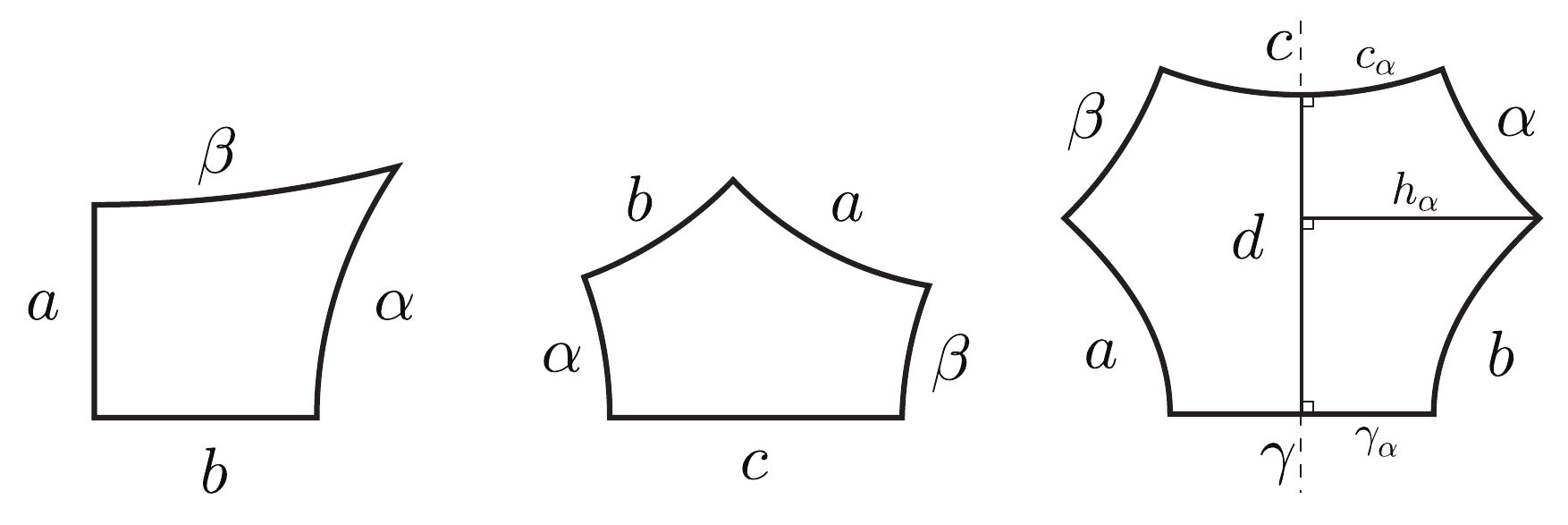}
    \caption{Hyperbolic polygons}
    \label{fig:polygon}
\end{figure}

We recall some formulas below which will be used later, and we refer the readers to \cite[Chapter 2]{Buser1992} for more details. 
\begin{lemma}[{\cite[Theorem 2.3.1]{Buser1992}}]\label{lem:trirectangle}
    For any trirectangle with consecutive sides $a$, $b$, $\alpha$ and $\beta$ (see Figure \ref{fig:polygon}), one has 
    \begin{enumerate}
        \item $\tanh \alpha = \cosh b \cdot \tanh a$. 
        \item $\sinh \alpha = \cosh \beta \cdot \sinh a$. 
    \end{enumerate}
\end{lemma}
\begin{proof}
Part (2) is exactly \cite[(v) of Theorem 2.3.1]{Buser1992}. Part (1) directly follows from \cite[(i), (ii) and (iv) of Theorem 2.3.1]{Buser1992}. 
\end{proof}

\begin{lemma}[{\cite[Theorem 2.3.4]{Buser1992}}]\label{lem:pentagon}
    For any convex right-angled pentagon with consecutive sides $a$, $b$, $\alpha$, $c$ and $\beta$ (see Figure \ref{fig:polygon}), one has 
    \begin{enumerate}
        \item $\cosh c = \sinh a \cdot \sinh b$. 
        \item $\cosh c = \coth \alpha \cdot \coth \beta$. 
    \end{enumerate}
\end{lemma}

\begin{lemma}\label{lem:pentagon'}
    Using the same notation as in Lemma \ref{lem:pentagon}, one has 
    \begin{enumerate}
        \item $\sinh^2\! c = 1/\sinh^2\! \beta + 1/(\sinh^2\! \alpha\tanh^2\! \beta)$. 
        \item $\tanh^2\! c = 1-\tanh^2\! \alpha\tanh^2\! \beta$. 
    \end{enumerate}
\end{lemma}
\begin{proof}
    By Lemma \ref{lem:pentagon}, one has $\cosh c = 1/(\tanh \alpha\tanh \beta)$, therefore
    \begin{equation*}
        \begin{aligned}
        \sinh^2\! c 
        & = \cosh^2\! c - 1\\ 
        & = \frac{1 - \tanh^2\! \alpha\tanh^2\! \beta}{\tanh^2\! \alpha\tanh^2\! \beta}\\ 
        & = \frac{\tanh^2\! \alpha - \tanh^2\! \alpha\tanh^2\! \beta + 1 - \tanh^2\! \alpha}{\tanh^2\! \alpha\tanh^2\! \beta}\\ 
        & = \frac{1}{\sinh^2\! \beta} + \frac{1}{\sinh^2\! \alpha\tanh^2\! \beta}, 
        \end{aligned}
    \end{equation*}
    and $\tanh^2\! c = 1-1/\cosh^2\! c = 1 - \tanh^2\! \alpha\tanh^2\! \beta$. 
\end{proof}

\begin{lemma}[{\cite[Theorem 2.4.1]{Buser1992}}]\label{lem:hexagon}
    For any convex right-angled hexagon with consecutive sides $a$, $\gamma$, $b$, $\alpha$, $c$ and $\beta$ (see Figure \ref{fig:polygon}), one has 
    \begin{equation*}
        \cosh \gamma = \frac{\cosh c + \cosh a \cdot \cosh b}{\sinh a \cdot \sinh b}. 
    \end{equation*}
\end{lemma}

By the formulas above one has the following lemmas: 

\begin{lemma}\label{lem:hexagon'}
    Using the same notation as in Lemma \ref{lem:hexagon}, let $d$ be the common perpendicular between $c$ and $\gamma$ (see Figure \ref{fig:polygon}), then one has 
    \begin{equation*}
        \sinh^2\! d = \frac{\cosh^2\! \alpha+\cosh^2\! \beta + 2\cosh\alpha\cosh\beta\cosh\gamma}{\sinh^2\! \gamma}. 
    \end{equation*}
\end{lemma}
\begin{proof}
    By Lemma \ref{lem:pentagon} and \ref{lem:hexagon}, one has 
    \begin{align*}
        \sinh^2\! d 
        & = \sinh^2\! \alpha \sinh^2\! b - 1\\ 
        & = \sinh^2\! \alpha \left(\frac{\cosh \beta + \cosh\alpha\cosh\gamma}{\sinh\alpha\sinh\gamma}\right)^2 - \sinh^2\! \alpha - 1\\ 
        & = \frac{(\cosh \beta + \cosh\alpha\cosh\gamma)^2 - \cosh^2\! \alpha\sinh^2\! \gamma}{\sinh^2\! \gamma}\\
        & = \frac{\cosh^2\! \alpha+\cosh^2\! \beta + 2\cosh\alpha\cosh\beta\cosh\gamma}{\sinh^2\! \gamma}. 
    \end{align*}
    The proof is complete. 
\end{proof}

\begin{lemma}\label{lem:hexagon_altitude}
   Using the same notation as in Lemma \ref{lem:hexagon'}, let $h_\alpha$ be the altitude of the pentagon bounded by $\alpha$, $b$ and $d$ (see Figure \ref{fig:polygon}), then one has 
    \begin{equation*}
        \sinh h_\alpha \leq \frac{\cosh^2\! \gamma}{\tanh \alpha}. 
    \end{equation*}
\end{lemma}
\begin{proof}
    By Lemma \ref{lem:trirectangle} and \ref{lem:pentagon}, we have 
    \begin{equation}\label{eqn:hexagon1}
        \sinh h_\alpha = \cosh \alpha \sinh c_\alpha = \frac{\cosh \gamma_\alpha}{\tanh \alpha} = \frac{1}{\tanh \alpha\tanh b\tanh d}. 
    \end{equation}
    From Lemma \ref{lem:hexagon} and \ref{lem:hexagon'}, one can find that 
    \begin{equation*}
        \sinh b\cdot\sinh \gamma \geq 1 \quad \textrm{and} \quad \sinh d\cdot\sinh \gamma \geq 1, 
    \end{equation*}
    which leads to  
    \begin{equation}\label{eqn:hexagon2}
        \frac{1}{\tanh b} \leq \cosh \gamma \quad \textrm{and}  \quad \frac{1}{\tanh d} \leq \cosh \gamma. 
    \end{equation}
    Then the lemma follows from \eqref{eqn:hexagon1} and \eqref{eqn:hexagon2}. 
\end{proof}

\subsection{Proof of Theorem \ref{thm:stability}}
In this subsection, we prove Theorem \ref{thm:stability} assuming the following technical proposition, whose proof is delayed to the next subsection. 
\begin{proposition}\label{prop:mu_pants}
    Let $P$ (resp., $P'$) be a hyperbolic pair of pants with boundaries $\alpha$, $\beta$ and $\gamma$ (resp., $\alpha'$, $\beta'$ and $\gamma'$). Suppose that there exists a constant $\ell>0$ such that $\length(\gamma), \length(\gamma') \leq \ell$. Let 
    \[ \delta \df \abs{\frac{\length(\gamma)}{\length(\gamma')}-1}, \] 
    and assume that $450e^{5\ell}\sqrt{\delta} \leq 1$. If $\length(\alpha) = \length(\alpha')$, $\length(\beta) = \length(\beta')$, and $\length(\alpha), \length(\beta) \geq 2\arcsinh1$, then there exists a piecewise smooth homeomorphism 
    \[ \mu:P\to P', \] 
    which maps the boundaries of $P$ to the corresponding boundaries of $P'$, such that 
    \begin{enumerate}
        \item $\eval{\mu}_{\alpha}$ and $\eval{\mu}_{\beta}$ are the identity maps. 
        \item At the points where $\mu$ is smooth, we have \[ \big(1+O(\sqrt{\delta})\big) \cdot ds_{P}^2 \geq \mu^*(ds_{P'}^2) \geq \big(1-O(\sqrt{\delta})\big) \cdot ds_{P}^2, \] where the implied constants depend only on $\ell$. 
    \end{enumerate}
\end{proposition}
\begin{remark*}
    Given two hyperbolic pairs of pants $P=P(\alpha, \beta, \gamma)$ and $P'=P(\alpha,\beta,\gamma')$, assume that $\varepsilon=\abs{\log(\gamma'/\gamma)}\leq2$. Bishop \cite[Theorem 1.1]{Bishop2002} had shown that, if $\alpha,\beta,\gamma,\gamma'\leq L$ for some constant $L>0$, then there exists a quasiconformal map $f:P\to P'$ which is affine on each of the boundary geodesics, such that its maximal dilation $K(f)$ satisfies 
    \[ K(f)\leq 1 + O(\varepsilon), \] 
    where the implied constant depends only on $L$. One may also see \cite{PT2012} of Papadopoulos-Th{\'e}ret and \cite{PY2015} of Papadopoulos-Yamada for related topics. 
\end{remark*}

To prove Theorem \ref{thm:stability}, we will apply the following topological lemma. 
\begin{lemma}\label{lem:largepants}
    Let $Y=Y_{g,n}$ be a bordered hyperbolic surface with genus $g\geq1$, and let $\{\gamma_i\}_{1\leq i\leq n}$ be its boundary components. Then there exists a pants decomposition of $Y$ such that 
    \begin{enumerate}
        \item each pair of pants contains at most one boundary component. 
        \item For each $1\leq  i \leq n$, if $P(\alpha,\beta,\gamma_i)$ is the pair of pants containing $\gamma_i$, then we have $\min\{\length(\alpha), \length(\beta)\} \geq 2\arcsinh1$. 
    \end{enumerate}
\end{lemma}
\begin{proof}
    Since $g\geq1$, there exist at least two non-separating simple closed geodesics in $Y$ that intersect each other at one point which both cut $Y$ into a surface of genus $g-1$ with $n+2$ boundary geodesics. It follows from the standard Collar Lemma that at least one of them must have length $\geq 2\arcsinh 1$, which is denoted by $\eta$. Next we take a pants decomposition $$P=\{\alpha_0 = \eta,\alpha_1,\alpha_2,\cdots,\alpha_{n+1},\cdots\}$$ of $Y$ as illustrated in Figure \ref{fig:decomposition}. In particular, if $g=1$, then $\alpha_n=\eta$ and there is no $\alpha_{n+1}$. 
    \begin{figure}
        \centering
        \includegraphics[scale = 0.34]{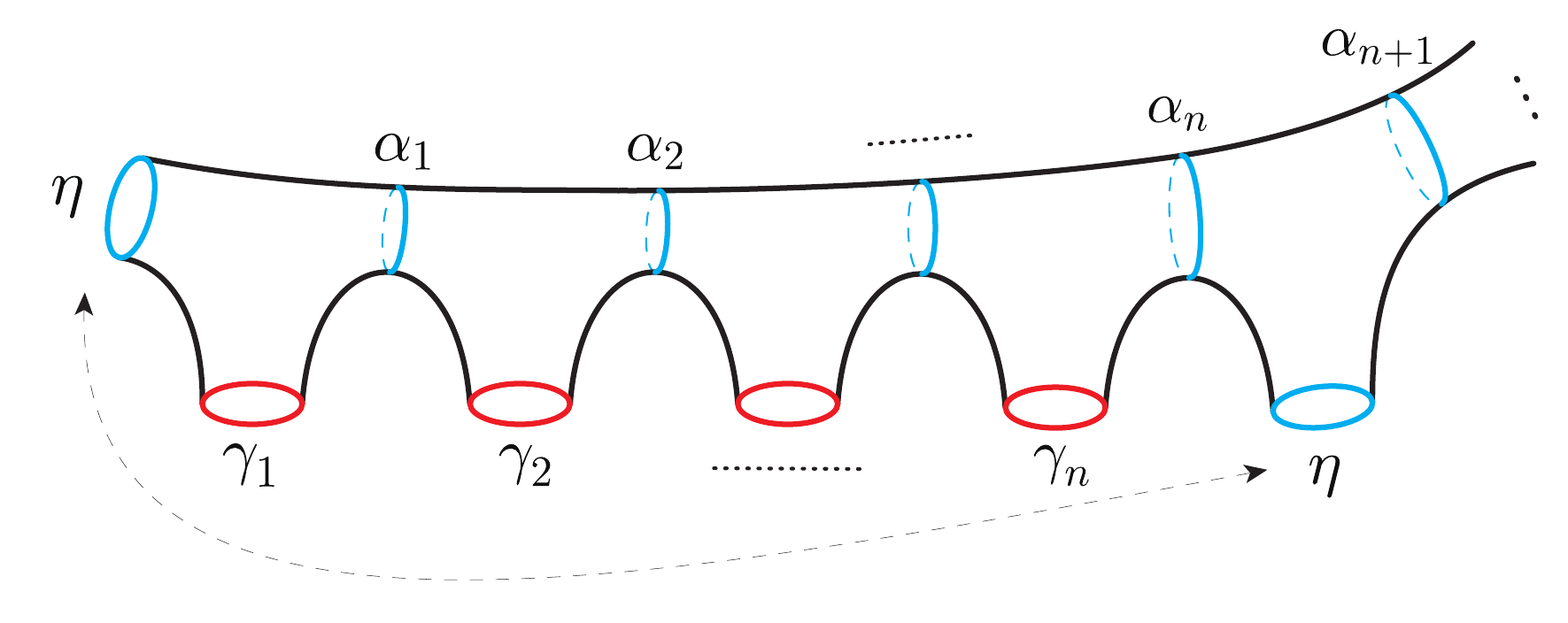}
        \caption{One special pants decomposition of $Y$.}
        \label{fig:decomposition}
    \end{figure}
    If $\length(\alpha_i) < 2\arcsinh1$ for some $1\leq i\leq n$, then one may apply the Whitehead move as illustrated in Figure \ref{fig:WhiteheadMove} to $\alpha_i$ to get another simple closed geodesic $\widetilde{\alpha}_i$ intersecting with $\alpha_i$. 
    \begin{figure}
        \centering
        \includegraphics[scale=0.25]{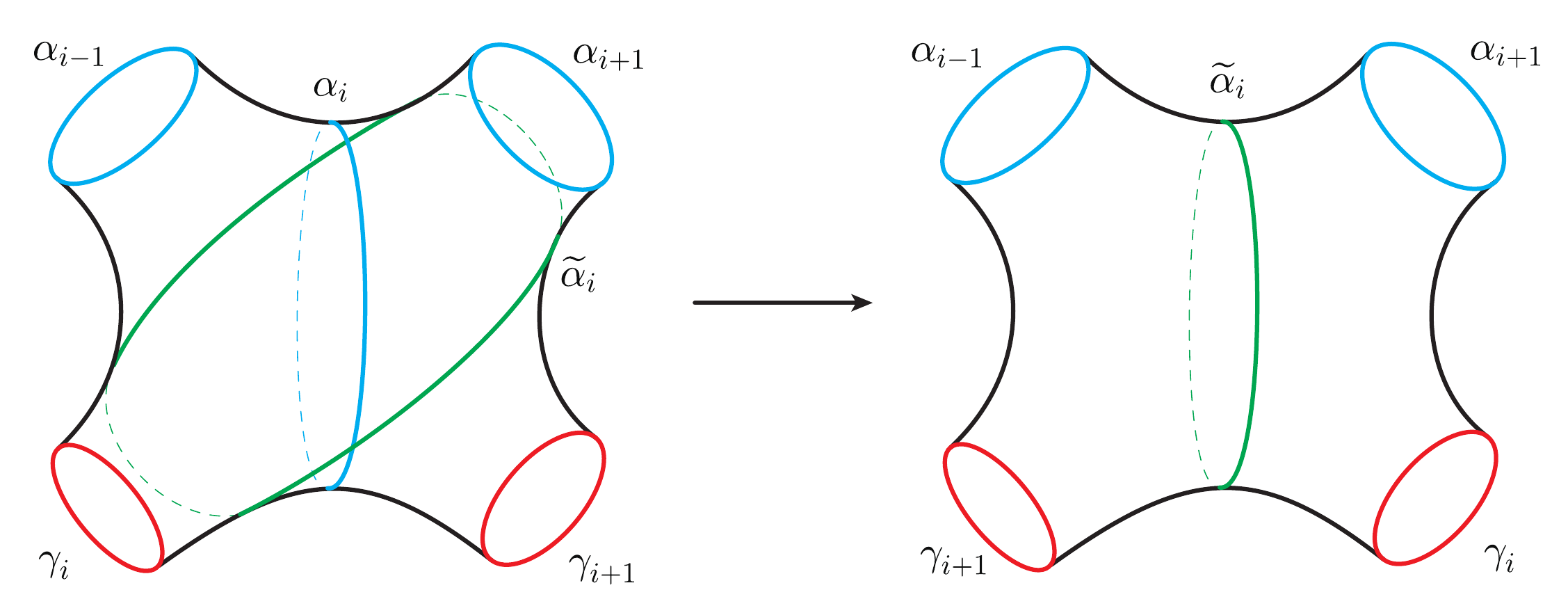}
        \caption{One Whitehead move with the property that each new pair of pants still has at most one boundary component of $Y$.}
        \label{fig:WhiteheadMove}
    \end{figure} 
    Now we replace $\alpha_i$ by $\widetilde{\alpha}_i$. The curve system $\left(\cup_{\gamma \in P\setminus \alpha_i}\gamma\right)  \cup \widetilde{\alpha}_i$ still forms a pants decomposition of $Y$ satisfying Part (1) of the lemma. From the standard Collar Lemma we have that $\length(\tilde{\alpha}_i)\geq 2 \arcsinh 1$. The conclusion then follows from the induction on $n$. 
\end{proof}

\underline{Construction of $Y'_{g,n}$}. 
Let $Y = Y_{g.n}$. Since $g\geq1$, by Lemma \ref{lem:largepants}, we can find a pants decomposition of $Y$ such that 
\begin{equation*}
    Y = \bigcup_{i=1}^{2g-2+n} P_i, \textrm{\quad and\quad} \gamma_i\subset P_i,\ \Forall 1\leq i\leq n. 
\end{equation*}
For each $1\leq i\leq n$, replace the pair of pants $P_i=P_i(\alpha_i, \beta_i, \gamma_i)$ by the pair of pants $P_i(\alpha_i, \beta_i, \gamma'_i)$, where $\ell(\gamma'_i) = (1+\delta_i)\cdot\ell(\gamma_i)$. Then the bordered compact hyperbolic $Y_{g.n}$ is obtained by gluing back the pairs of pants 
$$\{P_1(\alpha_1, \beta_1, \gamma'_1),\cdots, P_n(\alpha_n, \beta_n, \gamma'_n),P_{n+1},\cdots, P_{2g-2+n}\}$$ 
in an obvious way (see Figure \ref{fig:Y'} for illustration), without changing the twist parameters along the simple closed geodesics other than $\{\alpha_i,\beta_i\}_{i=1}^n$ in the pants decomposition. 
\begin{figure}
    \centering
    \includegraphics[scale=0.32]{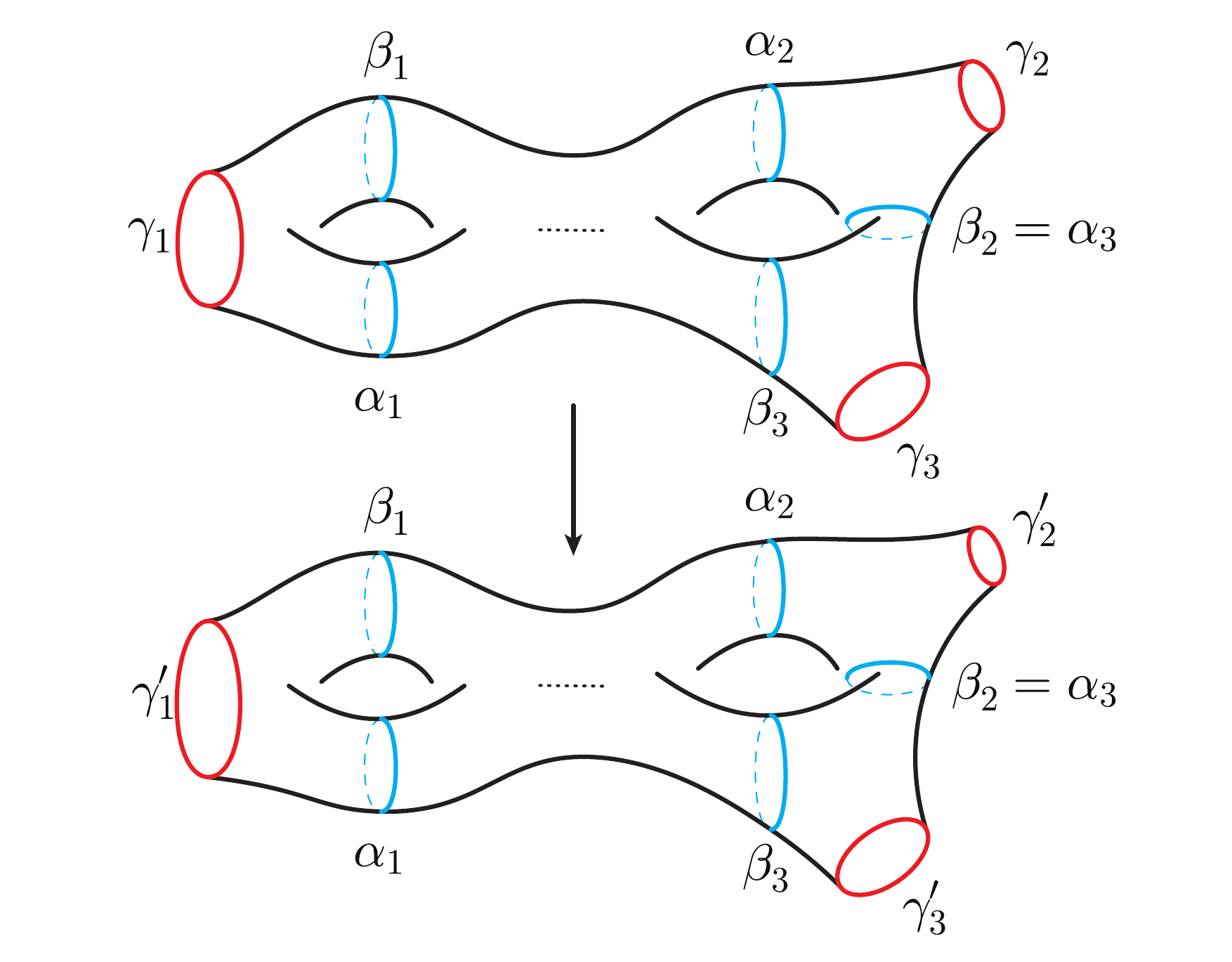}
    \caption{The construction of $Y'$.}
    \label{fig:Y'}
\end{figure}

\begin{remark*}
    Similar constructions had also been studied by Parlier \cite{Parlier2005} and Wu-Xue \cite{WX2022} for other interests. 
\end{remark*}

By an abuse of notation, define a map $\mu:Y\to Y'=Y'_{g,n}$ by 
\begin{equation}\label{eqn:mu_of_Y}
    \mu\,\big|_{P_i} = 
    \begin{cases}
        \,\mu:P_i\to P'_i, &1\leq i\leq n,\\
        \mathrm{id}:P_i\to P_i, &n+1\leq i, 
    \end{cases}
\end{equation}
where $\mathrm{id}$ denotes the identity map, and $\mu:P_i\to P'_i$ is given by Lemma \ref{prop:mu_pants}. Since $\mu|_{\alpha_i}$ and $\mu|_{\beta_i}$ are identity maps for each $i\in[1,n]$,  it follows that $\mu:Y\to Y'$ is a piecewise smooth homeomorphism. 
Denote by $ds_Y^2$ and $ds_{Y'}^2$ the hyperbolic metrics of $Y$ and $Y'$, respectively, then at the points where $\mu$ is smooth, by Part (2) of Proposition \ref{prop:mu_pants} we have 
\begin{proposition}\label{prop:comparison_metric_Y}
    For $\delta$ sufficiently small, we have 
    \[ \big(1+O(\sqrt{\delta})\big) \cdot ds_{Y}^2 \geq \mu^*(ds_{Y'}^2) \geq \big(1-O(\sqrt{\delta})\big) \cdot ds_{Y}^2, \] 
    where the implied constants depend only on $\ell$. 
\end{proposition}
As direct corollaries of Proposition \ref{prop:comparison_metric_Y}, we have 
\begin{corollary}\label{cor:lipschitz}
    For $\delta$ sufficiently small, $\mu:Y\to Y'$ is a bi-Lipschitz map with Lipschitz constant $K_\mu$ satisfying
    \[ \abs{K_\mu - 1} = O(\sqrt{\delta}), \] 
    where the implied constant depends only on $\ell$. 
\end{corollary}

\begin{corollary}\label{cor:comparison_volume_Y}
    For $\delta$ sufficiently small, we have 
    \[ \big(1+O(\sqrt{\delta})\big) \cdot \dvol_{Y} \geq \mu^*(\dvol_{Y'}) \geq \big(1-O(\sqrt{\delta})\big) \cdot \dvol_{Y}, \] 
    where the implied constants depend only on $\ell$. 
\end{corollary}

\underline{Comparison between Neumann eigenvalues}. 
In what follows, we compare $\sigma_1(Y)$ and $\sigma_1(Y')$. Let $f$ be a first Neumann eigenfunction of $Y'$, then $f\circ\mu$ is a test function in the Sobolev space $H^{1,2}(Y)$. Similarly to the proof of Theorem \ref{thm:comparison_eigenvalue}, we will compare the Rayleigh quotients of $f$ and $f\circ \mu$. The proof is divided into three lemmas: 

\begin{lemma}\label{lem:stab2}
    We have \[ \int_Y (f\circ \mu)^2 \dvol_{Y} \geq \big(1-O(\sqrt{\delta})\big) \cdot \int_{Y'} f^2 \dvol_{Y'}. \] 
\end{lemma}
\begin{proof}
    Since $\mu$ is Lipschitz, we can apply the change of variables formula to obtain
    \begin{equation*}
        \int_{Y'} f^2 \dvol_{Y'} = \int_Y (f\circ \mu)^2 \,\mu^*(\dvol_{Y'}),
    \end{equation*}
    thus combining with Corollary \ref{cor:comparison_volume_Y} the conclusion follows. 
\end{proof}

\begin{lemma}\label{lem:stab3}
    We have \[ \frac{1}{\area(Y)}\Big(\int_Y f\circ \mu \dvol_{Y}\Big)^2 \leq O(\sqrt{\delta}) \cdot \int_{Y'} f^2 \dvol_{Y'}. \] 
\end{lemma}
\begin{proof}
    Since $f$ is a first Neumann eigenfunction of $Y'$, and $\mu$ is Lipschitz, we have 
    \begin{equation*}
        0 = \int_{Y'} f \dvol_{Y'} = \int_Y (f\circ \mu) \,\mu^*(\dvol_{Y'}). 
    \end{equation*}
    By Corollary \ref{cor:comparison_volume_Y} we have 
    \begin{equation}\label{eqn:stab3-1}
        \begin{aligned}
            \abs{\int_Y f\circ \mu \dvol_{Y}}
            &  =  \abs{\int_Y f\circ \mu \dvol_{Y} - \int_Y (f\circ \mu) \,\mu^*(\dvol_{Y'})}\\ 
            &\leq O(\sqrt{\delta}) \cdot \int_{Y} \abs{f\circ \mu} \dvol_{Y}.
        \end{aligned}
    \end{equation}
    By Cauchy's inequality, 
    \begin{equation}\label{eqn:stab3-2}
        \int_{Y} \abs{f\circ \mu} \dvol_{Y} \leq \sqrt{\area(Y)}\sqrt{\int_{Y} (f \circ \mu)^2 \dvol_{Y}}, 
    \end{equation}
    again by Corollary \ref{cor:comparison_volume_Y} we have 
    \begin{equation}\label{eqn:stab3-3}
            \begin{aligned}
        \int_Y (f\circ \mu)^2 \dvol_{Y} &\leq \big(1+O(\sqrt{\delta})\big) \cdot \int_Y (f\circ \mu)^2 \mu^*(\dvol_{Y'})\\
        &= \big(1+O(\sqrt{\delta})\big)\cdot \int_{Y'} f^2 \dvol_{Y'}.
            \end{aligned}
    \end{equation}
    Combining \eqref{eqn:stab3-1}, \eqref{eqn:stab3-2}, and \eqref{eqn:stab3-3} the conclusion follows. 
\end{proof}

\begin{lemma}\label{lem:stab1}
    We have 
    \[ \int_{Y'}\abs{\nabla_{Y'} f}^2_{Y'} \dvol_{Y'} \geq \big(1-O(\sqrt{\delta})\big) \cdot \int_Y\abs{\nabla_{Y} (f\circ \mu)}^2_{Y} \dvol_{Y}. \] 
\end{lemma}
\begin{proof}
    Since $\mu$ is Lipschitz, by the change of variables formula we have 
    \begin{equation}\label{eqn:stab1-1}
        \begin{aligned}
            \int_{Y'}\abs{\nabla_{Y'} f}^2_{Y'} \dvol_{Y'} 
            & = \int_{Y'}\abs{\dif f}^2_{Y'} \dvol_{Y'}\\ 
            & = \int_{Y} \big(\abs{\dif f}^2_{Y'} \circ \mu \big) \mu^*(\dvol_{Y'}).
        \end{aligned}
    \end{equation}
    Moreover, the Jacobian matrix $\jac(\mu)$ of $\mu$ is a.e. well-defined, hence we have 
    \begin{equation}\label{eqn:stab1-2}
        \begin{aligned}
            \int_Y\abs{\nabla_{Y} (f\circ \mu)}^2_{Y} \dvol_{Y}
            & = \int_Y\abs{\dif (f\circ \mu)}^2_{Y} \dvol_{Y}\\
            & = \int_Y\abs{(\dif f\circ \mu) \cdot \jac(\mu)}^2_{Y} \dvol_{Y}. 
        \end{aligned}
    \end{equation}
    Recall that Corollary \ref{cor:comparison_volume_Y} tells 
    \begin{equation*}
        \mu^*(\dvol_{Y'}) \geq  \big(1-O(\sqrt{\delta})\big) \cdot \dvol_{Y}, 
    \end{equation*}
    thus by \eqref{eqn:stab1-1} and \eqref{eqn:stab1-2}, to prove the lemma it suffices to prove that 
    \begin{equation}\label{eqn:stab1-3}
        \abs{\dif f}^2_{Y'} \circ \mu \geq \big(1-O(\sqrt{\delta})\big) \cdot \abs{(\dif f\circ \mu) \cdot \jac(\mu)}^2_{Y}. 
    \end{equation}
    Suppose in local coordinates $z$ and $w$ the metrics $ds_Y^2$ and $ds_{Y'}^2$ are written as 
    \[ ds_Y^2 = h_{i,j}\dif z^i\dif z^j \textrm{\quad and\quad} ds_{Y'}^2 = h'_{i,j}\dif w^i\dif w^j \] 
    respectively, where $1\leq i,j\leq 2$, then 
    \begin{equation*}
        \begin{aligned}
            \abs{\dif f}^2_{Y'} \circ \mu &= \left( \pa_{1}f, \pa_{2}f \right) \cdot \left(h'_{i,j}(w(z))\right)^{-1} \cdot \left( \pa_{1}f, \pa_{2}f \right)^T,\\ 
            \abs{(\dif f\circ \mu) \cdot \jac(\mu)}^2_{Y} &= \left( \pa_{1}f, \pa_{2}f \right) \cdot \jac(\mu) \left(h_{i,j}(z)\right)^{-1} \jac(\mu)^T \cdot \left( \pa_{1}f, \pa_{2}f \right)^T, 
        \end{aligned}
    \end{equation*}
    where $\jac(\mu)^T$ is the transpose of $\jac(\mu)$. Thus to prove \eqref{eqn:stab1-3} it suffices to prove that, as positive definite matrices, we have 
    \begin{equation}\label{eqn:matrix_metric}
        \left(h'_{i,j}(w(z))\right)^{-1} \geq \big(1-O(\sqrt{\delta})\big) \cdot \jac(\mu) \left(h_{i,j}(z)\right)^{-1} \jac(\mu)^T. 
    \end{equation} 
    By linear algebra we know that, if $A$ and $B$ are two positive definite matrices, then 
    \[ A\geq B \iff B^{-1}\geq A^{-1}. \] 
    Thus \eqref{eqn:matrix_metric} is equivalent to 
    \[ \left(h_{i,j}(z)\right) \geq \big(1-O(\sqrt{\delta})\big) \cdot \jac(\mu)^T \left(h'_{i,j}(w(z))\right) \jac(\mu), \] 
    that is equivalent to saying
    \[ \big(1+O(\sqrt{\delta})\big) \cdot ds_{Y}^2 \geq \mu^*(ds_{Y'}^2). \] 
    This is guaranteed by Proposition \ref{prop:comparison_metric_Y}. So the conclusion follows. 
\end{proof}

Now we are ready to prove Theorem \ref{thm:stability}.
\begin{proof}[Proof of Theorem \ref{thm:stability}]
    Suppose $\delta$ is sufficiently small (\eg $900e^{10\ell}\sqrt{\delta}\leq1$) such that Proposition \ref{prop:mu_pants} holds. From the construction of $Y'$, it is clear that for all $1\leq i\leq n$, $$\ell(\gamma'_i) = \ell(\gamma_i)+\delta_i.$$ The second statement clearly follows from Corollary \ref{cor:lipschitz}. By Lemma \ref{lem:stab2}, \ref{lem:stab3} and \ref{lem:stab1}, we have 
    \begin{equation*}
        \begin{aligned}
            \sigma_1(Y') 
            &  =  
            \frac{\int_{Y'}\abs{\nabla_{Y'} f}^2 \dvol_{Y'}}{\int_{Y'} f^2 \dvol_{Y'}}\\ 
            &\geq \big(1-O(\sqrt{\delta})\big) \cdot \frac{\int_Y\abs{\nabla_{Y} (f\circ \mu)}^2 \dvol_{Y}}{\int_Y (f\circ \mu)^2 \dvol_{Y} - \frac{1}{\area(Y)}\left(\int_Y f\circ \mu \dvol_{Y}\right)^2}\\ 
            &\geq \big(1-O(\sqrt{\delta})\big) \cdot \sigma_1(Y). 
        \end{aligned}
    \end{equation*}
    The inequality in the inverse direction $$\sigma_1(Y)\geq\big(1-O(\sqrt{\delta})\big) \cdot \sigma_1(Y')$$ follows by switching $Y$ and $Y'$. 
\end{proof}

\subsection{Maps between pairs of pants}\label{subs-l1}
In this subsection we prove Proposition \ref{prop:mu_pants}. Following \cite{Buser1992}, a side of some polygon and its hyperbolic length will always be denoted by the same letter. 

We will use the pair of Fermi coordinates, which one may see in \cite[Page 4]{Buser1992} for more details. We briefly introduce it here. Let $\eta:\R\to \H$ be a geodesic line with parameter $u$, and let $o=\eta(0)$. Since $\H$ is separated by $\eta$ into two half-planes, each point $p\in\H$ has a signed distance $v$ from $p$ to $\eta$, positive on one side and negative on the other. Let $p_0$ be the projection of $p$ to $\eta$, and let $u$ be the signed distance from $p_0$ to $o$. The pair $(u,v)$ is called the Fermi coordinates with respect to $\eta$, in which the hyperbolic metric is expressed by 
\begin{equation}\label{eqn:metric_fermi}
    \cosh^2\!v \dif u^2 + \dif v^2. 
\end{equation}

We begin our proof with constructing maps between right-angled pentagons. 
\begin{proposition}\label{prop:mu_pentagon}
Let $\pent$ (resp., $\pent '$) be a right-angled pentagon with consecutive sides $\alpha$, $b$, $\gamma$, $d$ and $c$ (resp., $\alpha'$, $b'$, $\gamma'$, $d'$ and $c'$), and let $h_0$ (resp., $h'_0$) be the altitude of $\pent$ (resp., $\pent'$) perpendicular to $d$ (resp., $d'$). Set 
\[ \delta_d \df \abs{\frac{\sinh d'}{\sinh d} - 1}. \] 
Assume that $54e^{2\bar h}\sqrt{\delta_d}\leq1$, where $\bar h \df \max\{h_0,h'_0\}$. If $\alpha = \alpha'$, then there exists a piecewise homeomorphism 
\[ \mu:\pent \to \pent ', \] 
which maps the sides of $\pent$ to the corresponding sides of $\pent'$, such that 
\begin{enumerate}
    \item $\eval{\mu}_{\alpha}$ is the identity map. 
    \item At the points where $\mu$ is smooth, we have \[ ( 1 + 54e^{2\bar h}\sqrt{\delta_d} ) \cdot ds_{\pent}^2 \geq \mu^*(ds_{\pent'}^2) \geq ( 1 - 54e^{2\bar h}\sqrt{\delta_d} ) \cdot ds_{\pent}^2 \]where $ds_{\pent}^2$ and $ds_{\pent'}^2$ are standard hyperbolic metrics.
\end{enumerate}
\end{proposition}

\underline{Construction of $\mu$}. 
Let $(u,v)$ be the pair of Fermi coordinates with respect to the geodesic containing $d$, such that the two endpoints of $d$ have coordinates $(0,0)$ and $(d,0)$. For any $u\in[0,d]$, let $h=h(u)$ be the altitude of $\pent$ at $(u,0)$, and let $u_0\in[0,d]$ be such that 
\[ h(u_0) = h_0. \]  
See Figure \ref{fig:MuPentagon} for an illustration. Let $u'$, $v'$, $h'=h'(u')$, and $u'_0$ be defined similarly on $\pent'$. Then we define the map $\mu:(u,v)\mapsto (u',v')$ by 
\begin{equation}\label{eqn:mu_def}
    \frac{\tanh u'}{\tanh d'} = \frac{\tanh u}{\tanh d} \textrm{\quad and\quad} \frac{\tanh v'}{\tanh h'} = \frac{\tanh v}{\tanh h}, 
\end{equation}
One can easily see that $\mu$ is a well-defined piecewise smooth homeomorphism. 
\begin{figure}
    \centering
    \includegraphics[scale=0.36]{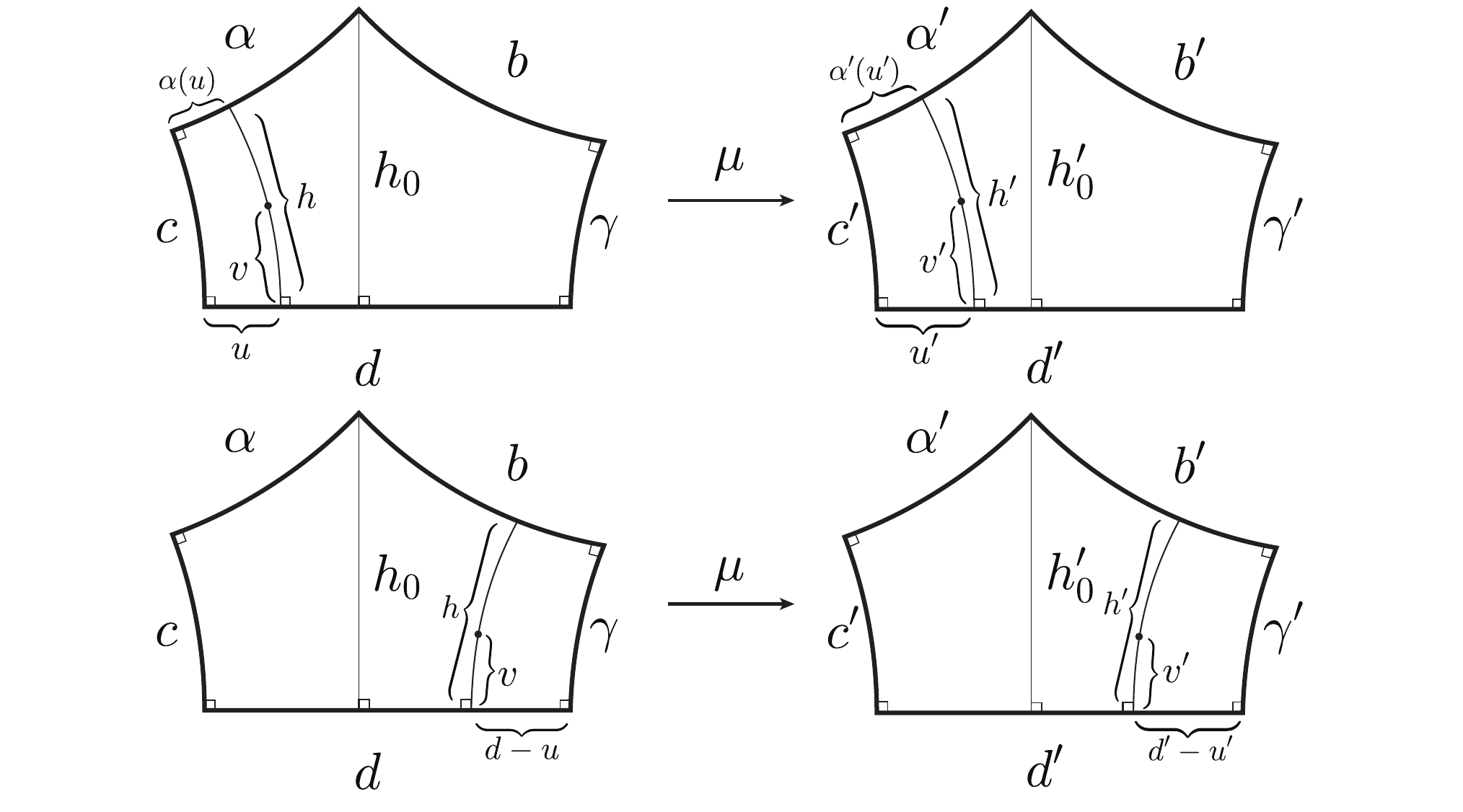}
    \caption{The construction of the map $ \mu:\pent \to \pent '$.}
    \label{fig:MuPentagon}
\end{figure}

\begin{lemma}\label{lem:mu_pf1}
    The restriction $\eval{\mu}_{\alpha}$ is the identity map. 
\end{lemma}
\begin{proof}
    Let $\alpha(u)$ and $\alpha'(u')$ be the segments as shown in Figure \ref{fig:MuPentagon}, then by definition $\mu$ maps $\alpha(u)$ to $\alpha'(u')$. By Lemma \ref{lem:trirectangle} and \ref{lem:pentagon}, in $\pent$ we have 
    \begin{equation}\label{eqn:mu_pf1-1}
        \begin{aligned}
            \tanh \alpha(u) 
            &= \cosh c \cdot \tanh u\\ 
            &= \coth \alpha \cdot \coth d \cdot \tanh u. 
        \end{aligned}
    \end{equation}
    Similarly in $\pent'$ we also have 
    \begin{equation}\label{eqn:mu_pf1-2}
        \tanh \alpha'(u') = \coth \alpha' \cdot \coth d' \cdot \tanh u'. 
    \end{equation}
    It follows from \eqref{eqn:mu_def}, \eqref{eqn:mu_pf1-1} and \eqref{eqn:mu_pf1-2} that 
    \begin{equation*}
        \tanh \alpha(u) \cdot \tanh \alpha = \tanh \alpha'(u') \cdot \tanh \alpha'. 
    \end{equation*}
    Thus if $\alpha = \alpha'$, then $\alpha(u) = \alpha'(u')$, \ie $\eval{\mu}_{\alpha}$ is the identity map. 
\end{proof}

\underline{Comparison between metrics}. Recall that the matrix coefficient of $ds_{\pent}^2$ is 
\begin{equation*}
    A \df 
    \begin{pmatrix}
        \cosh^2\! v & 0\\
        0 & 1
    \end{pmatrix}. 
\end{equation*}
From \eqref{eqn:mu_def} it is clear that $\pa u'/\pa v=0$. A direct computation shows that the matrix coefficient of $\mu^*(ds_{\pent'}^2)$ in the pair of coordinates $(u,v)$ is 
\begin{align*}
    B \df 
    \jac(\mu)^T 
    \begin{pmatrix}
        \cosh^2\! v' & 0\\
        0 & 1
    \end{pmatrix} 
    \jac(\mu) 
    & = 
    \renewcommand*{\arraystretch}{2}
    \begin{pmatrix}
        \Big(\dfrac{\pa u'}{\pa u}\Big)^2\! \cosh^2\! v' + \Big(\dfrac{\pa v'}{\pa u}\Big)^2 & \dfrac{\pa v'}{\pa u}\dfrac{\pa v'}{\pa v}\\ 
        \dfrac{\pa v'}{\pa u}\dfrac{\pa v'}{\pa v} & \Big(\dfrac{\pa v'}{\pa v}\Big)^2
    \end{pmatrix}. 
\end{align*}

\noindent By definition, proving that
\[ (1+\delta_0)\cdot ds_{\pent}^2 \geq \mu^*(ds_{\pent'}^2) \geq (1-\delta_0)\cdot ds_{\pent}^2 \] 
for some $0<\delta_0\leq1$ is equivalent to proving that
\[ (1+\delta_0)A - B \textrm{\quad and\quad} B - (1-\delta_0)A \] 
are non-negative definite matrices, which can be reduced to prove the following inequalities: 
\begin{gather*}
    (1+\delta_0)\cdot \cosh^2\! v > \Big(\dfrac{\pa u'}{\pa u}\Big)^2\!\! \cosh^2\! v' + \Big(\dfrac{\pa v'}{\pa u}\Big)^2,\\
    \bigg( (1+\delta_0)\cosh^2\! v - \Big(\dfrac{\pa u'}{\pa u}\Big)^2\!\! \cosh^2\! v' - \Big(\dfrac{\pa v'}{\pa u}\Big)^2 \bigg) \bigg( (1+\delta_0) - \Big(\dfrac{\pa v'}{\pa v}\Big)^2 \bigg) \geq \Big( \dfrac{\pa v'}{\pa u}\dfrac{\pa v'}{\pa v} \Big)^2\!; 
\end{gather*}
and 
\begin{gather*}
    \Big(\dfrac{\pa u'}{\pa u}\Big)^2\!\! \cosh^2\! v' + \Big(\dfrac{\pa v'}{\pa u}\Big)^2 > (1-\delta_0)\cdot \cosh^2\! v,\\
    \bigg( \Big(\dfrac{\pa u'}{\pa u}\Big)^2\!\! \cosh^2\! v' + \Big(\dfrac{\pa v'}{\pa u}\Big)^2\!\! - (1-\delta_0)\cosh^2\! v \bigg) \bigg( \Big(\dfrac{\pa v'}{\pa v}\Big)^2\!\! - (1-\delta_0) \bigg) \geq \Big( \dfrac{\pa v'}{\pa u}\dfrac{\pa v'}{\pa v} \Big)^2\!. 
\end{gather*}
Rearranging the inequalities above, it suffices to prove that for any $u\in[0,d]$ and any $v\in[0,h(u)]$ the following inequalities hold: 
\begin{align}
    \label{eqn:mu_pf2-1} 1-\delta_0^2/9 \leq &\Big(\dfrac{\pa u'}{\pa u}\Big)^2 \leq 1+\delta_0^2/9,\\
    \label{eqn:mu_pf2-2} 1-\delta_0^2/9 \leq &\Big(\dfrac{\pa v'}{\pa v}\Big)^2 \leq 1+\delta_0^2/9,\\
    \label{eqn:mu_pf2-3} &\Big(\dfrac{\pa v'}{\pa u}\Big)^2 \leq \delta_0^2/9,\\
    \label{eqn:mu_pf2-4} \textrm{and\quad} 1-\delta_0^2/9 \leq &\dfrac{\cosh^2\! v'}{\cosh^2\! v} \leq 1+\delta_0^2/9. 
\end{align}
In next subsection we will prove that 
\begin{lemma}\label{lem:mu_pf2}
    The inequalities \eqref{eqn:mu_pf2-1}, \eqref{eqn:mu_pf2-2}, \eqref{eqn:mu_pf2-3} and \eqref{eqn:mu_pf2-4} hold for 
    \[ \delta_0=54e^{2\bar h}\sqrt{\delta_d}. \] 
\end{lemma}

Now we can prove Proposition \ref{prop:mu_pentagon}. 
\begin{proof}[Proof of Proposition \ref{prop:mu_pentagon}]
    The proposition follows from Lemma \ref{lem:mu_pf1} and \ref{lem:mu_pf2}. 
\end{proof}

Next we construct maps between right-angled hexagons and finish the proof of Proposition \ref{prop:mu_pants}. 
\begin{proposition}\label{prop:mu_hexagon}
Let $\hexa$ (resp., $\hexa'$) be a right-angled hexagon with consecutive sides $\alpha$, $c$ $\beta$, $a$, $\gamma$, and $b$ (resp., $\alpha'$, $c'$ $\beta'$, $a'$, $\gamma'$, and $b'$). Suppose that there exists a constant $\ell>0$ such that $\gamma, \gamma' \leq \ell$. Set
\[ \delta_\gamma \df \abs{\frac{\gamma}{\gamma'}-1}, \] 
and assume that $1350e^{5\ell}\sqrt{\delta_\gamma} \leq 1$. If $\alpha = \alpha'$, $\beta = \beta'$ and $\alpha, \beta \geq \arcsinh1$, then there exists a piecewise smooth homeomorphism 
\[ \mu:\hexa \to \hexa', \] 
which maps the sides of $\hexa$ to the corresponding sides of $\hexa'$, such that 
    \begin{enumerate}
        \item $\eval{\mu}_{\alpha}$ and $\eval{\mu}_{\beta}$ are the identity maps. 
        \item At the points where $\mu$ is smooth, we have \[ \big( 1 + O(\sqrt{\delta_\gamma}) \big) \cdot ds_{\hexa}^2 \geq \mu^*(ds_{\hexa'}^2) \geq \big( 1 - O(\sqrt{\delta_\gamma}) \big) \cdot ds_{\hexa}^2, \] where the implied constants depend only on $\ell$. 
    \end{enumerate}
\end{proposition}
\begin{proof}
    Let $d$ be the common perpendicular between $c$ and $\gamma$, then $d$ decomposes $\hexa$ into two right-angled pentagons, $\pent_\alpha$ and $\pent_\beta$. Similarly, on $\hexa'$ we define $d'$ and the two right-angled pentagons separated by $d'$.  Define the map $$\mu:\hexa\to\hexa'$$ such that $\mu$ maps these two pentagons of $\hexa$ to the corresponding pentagons of $\hexa'$, and coincides with the map given by Proposition \ref{prop:mu_pentagon} on each pentagon. See Figure \ref{fig:MuHexagon} for an illustration. It is clear that $\eval{\mu}_{\alpha}$ and $\eval{\mu}_{\beta}$ are the identity maps.
    \begin{figure}
        \centering
        \includegraphics[scale=0.38]{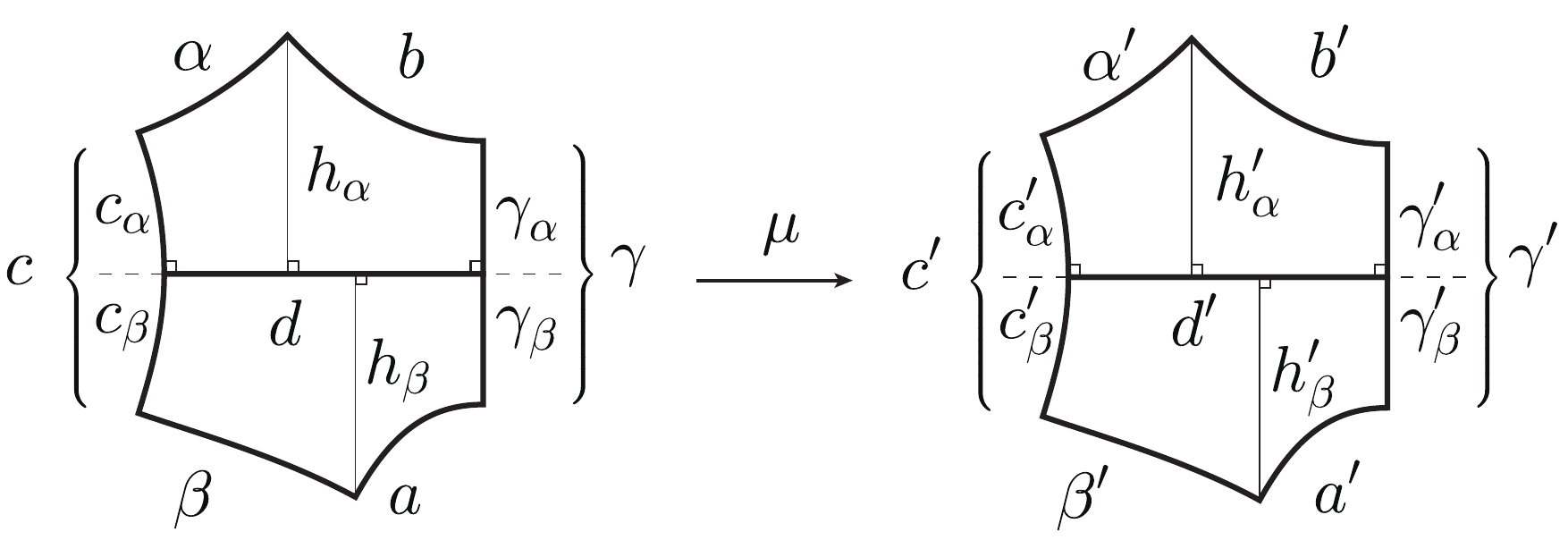}
        \caption{An illustration for the map  $\mu:\hexa \to \hexa'$.}
        \label{fig:MuHexagon}
    \end{figure}
                
    By Lemma \ref{lem:hexagon'} we have 
    \begin{equation*}
        \frac{\sinh^2\! d'}{\sinh^2\! d} = \frac{\sinh^2\! \gamma}{\sinh^2\! \gamma'}\frac{\cosh^2\! \alpha + \cosh^2\! \beta + 2\cosh\alpha\cosh\beta\cosh\gamma'}{\cosh^2\! \alpha + \cosh^2\! \beta + 2\cosh\alpha\cosh\beta\cosh\gamma}, 
    \end{equation*}
    it follows that 
    \begin{equation*}
        \min\left\{1,\, \frac{\sinh^2\! \gamma}{\sinh^2\! \gamma'}\right\} \leq \frac{\sinh^2\! d'}{\sinh^2\! d} \leq \max\left\{1,\, \frac{\sinh^2\! \gamma}{\sinh^2\! \gamma'}\right\}. 
    \end{equation*}
    Recall that $\gamma,\gamma'\leq \ell$ and $\sinh(\gamma')\geq \gamma'$, thus by the Mean Value Theorem we have 
    \begin{equation}\label{eqn:hexa1}
        \begin{aligned}
            \delta_d 
            &= \abs{\frac{\sinh d'}{\sinh d}-1}
            \leq \abs{\frac{\sinh \gamma}{\sinh \gamma'}-1}\\
            &\leq \frac{\cosh \ell}{\gamma'}\cdot\abs{\gamma-\gamma'} 
              =  \cosh \ell\cdot\delta_\gamma. 
        \end{aligned}
    \end{equation}
    Let $h_\alpha$ be the altitude of $\pent_\alpha$ perpendicular to $d$, and let $h_\beta$, $h'_\alpha$ and $h'_\beta$ be defined similarly. 
    If $\alpha \geq \arcsinh 1$, then by Lemma \ref{lem:hexagon_altitude} we have 
    \begin{equation*}
        \sinh h_\alpha \leq \frac{\cosh^2\! \gamma}{\tanh \alpha} \leq 2\cosh^2\! \gamma \leq 2e^{2\ell}, 
    \end{equation*}
    it follows that 
    \begin{equation}\label{eqn:hexa2}
        e^{h_\alpha} = \sinh h_\alpha + \cosh h_\alpha \leq 5e^{2\ell}.
    \end{equation}
    From \eqref{eqn:hexa1}, \eqref{eqn:hexa2} and our assumption we deduce that 
    \begin{equation}\label{eqn:hexa3}
        54e^{2h_\alpha}\sqrt{\delta_d} \leq 1350e^{5\ell}\sqrt{\delta_\gamma} \leq 1. 
    \end{equation}
    Similarly, the same upper bound \eqref{eqn:hexa3} also holds for $h_\beta$, $h'_\alpha$ and $h'_\beta$. The conclusion then follows from Proposition \ref{prop:mu_pentagon}. 
\end{proof}

Now we can prove Proposition \ref{prop:mu_pants}. 
\begin{proof}[Proof of Proposition \ref{prop:mu_pants}]
    Each pair of pants $P$ can be decomposed by the common perpendiculars between the three pairs of boundary geodesics into two isometric right-angled hexagons. Define the map $\mu$ from $P$ to $P'$ such that it maps the two hexagons of $P$ to the corresponding ones of $P'$, and coincides with the map given by Proposition \ref{prop:mu_hexagon} on each hexagon. It is clear from Proposition \ref{prop:mu_hexagon} that this map satisfies all the desired properties. 
\end{proof}

\subsection{Proof of Lemma \ref{lem:mu_pf2}}\label{subs-l2}
We will frequently use the following simple lemma whose proof follows from a direct computation that we omit here.
\begin{lemma}\label{lem:mu_inequality}
    For any positive constant $x$ and $x'$, we have 
    \begin{align*}
        \max\Big\{\frac{\tanh x'}{\tanh x}, \frac{\cosh x'}{\cosh x}\Big\} &\leq \max\Big\{1,\, \frac{\sinh x'}{\sinh x}\Big\},\\ 
        \min\Big\{\frac{\tanh x'}{\tanh x}, \frac{\cosh x'}{\cosh x}\Big\} &\geq \min\Big\{1,\, \frac{\sinh x'}{\sinh x}\Big\}. 
    \end{align*}
\end{lemma}

We begin our proof of Lemma \ref{lem:mu_pf2} with some estimates. Recall that 
\[ \delta_d = \abs{\frac{\sinh d'}{\sinh d} - 1}. \] 
\begin{lemma}\label{lem:mu_lem1}
    We have 
    \begin{align}
        \label{eqn:mu_lem1-1} 1-\delta_d \leq \tanh d' &/ \tanh d \leq 1+\delta_d,\\
        \label{eqn:mu_lem1-2} 1-\delta_d \leq \cosh d' &/ \cosh d \leq 1+\delta_d. 
    \end{align}
\end{lemma}
\begin{proof}
    It follows directly from the definition of $\delta_d$ and Lemma \ref{lem:mu_inequality}. 
\end{proof}

Recall that the map $(u,v)\mapsto(u',v')$ is defined by
\[ \frac{\tanh u'}{\tanh d'} = \frac{\tanh u}{\tanh d} \textrm{\quad and\quad} \frac{\tanh v'}{\tanh h'} = \frac{\tanh v}{\tanh h}. \] 
\begin{lemma}\label{lem:mu_lem2}
    Suppose $\delta_d\leq 1$. Then for any $u\in[0,d]$ we have 
    \begin{align}
        \label{eqn:mu_lem2-1} 1-\delta_d \leq \cosh u' &/ \cosh u \leq 1+\delta_d,\\
        \label{eqn:mu_lem2-2} 1-2\delta_d \leq \tanh (d'-u') &/ \tanh (d-u) \leq 1+3\delta_d,\\
        \label{eqn:mu_lem2-3} 1-\delta_d \leq \cosh (d'-u') &/ \cosh (d-u) \leq 1+\delta_d. 
    \end{align}
\end{lemma}
\begin{proof}
    \underline{Proof of \eqref{eqn:mu_lem2-1}}. 
    By definition we have 
    \[ \frac{\cosh u'}{\cosh u} = \frac{1}{\cosh u(1-\tanh^2\! u')^{1/2}} = \bigg(\frac{1}{\cosh^2\!u - \frac{\tanh^2\! d'}{\tanh^2\! d}\sinh^2\! u}\bigg)^{1/2}, \] 
    hence it is a monotonic function with respect to $u\in[0,d]$. It follows that 
    \begin{equation*}
        \min\Big\{1,\, \frac{\cosh d'}{\cosh d}\Big\} \leq \frac{\cosh u'}{\cosh u} \leq \max\Big\{1,\, \frac{\cosh d'}{\cosh d}\Big\}, 
    \end{equation*}
    thus \eqref{eqn:mu_lem2-1} follows from \eqref{eqn:mu_lem1-2}. 

    \underline{Proof of \eqref{eqn:mu_lem2-2}}. 
    By definition we have 
    \begin{equation}\label{eqn:mu_lem2-4}
        \begin{aligned}
            \frac{\tanh (d'-u')}{\tanh (d-u)} 
            & = \frac{\tanh d' - \tanh u'}{\tanh d - \tanh u}\cdot \frac{1-\tanh d\tanh u}{1-\tanh d'\tanh u'}\\ 
            & = \frac{\tanh d'}{\tanh d}\cdot \frac{1-\tanh d\tanh u}{1-\frac{\tanh^2\! d'}{\tanh d}\tanh u}, 
        \end{aligned}
    \end{equation}
    hence it is also a monotonic function with respect to $u\in[0,d]$. 
    It follows from \eqref{eqn:mu_lem2-4} that 
    \begin{equation}\label{eqn:mu_lem2-5}
        \begin{aligned}
            \frac{\tanh (d'-u')}{\tanh (d-u)} &\leq \max\Big\{\frac{\tanh d'}{\tanh d},\, \frac{\sinh d'\cosh d'}{\sinh d\cosh d}\Big\},\\
            \frac{\tanh (d'-u')}{\tanh (d-u)} &\geq \min\Big\{\frac{\tanh d'}{\tanh d},\, \frac{\sinh d'\cosh d'}{\sinh d\cosh d}\Big\}.
        \end{aligned}
    \end{equation}
    Thus, \eqref{eqn:mu_lem2-2} follows from Lemma \ref{lem:mu_lem1} and our assumption that $\delta_d\leq1$. 

    \underline{Proof of \eqref{eqn:mu_lem2-3}}. 
    From \eqref{eqn:mu_lem2-5} we also find that $d'-u'\leq d-u$ if and only if $d'\leq d$. Meanwhile, by definition of $u'$ we also have that $u'\leq u$ if and only if $d'\leq d$. To estimate $\cosh(d'-u')/\cosh(d-u)$, note that 
    \begin{equation}\label{eqn:mu_lem2-6}
        \frac{\sinh d'}{\sinh d} = \frac{\sinh u'\cosh (d'-u') + \cosh u'\sinh (d'-u')}{\sinh u\cosh (d-u) + \cosh u\sinh (d-u)}, 
    \end{equation}
    if $d'\leq d$, then $u'\leq u$ and $d'-u'\leq d-u$, hence by \eqref{eqn:mu_lem2-6} and Lemma \ref{lem:mu_inequality} we have 
    \begin{equation*}
        1-\delta_d \leq \frac{\sinh d'}{\sinh d} \leq \frac{\cosh u'\cosh (d'-u')}{\cosh u\cosh (d-u)} \leq \frac{\cosh (d'-u')}{\cosh (d-u)} \leq 1, 
    \end{equation*}
    if $d'\geq d$, then $u'\geq u$ and $d'-u'\geq d-u$, hence by \eqref{eqn:mu_lem2-6} and Lemma \ref{lem:mu_inequality} we have 
    \begin{equation*}
        1+\delta_d \geq \frac{\sinh d'}{\sinh d} \geq \frac{\cosh u'\cosh (d'-u')}{\cosh u\cosh (d-u)} \geq \frac{\cosh (d'-u')}{\cosh (d-u)} \geq 1. 
    \end{equation*}
    
    The proof is complete. 
\end{proof}

\begin{lemma}\label{lem:mu_lem3}
    Suppose $\delta_d\leq 1/2$. Then we have 
    \begin{align}
        \label{eqn:mu_lem3-1} 1-\delta_d \leq \sinh c' &/ \sinh c \leq 1+2\delta_d,\\
        \label{eqn:mu_lem3-2} 1-\delta_d \leq \sinh \gamma' &/ \sinh \gamma \leq 1+2\delta_d. 
    \end{align}
\end{lemma}
\begin{proof}
    \underline{Proof of \eqref{eqn:mu_lem3-1}}. 
    Since $\alpha'=\alpha$, by Lemma \ref{lem:pentagon'} we have 
    \begin{equation*}
        \frac{\sinh c'}{\sinh c} =  \bigg( \frac{1 / \sinh^2\! d' + 1/(\sinh^2\! \alpha\tanh^2\! d')}{1 / \sinh^2\! d + 1/(\sinh^2\! \alpha\tanh^2\! d)} \bigg)^{1/2}, 
    \end{equation*}
    thus by Lemma \ref{lem:mu_inequality} and our assumption that $\delta_d\leq1/2$ we have 
    \begin{equation*}
        \begin{aligned}
            \frac{\sinh c'}{\sinh c} &\leq \max\Big\{\frac{\sinh d}{\sinh d'},\, \frac{\tanh d}{\tanh d'}\Big\} \leq \max\Big\{1,\,\frac{\sinh d}{\sinh d'}\Big\} \leq 1+2\delta_d,\\ 
            \frac{\sinh c'}{\sinh c} &\geq \min\Big\{\frac{\sinh d}{\sinh d'},\, \frac{\tanh d}{\tanh d'}\Big\} \geq \min\Big\{1,\,\frac{\sinh d}{\sinh d'}\Big\} \geq 1-\delta_d. 
        \end{aligned}
    \end{equation*}
    Then \eqref{eqn:mu_lem3-1} follows. 

    \underline{Proof of \eqref{eqn:mu_lem3-2}}. 
    By Lemma \ref{lem:pentagon} we have 
    \begin{equation*}
        1-\delta_d \leq \frac{\sinh\gamma'}{\sinh\gamma} = \frac{\cosh \alpha/\sinh d'}{\cosh \alpha/\sinh d} = \frac{\sinh d}{\sinh d'} \leq 1+2\delta_d, 
    \end{equation*}
    thus \eqref{eqn:mu_lem3-2} also follows. The proof is complete. 
\end{proof}

\begin{lemma}\label{lem:mu_lem4}
    Suppose $\delta_d\leq1/2$. Then for any $u\in[0,d]$ we have 
    \begin{equation}\label{eqn:mu_lem4-1}
        1-2\delta_d \leq \tanh h'/\tanh h \leq 1+4\delta_d,\\
    \end{equation}
\end{lemma}
\begin{proof}
    If $u\in[0,u_0]$, then by Lemma \ref{lem:trirectangle}, Lemma \ref{lem:mu_pf1} and \eqref{eqn:mu_lem3-1} we have 
    \begin{equation}\label{eqn:mu_lem4-2}
        1-\delta_d \leq \frac{\sinh h'}{\sinh h} = \frac{\cosh \alpha'(u')\sinh c'}{\cosh \alpha(u)\sinh c} = \frac{\sinh c'}{\sinh c} \leq 1+2\delta_d 
    \end{equation}
    which together with Lemma \ref{lem:mu_inequality} implies that whenever $u\in[0,u_0]$,
     \begin{equation}\label{eqn:mu_lem4-2-1}
        1-\delta_d \leq \frac{\tanh h'}{\tanh h} \leq 1+2\delta_d
    \end{equation}
    If $u\in[u_0,d]$, then by Lemma \ref{lem:trirectangle}, Lemma \ref{lem:mu_inequality}, and \eqref{eqn:mu_lem2-3}, \eqref{eqn:mu_lem3-2} we have 
    \begin{equation}\label{eqn:mu_lem4-3}
        1-2\delta_d \leq \frac{\tanh h'}{\tanh h} = \frac{\cosh(d'-u')\tanh \gamma'}{\cosh(d-u)\tanh \gamma} \leq 1+4\delta_d. 
    \end{equation}
    The conclusion then follows from \eqref{eqn:mu_lem4-2-1} and \eqref{eqn:mu_lem4-3}. 
\end{proof}

\begin{lemma}\label{lem:mu_lem5}
    Suppose $\delta_d\leq1/5$. Then for any $u\in[0,d]$ we have 
    \begin{equation}\label{eqn:mu_lem5-1}
        1-3\delta_d \leq \cosh h'/\cosh h \leq 1+7\delta_d. 
    \end{equation}
\end{lemma}
\begin{proof}
    If $u\in[0,u_0]$, then \eqref{eqn:mu_lem5-1} follows from Lemma \ref{lem:mu_inequality} and \eqref{eqn:mu_lem4-2}. 
    If $u\in[u_0,d]$, then by Part (1) of Lemma \ref{lem:trirectangle}, we have 
    \begin{equation*}
        \frac{1}{\cosh^2\! h} = 1-\tanh^2\! h = 1-\cosh^2(d-u)\tanh^2\! \gamma. 
    \end{equation*}
    By Part (2) of Lemma \ref{lem:pentagon}, we have 
    \begin{equation*}
        \tanh^2\! \gamma\tanh^2\! c = \frac{1}{\cosh^2\! d}. 
    \end{equation*}
    By Part (2) of Lemma \ref{lem:pentagon'}, we have 
    \begin{equation*}
        \tanh^2\! c = 1-\tanh^2\! \alpha\tanh^2\! d. 
    \end{equation*}
    Combining the three equalities above, we obtain that 
    \begin{equation}\label{eqn:mu_lem5-2}
        \begin{aligned}
            \frac{\tanh^2\! c}{\cosh^2\! h}\frac{1}{\cosh^2\! u}
            & = \frac{1-\tanh^2\! \alpha\tanh^2\! d}{\cosh^2\! u} - \frac{\cosh^2(d-u)}{\cosh^2\! d\cosh^2\! u}\\
            & = (1-\tanh^2\! u)(1-\tanh^2\! \alpha\tanh^2\! d) - (1-\tanh u\tanh d)^2.
        \end{aligned}
    \end{equation}
    Now we set 
    \[ t \df \frac{\tanh u}{\tanh d} = \frac{\tanh u'}{\tanh d'}, \] 
    then by \eqref{eqn:mu_pf1-1} we have 
    \[ \tanh u_0 = \tanh \alpha(u_0)\tanh \alpha\tanh d = \tanh^2\! \alpha \cdot \tanh d, \] 
    which means that if $u=u_0$, then $t=\tanh^2\!\alpha$. Next by taking a substitution of $\tanh u = t\cdot\tanh d$ into the RHS of \eqref{eqn:mu_lem5-2}, we find that it can be written as a function of $t$: 
    \begin{equation*}
        \tanh^2\! d\cdot\Big( (\tanh^2\! \alpha\tanh^2\! d-\tanh^2\! d-1)t^2 + 2t - \tanh^2\! \alpha \Big). 
    \end{equation*}
    Then we define the function $F_{d}(t)$ by 
    \begin{equation*}
        F_{d}(t) \df (\tanh^2\! \alpha\tanh^2\! d-\tanh^2\! d-1)t^2 + 2t - \tanh^2\! \alpha, 
    \end{equation*}
    and let $F_{d'}(t)$ be similarly defined, \ie 
    \begin{equation*}
        F_{d'}(t) \df (\tanh^2\! \alpha\tanh^2\! d'-\tanh^2\! d'-1)t^2 + 2t - \tanh^2\! \alpha. 
    \end{equation*}
    By a direct computation we have 
    \begin{equation*}
        \frac{\dif}{\dif t}\Big( \frac{F_{d}}{F_{d'}} \Big) 
        = \frac{2t\cdot (t-\tanh^2 \alpha)(\tanh^2\! \alpha-1)(\tanh^2\! d-\tanh^2\! d')}{(F_{d'})^2}, 
    \end{equation*}
    which tells that $F_{d}(t)/F_{d'}(t)$ is a monotonic function for $t\in[\tanh^2\!\alpha,1]$. Thus we conclude that 
    \begin{equation*}
        \frac{\cosh h'\cosh u'}{\cosh h\cosh u} = \frac{\tanh^2\! d\tanh^2\! c'}{\tanh^2\! d'\tanh^2\! c}\cdot \frac{F_{d}}{F_{d'}} 
    \end{equation*}
    is also a monotonic function for $u\in[u_0,d]$. 
    It follows that for any $u\in[u_0,d]$ we have 
    \begin{equation}\label{eqn:mu_lem5-3}
        \frac{\cosh h'\cosh u'}{\cosh h\cosh u} \leq \max\left\{ \frac{\cosh h'_0\cosh u'_0}{\cosh h_0\cosh u_0},\, \frac{\cosh \gamma'\cosh d'}{\cosh \gamma\cosh d} \right\}. 
    \end{equation}
    By Lemma \ref{lem:trirectangle} we have 
    \begin{equation}\label{eqn:mu_lem5-4}
        \cosh h_0\cosh u_0 = \frac{\sinh h_0\cosh u_0}{\tanh h_0}
        = \frac{\cosh\alpha\sinh c\cosh u_0}{\cosh u_0\tanh c}
        = \cosh \alpha\cosh c, 
    \end{equation}
    and by Lemma \ref{lem:pentagon} we have 
    \begin{equation}\label{eqn:mu_lem5-5}
        \begin{aligned}
            \cosh \alpha\cosh c &= \cosh \alpha\coth \alpha \coth d,\\
            \cosh\gamma &= \sinh\alpha\sinh c. 
        \end{aligned} 
    \end{equation}
    Thus by \eqref{eqn:mu_lem5-3}, \eqref{eqn:mu_lem5-4} and \eqref{eqn:mu_lem5-5}, we have 
    \begin{equation*}
        \frac{\cosh h'}{\cosh h} \leq \max\left\{\frac{\tanh d\cosh u}{\tanh d'\cosh u'},\, \frac{\sinh c'\cosh d'\cosh u}{\sinh c\cosh d\cosh u'}\right\}. 
    \end{equation*}
    Then it follows from Lemma \ref{lem:mu_lem1}, \eqref{eqn:mu_lem2-1} and \eqref{eqn:mu_lem3-1} that 
    \begin{equation}\label{eqn:mu_lem5-6}
        \frac{\cosh h'}{\cosh h} \leq (1+2\delta_d)^2(1+\delta_d) \leq 1+7\delta_d. 
    \end{equation}
    Similarly, we also have 
     \begin{equation*}
        \frac{\cosh h'}{\cosh h} \geq \min\left\{\frac{\tanh d\cosh u}{\tanh d'\cosh u'},\, \frac{\sinh c'\cosh d'\cosh u}{\sinh c\cosh d\cosh u'}\right\}, 
    \end{equation*}
    which together with Lemma \ref{lem:mu_lem1}, \eqref{eqn:mu_lem2-1} and \eqref{eqn:mu_lem3-1} implies that
    \begin{equation}\label{eqn:mu_lem5-7}
        \frac{\cosh h'}{\cosh h} \geq(1-\delta_d)^3 \geq 1-3\delta_d, 
    \end{equation}
    which finishes the proof. 
\end{proof}

\begin{lemma}\label{lem:mu_lem6}
    Suppose $\delta_d\leq1/5$. Then for any $u\in[0,d]$ and any $v\in[0,h(u)]$ we have 
    \begin{equation}\label{eqn:mu_lem6-1}
        1-3\delta_d \leq \cosh v'/\cosh v \leq 1+7\delta_d. 
    \end{equation}
\end{lemma}
\begin{proof}
    For any fixed $u\in[0,d]$, by definition we have 
    \begin{equation*}
        \frac{\cosh v'}{\cosh v} = \bigg(\frac{1}{\cosh^2\!v - \frac{\tanh^2\! h'}{\tanh^2\! h}\sinh^2\! v}\bigg)^{1/2}, 
    \end{equation*}
    hence it is a monotonic function with respect to $v\in[0,h(u)]$. It follows that 
    \begin{equation*}
        \min\Big\{ 1,\, \frac{\cosh h'}{\cosh h} \Big\} \leq \frac{\cosh v'}{\cosh v} \leq \max\Big\{ 1,\, \frac{\cosh h'}{\cosh h} \Big\}. 
    \end{equation*}
    Then the conclusion follows from \eqref{eqn:mu_lem5-1}. 
\end{proof}

Now we can go back to the proof of Lemma \ref{lem:mu_pf2}. Recall that $h(u_0) = h_0$, and the map $\mu$ from $\pent$ to $\pent'$ is defined as follows.
\[ (u,v) \mapsto (u',v') = \bigg( \arctanh \Big( \frac{\tanh d'}{\tanh d}\tanh u \Big), \arctanh \Big( \frac{\tanh h'(u')}{\tanh h(u)}\tanh v \Big) \bigg). \] 

\begin{lemma}\label{lem:mu_pf2-1}
    Suppose $\delta_d \leq 1/15$. Then \eqref{eqn:mu_pf2-1} holds for $\delta_0 = 10\sqrt{\delta_d}$. 
\end{lemma}
\begin{proof}
    By definition we have 
    \begin{equation*}
        \dfrac{\pa u'}{\pa u}  = \frac{\tanh d'}{\tanh d} \frac{\cosh^2\! u'}{\cosh^2\! u}, 
    \end{equation*}
    thus by \eqref{eqn:mu_lem1-1} and \eqref{eqn:mu_lem2-1} we have 
    \begin{equation*}
        1-6\delta_d \leq (1-\delta_d)^6 \leq \Big(\frac{\pa u'}{\pa u}\Big)^2 \leq (1+\delta_d)^6 \leq 1+11\delta_d. 
    \end{equation*}
    Then the conclusion follows. 
\end{proof}

\begin{lemma}\label{lem:mu_pf2-2}
    Suppose $\delta_d \leq 1/735$. Then \eqref{eqn:mu_pf2-2} holds for $\delta_0 = 21\sqrt{\delta_d}$. 
\end{lemma}
\begin{proof}
    By definition we have 
    \begin{equation*}
        \dfrac{\pa v'}{\pa v}  = \frac{\tanh h'}{\tanh h} \frac{\cosh^2\! v'}{\cosh^2\! v}, 
    \end{equation*}
    thus by \eqref{eqn:mu_lem4-1} and \eqref{eqn:mu_lem6-1} we have 
    \begin{equation*}
        1-18\delta_d \leq (1-3\delta_d)^6 \leq \Big(\frac{\pa v'}{\pa v}\Big)^2 \leq (1+7\delta_d)^6 \leq 1+47\delta_d. 
    \end{equation*}
    Then the conclusion follows. 
\end{proof}

\begin{lemma}\label{lem:mu_pf2-3}
    Suppose $\delta_d\leq1/15$. Then \eqref{eqn:mu_pf2-3} holds for $\delta_0 = 54e^{2\bar h}\delta_d$. 
\end{lemma}
\begin{proof}
    By definition we have 
    \begin{equation}\label{eqn:mu_pf2-3-1}
        \abs{\frac{\pa v'}{\pa u}} = \abs{ \frac{\pa}{\pa u}\Big(\frac{\tanh h'}{\tanh h}\Big)\cdot \tanh v\cdot \cosh^2\! v' } \leq \abs{ \frac{\pa}{\pa u}\Big(\frac{\tanh h'}{\tanh h}\Big) }\cdot e^{2\bar h}. 
    \end{equation}

    If $u\in[0,u_0]$, then by Lemma \ref{lem:trirectangle} we have $\tanh h = \cosh u \tanh c$, thus 
    \begin{equation}\label{eqn:mu_pf2-3-2}
        \frac{\pa}{\pa u}(\tanh h) = \sinh u\tanh c = \tanh u\tanh h. 
    \end{equation}
    Similarly, for $\tanh h'$ we also have 
    \begin{equation}\label{eqn:mu_pf2-3-3}
        \frac{\pa}{\pa u'}(\tanh h') = \sinh u'\tanh c' = \tanh u'\tanh h'. 
    \end{equation}
    It follows from \eqref{eqn:mu_pf2-3-2} and \eqref{eqn:mu_pf2-3-3} that 
    \begin{equation}\label{eqn:mu_pf2-3-4}
        \begin{aligned}
            \abs{ \frac{\pa}{\pa u}\Big(\frac{\tanh h'}{\tanh h}\Big) }
            & = \abs{ \frac{\tanh u'\tanh h'\tanh h\cdot \frac{\pa u'}{\pa u} - \tanh u\tanh h\tanh h'}{\tanh^2\! h} }\\
            & = \abs{ \frac{\tanh u'\cdot \frac{\pa u'}{\pa u} - \tanh u}{\tanh u}\cdot \tanh u\cdot \frac{\tanh h'}{\tanh h} }\\
            &\leq \abs{\frac{\tanh d'}{\tanh d}\frac{\pa u'}{\pa u}-1}\cdot \frac{\tanh h'}{\tanh h}. 
        \end{aligned}
    \end{equation}
    Thus by \eqref{eqn:mu_lem1-1}, \eqref{eqn:mu_lem4-1} and Lemma \ref{lem:mu_pf2-1} we have 
    \begin{equation}\label{eqn:mu_pf2-3-5}
        \abs{ \frac{\pa}{\pa u}\Big(\frac{\tanh h'}{\tanh h}\Big) } \leq 14\delta_d. 
    \end{equation}

    If $u\in[u_0,d]$, then by Lemma \ref{lem:trirectangle} we have $\tanh h = \cosh(d-u)\tanh \gamma$, thus 
    \begin{equation*}
        \frac{\pa }{\pa u}(\tanh h) = -\sinh(d-u)\tanh \gamma = -\tanh(d-u)\tanh h. 
    \end{equation*}
    Similar to \eqref{eqn:mu_pf2-3-4}, we have 
    \begin{equation*}
        \abs{ \frac{\pa}{\pa u}\Big(\frac{\tanh h'}{\tanh h}\Big) } \leq \abs{\frac{\tanh (d'-u')}{\tanh (d-u)}\frac{\pa u'}{\pa u}-1}\cdot \frac{\tanh h'}{\tanh h}. 
    \end{equation*}
    Thus by \eqref{eqn:mu_lem2-2}, \eqref{eqn:mu_lem4-1} and Lemma \ref{lem:mu_pf2-1} we have 
    \begin{equation}\label{eqn:mu_pf2-3-6}
        \abs{ \frac{\pa}{\pa u}\Big(\frac{\tanh h'}{\tanh h}\Big) } \leq 18\delta_d. 
    \end{equation}
    By \eqref{eqn:mu_pf2-3-1}, \eqref{eqn:mu_pf2-3-5} and \eqref{eqn:mu_pf2-3-6} we have 
    \begin{equation*}
        \abs{\frac{\pa v'}{\pa u}} \leq 18e^{2\bar h}\delta_d. 
    \end{equation*}
    Then the conclusion follows. 
\end{proof}

\begin{lemma}\label{lem:mu_pf2-4}
    Suppose $\delta_d\leq1/49$. Then \eqref{eqn:mu_pf2-4} holds for $\delta_0 = 12\sqrt{\delta_d}$. 
\end{lemma}
\begin{proof}
    It follows directly from \eqref{eqn:mu_lem6-1}. 
\end{proof}

Now we complete the proof of Lemma \ref{lem:mu_pf2}.

\begin{proof}[Proof of Lemma \ref{lem:mu_pf2}]
It is clear that the conclusion follows directly from Lemma \ref{lem:mu_pf2-1}, \ref{lem:mu_pf2-2}, \ref{lem:mu_pf2-3} and \ref{lem:mu_pf2-4}. 
\end{proof}

\section{Several further remarks}\label{sect-last}
In this last section, we make several remarks. 

First, we recall the following two well-known results:
\begin{enumerate}
\item $\mathrm{Area}(B_\H(0,r))=2\pi(\cosh(r)-1).$

\item Let $X_g$ be a closed hyperbolic surface of genus $g$ $(g>1)$. For any $p\in X_g$ and $r>0$, we let $B(p,r)\subset X_g$ be a geodesic ball of radius $r>0$ centered at $p$, then it is known from \eg \cite[Theorem 1.1]{Cheng1975} that
\begin{equation}\label{ls-e-1}
    \lambda_1(B(p,r))\leq \lambda_1(B_{\H}(0,r))\leq \frac{1}{4}+\left(\frac{2\pi}{r}\right)^2
\end{equation}
where $\lambda_1$ is the first eigenvalue with respect to Dirichlet boundary condition.
\end{enumerate}

Then we show
\begin{proposition}\label{t-up}
For any $k(g)>0$,
\[\sup\limits_{X_g\in \sM_g}\lambda_{k(g)}(X_g)\leq \frac{1}{4}+\left(\frac{4\pi}{\cosh^{-1}\left(1+\frac{2(g-1)}{k(g)+1} \right)}\right)^2.\]
In particular, for $k(g)=o(g)$,
\[\limsup\limits_{g\to \infty}\sup\limits_{X_g\in \sM_g}\lambda_{k(g)}(X_g)\leq \frac{1}{4}.\]
\end{proposition}
\begin{proof}
Set $$r(g)=\cosh^{-1}\left(1+\frac{2(g-1)}{k(g)+1} \right).$$
Let $S=\{p_i\}\subset X_g$ be a maximal $r(g)$-net of $X_g$. That is, $\dist(p_i,p_j)\geq r(g)$ for all $i\neq j$ and $\cup_i B(p_i,r(g))=X_g$. Then it follows from the Gauss-Bonnet formula that
\[4\pi(g-1)=\mathrm{Area}(X_g)\leq \sum_i \mathrm{Area}\big(B(p_i,r(g))\big)\leq |S|\cdot \mathrm{Area}\big(B_\H(0,r(g))\big)\]
which implies that
\[|S|\geq k(g)+1.\]
By the choice of $S$ we know that $$B\left(p_i,\frac{r(g)}{2}\right)\cap B\left(p_j,\frac{r(g)}{2}\right)=\emptyset$$ for all $i\neq j$. Then it follows from  the Mini-max principle (see \eg Part (1) of Theorem \ref{thm:mini-max}) and \eqref{ls-e-1} that
\begin{equation}
\begin{aligned}
  \lambda_{k(g)}(X_g)&\leq \max_{i}\left\{\lambda_1\left(B\left(p_i, \frac{r(g)}{2}\right)\right)\right\}\leq \frac{1}{4}+\left(\frac{4\pi}{r(g)}\right)^2\\
&=\frac{1}{4}+\left(\frac{4\pi}{\cosh^{-1}\left(1+\frac{2(g-1)}{k(g)+1} \right)}\right)^2. 
\end{aligned}  \nonumber  
\end{equation}
The proof is complete.
\end{proof}

We finish this section with the following two consequences.

Joint with Zhu, we show \cite[Theorem 4.1]{WZZ2024} that for any $\eta(g) \in [1,2g-2]$, 
    \[ \liminfg \sup_{X_g\in\sM_g}\left( \lambda_{\eta(g)}(X_g) - \lambda_{\eta(g)-1}(X_g) \right) \geq \frac 14. \] 
This together with Proposition \ref{t-up} implies that
\begin{theorem}
    If the indices $\eta(g) \in [1,2g-2]$ satisfy that $$\lim\limits_{g\to \infty}\frac{\eta(g)}{g}=0,$$ then
    \[ \limg \sup_{X_g\in\sM_g}\left( \lambda_{\eta(g)}(X_g) - \lambda_{\eta(g)-1}(X_g) \right) = \frac14. \] 
\end{theorem}

Secondly, consider the first two \emph{distinct} positive eigenvalues $\lambda_1^{\mathrm{nm}}(X_g), \lambda_2^{\mathrm{nm}}(X_g)$ of $X_g \in \sM_g$. It is clear that
\[\lambda_1(X_g)=\lambda_1^{\mathrm{nm}}(X_g)<\lambda_2^{\mathrm{nm}}(X_g).\]
The main result of this paper, \ie Theorem \ref{main}, tells that for any fixed $\epsilon>0$,
 \begin{equation}\label{e-spg-m-1}
     \liminfg \max_{X_g \in \sM_g^{\geq\epsilon}} \big(  \lambda_2^{\mathrm{nm}}(X_g)-  \lambda_1^{\mathrm{nm}}(X_g) \big) \geq \frac 14.   
 \end{equation}
On the other hand, it is known from a quite recent result of Letrouit and Machado \cite[Theorem 1.1]{LM-2024-mul} that for sufficiently large $g$,
\[\max_{X_g \in \sM_g^{\geq\epsilon}}\lambda_2^{\mathrm{nm}}(X_g)\leq \max_{X_g \in \sM_g^{\geq\epsilon}}\lambda_{k(g)}(X_g)\]
where
\[k(g)=\left\lfloor\frac{g}{\log \log \log g}\right\rfloor.\]
This together with Proposition \ref{t-up} implies that
\begin{equation}\label{e-spg-m-2}
    \limsupg \max_{X_g \in \sM_g^{\geq\epsilon}}  \lambda_2^{\mathrm{nm}}(X_g)  \leq \frac{1}{4}.
\end{equation}
Combining \eqref{e-spg-m-1} and \eqref{e-spg-m-2}, we have
\begin{theorem}
    The following limit holds:
    \begin{equation}
        \limg \max_{X_g \in \sM_g^{\geq\epsilon}}  \big(\lambda_2^{\mathrm{nm}}(X_g) - \lambda_1^{\mathrm{nm}}(X_g)\big)=\frac{1}{4}.\nonumber
    \end{equation}  
\end{theorem}

\bibliographystyle{amsalpha}
\bibliography{ref}

\end{document}

%% file: preamble.tex
\usepackage{amsmath, amsthm, amssymb, amsfonts}
\usepackage{graphicx, tikz, units, url}
\usepackage{color}

\usepackage{mathrsfs}
\usepackage{bm,bbm} 
\usepackage{colonequals} 
\usepackage{physics}
\usepackage{extarrows}
\usepackage{mathtools} 

\usepackage{enumitem} 
\usepackage{scalerel}
\usepackage{stackengine} 

\usepackage[colorlinks,citecolor=blue,linkcolor=red,backref=page]{hyperref}
\usepackage[nameinlink]{cleveref}

\newcommand{\sM}{\mathcal{M}}

\newcommand{\sP}{\mathcal{P}}

\newcommand{\sX}{\mathcal{X}}
\newcommand{\sY}{\mathcal{Y}}
\newcommand{\sZ}{\mathcal{Z}}

\newcommand{\Z}{\mathbb{Z}}

\newcommand{\R}{\mathbb{R}}
\newcommand{\C}{\mathbb{C}}
\newcommand{\D}{\mathbb{D}}
\renewcommand{\H}{\mathbb{H}}

\DeclareMathOperator{\Forall}{\forall}
\DeclareMathOperator{\arcsinh}{arcsinh}
\DeclareMathOperator{\arccosh}{arccosh}
\DeclareMathOperator{\arctanh}{arctanh}
\DeclareMathOperator{\dist}{dist}
\DeclareMathOperator{\length}{\ell} 
\DeclareMathOperator{\area}{Area}

\DeclareMathOperator{\sys}{sys}

\DeclareMathOperator{\jac}{Jac}

\DeclareMathOperator{\im}{Im}
\DeclareMathOperator{\collar}{\mathcal{C}} 
\DeclareMathOperator{\cusp}{Cusp} 
\DeclareMathOperator{\horo}{horo} 
\DeclareMathOperator{\prob}{Prob}
\DeclareMathOperator{\pent}{Pent}
\DeclareMathOperator{\hexa}{Hexa}

\newcommand{\dif}{\mathop{}\!\mathrm{d}}
\newcommand{\dvol}{\mathop{}\!\mathrm{dvol}}

\newcommand{\spec}{\mathrm{Spec}}

\newcommand{\Hom}{\mathrm{Hom}}

\newcommand{\RQ}{\mathrm{RayQ}}

\newcommand{\pa}{\partial}
\newcommand{\df}{\colonequals}
\newcommand{\limg}{\lim_{g\rightarrow\infty}}
\newcommand{\limsupg}{\limsup_{g\rightarrow\infty}}
\newcommand{\liminfg}{\liminf_{g\rightarrow\infty}}

\newcommand{\eg}{\textit{e.g.,\@ }}
\newcommand{\cf}{\textit{c.f.\@ }}
\newcommand{\ie}{\textit{i.e.\@ }}

\theoremstyle{plain}

\newtheorem{theorem}{Theorem}[section] 
\newtheorem{lemma}[theorem]{Lemma} %
\newtheorem{proposition}[theorem]{Proposition} %
\newtheorem{corollary}[theorem]{Corollary} %

\newtheorem*{construction*}{Construction}
\newtheorem*{theorem*}{\bf Theorem}
\newtheorem*{corollary*}{Corollary}
\newtheorem*{assumption*}{Assumption $(\star)$}
\newtheorem*{observation*}{Observation}
\newtheorem*{claim*}{Claim}

\theoremstyle{remark}
\newtheorem*{remark*}{Remark}
\newtheorem*{question*}{Question}

\theoremstyle{definition}
\newtheorem*{definition*}{Definition}

%% file: main.bbl
\providecommand{\bysame}{\leavevmode\hbox to3em{\hrulefill}\thinspace}
\providecommand{\MR}{\relax\ifhmode\unskip\space\fi MR }
\providecommand{\MRhref}[2]{%
  \href{http://www.ams.org/mathscinet-getitem?mr=#1}{#2}
}
\providecommand{\href}[2]{#2}
\begin{thebibliography}{MNP22}

\bibitem[BBD88]{BBD1988}
Peter Buser, Marc Burger, and J{\'o}zef Dodziuk, \emph{Riemann surfaces of large genus and large {$\lambda_1$}}, Geometry and Analysis on Manifolds (Toshikazu Sunada, ed.), Lecture Notes in Mathematics, vol. 1339, Springer Berlin, Heidelberg, 1988, pp.~54--63.

\bibitem[Ber16]{Bergeron2016}
Nicolas Bergeron, \emph{{The Spectrum of Hyperbolic Surfaces}}, Universitext, Springer, Cham; EDP Sciences, Les Ulis, 2016, Appendix C by Valentin Blomer and Farrell Brumley, Translated from 2011 French language edition by Farrell Brumley. \MR{2857626}

\bibitem[Bis02]{Bishop2002}
Christopher~J. Bishop, \emph{Quasiconformal mappings of y-pieces}, Revista Matematica Iberoamericana \textbf{18} (2002), 627--652.

\bibitem[BM01]{BM2001}
Robert Brooks and Eran Makover, \emph{Riemann surfaces with large first eigenvalue}, Journal d'Analyse Math{\'e}matique \textbf{83} (2001), 243--258.

\bibitem[Bor16]{Borthwick2016}
David Borthwick, \emph{{Spectral Theory of Infinite-Area Hyperbolic Surfaces}}, 2 ed., Progress in Mathematics, vol. 318, Birkhäuser Cham, 2016.

\bibitem[Bro99]{Brooks1999}
Robert Brooks, \emph{Platonic surfaces}, Commentarii Mathematici Helvetici \textbf{74} (1999), 156--170.

\bibitem[Bus78]{Buser1978}
Peter Buser, \emph{Cubic graphs and the first eigenvalue of a {R}iemann surface}, Mathematische Zeitschrift \textbf{162} (1978), 87--99.

\bibitem[Bus84]{Buser1984}
\bysame, \emph{On the bipartition of graphs}, Discrete Applied Mathematics \textbf{9} (1984), 105--109.

\bibitem[Bus92]{Buser1992}
\bysame, \emph{{Geometry and Spectra of Compact Riemann Surfaces}}, Progress in Mathematics, vol. 106, Birkh\"auser Boston, Inc., Boston, MA, 1992. \MR{1183224}

\bibitem[Cha84]{Chavel1984}
Isaac Chavel, \emph{{Eigenvalues in Riemannian Geometry}}, Pure and Applied Mathematics, vol. 115, Academic Press, Orlando, FL, 1984, Including a chapter by Burton Randol, with an appendix by J{\'o}zef Dodziuk. \MR{0768584}

\bibitem[Che75]{Cheng1975}
Shiu-yuen Cheng, \emph{Eigenvalue comparison theorems and its geometric applications}, Mathematische Zeitschrift \textbf{143} (1975), 289--297.

\bibitem[Dix69]{Dixon1969}
John~Douglas Dixon, \emph{The probability of generating the symmetric group}, Mathematische Zeitschrift \textbf{110} (1969), 199 -- 205.

\bibitem[Gam02]{Gamburd2002}
Alex Gamburd, \emph{On the spectral gap for infinite index “congruence” subgroups of {$\mathrm{SL}_2(\mathbb{Z})$}}, Israel Journal of Mathematics \textbf{127} (2002), 157--200.

\bibitem[HM23]{HM2023}
Will Hide and Michael Magee, \emph{Near optimal spectral gaps for hyperbolic surfaces}, Annals of Mathematics \textbf{198} (2023), no.~2, 791 -- 824.

\bibitem[HP24]{HP24}
Will Hide and Bram Petri, \emph{Limit points of uniform arithmetic bass notes}, preprint (2024), arXiv:2412.15111.

\bibitem[Hub74]{Huber1974}
Heinrich Huber, \emph{{{\"U}ber den ersten Eigenwert des Laplace-Operators auf kompakten Riemannschen Fl{\"a}chen}}, Commentarii Mathematici Helvetici \textbf{49} (1974), 251--259.

\bibitem[JL70]{JL1970}
Herv{\'e} Jacquet and Robert~Phelan Langlands, \emph{{Automorphic Forms on $\mathrm{GL}(2)$}}, Lecture Notes in Mathematics, vol. 114, Springer Berlin, Heidelberg, 1970.

\bibitem[Jos06]{Jost2006}
J{\"u}rgen Jost, \emph{Compact {R}iemann surfaces}, 3 ed., Universitext, Springer Berlin, Heidelberg, 2006.

\bibitem[Kim03]{Kim03}
Henry~H. Kim, \emph{Functoriality for the exterior square of {${\rm GL}_4$} and the symmetric fourth of {${\rm GL}_2$}}, Journal of the American Mathematical Society \textbf{16} (2003), no.~1, 139--183, With appendix 1 by Dinakar Ramakrishnan and appendix 2 by Kim and Peter Sarnak. \MR{1937203}

\bibitem[KM24]{KM2024}
Adam Klukowski and Vladimir Marković, \emph{Tangle free permutations and the {P}utman-{W}ieland property of random covers}, International Mathematics Research Notices (2024), to appear.

\bibitem[LM22]{LM2022}
Larsen Louder and Michael Magee, \emph{Strongly convergent unitary representations of limit groups}, Journal of Functional Analysis (2022), to appear, With an Appendix by Will Hide and Michael Magee.

\bibitem[LM24]{LM-2024-mul}
Cyril Letrouit and Simon Machado, \emph{Maximal multiplicity of {L}aplacian eigenvalues in negatively curved surfaces}, Geometric and Functional Analysis \textbf{34} (2024), no.~5, 1609--1645.

\bibitem[LP10]{LP2010}
Nati Linial and Doron Puder, \emph{Word maps and spectra of random graph lifts}, Random Structures \& Algorithms \textbf{37} (2010), no.~1, 100--135.

\bibitem[LRS95]{LRS95}
Wenzhi Luo, Ze\'ev Rudnick, and Peter Sarnak, \emph{On {S}elberg's eigenvalue conjecture}, Geometric and Functional Analysis \textbf{5} (1995), no.~2, 387--401. \MR{1334872}

\bibitem[Mag24]{Magee2024}
Michael Magee, \emph{The limit points of the bass notes of arithmetic hyperbolic surfaces}, preprint (2024), arxiv:2403.00928.

\bibitem[MN20]{MN2020}
Michael Magee and Fr{\'e}d{\'e}ric Naud, \emph{Explicit spectral gaps for random covers of {R}iemann surfaces}, Publications math{\'e}matiques de l'IH{\'E}S \textbf{132} (2020), 137--179.

\bibitem[MNP22]{MNP2022}
Michael Magee, Fr{\'e}d{\'e}ric Naud, and Doron Puder, \emph{A random cover of a compact hyperbolic surface has relative spectral gap {$\frac{3}{16}-\varepsilon$}}, Geometric and Functional Analysis \textbf{32} (2022), 595--661.

\bibitem[Mon15]{Mondal2015}
Sugata Mondal, \emph{On topological upper-bounds on the number of small cuspidal eigenvalues of hyperbolic surfaces of finite area}, International Mathematics Research Notices \textbf{2015} (2015), no.~24, 13208--13237.

\bibitem[MP23]{MP2023}
Michael Magee and Doron Puder, \emph{The asymptotic statistics of random covering surfaces}, Forum of Mathematics, Pi \textbf{11} (2023), e15:1--51.

\bibitem[Mum71]{Mumford1971}
David Mumford, \emph{A remark on {M}ahler's compactness theorem}, Proceedings of the American Mathematical Society \textbf{28} (1971), no.~1, 289--294.

\bibitem[Nau22]{Naud2022}
Fr{\'e}d{\'e}ric Naud, \emph{Random covers of compact surfaces and smooth linear spectral statistics}, Annales Henri Poincaré (2022), to appear.

\bibitem[Nau23]{Naud2023}
\bysame, \emph{{Determinants of Laplacians on random hyperbolic surfaces}}, Journal d'Analyse Math{\'e}matique \textbf{151} (2023), 265--291.

\bibitem[Nic94]{Nica1994}
Alexandru Nica, \emph{On the number of cycles of given length of a free word in several random permutations}, Random Structures \& Algorithms \textbf{5} (1994), no.~5, 703--730.

\bibitem[Par05]{Parlier2005}
Hugo Parlier, \emph{Lengths of geodesics on riemann surfaces with boundary}, Annales Academiae Scientiarum Fennicae. Mathematica \textbf{30} (2005), no.~2, 227--236.

\bibitem[PS85]{PS1985}
Ralph Phillips and Peter Sarnak, \emph{On cusp forms for co-finite subgroups of {$\mathrm{PSL}(2,\mathbb{R})$}}, Inventiones mathematicae \textbf{80} (1985), 339--364.

\bibitem[PT12]{PT2012}
Athanase Papadopoulos and Guillaume Th{\'e}ret, \emph{{Some Lipschitz maps between hyperbolic surfaces with applications to Teichm{\"u}ller theory}}, Geometriae Dedicata \textbf{161} (2012), 63 -- 83.

\bibitem[PY15]{PY2015}
Athanase Papadopoulos and Sumio Yamada, \emph{{Deforming hyperbolic hexagons with applications to the arc and the Thurston metrics on Teichm{\"u}ller spaces}}, Monatshefte f{\"u}r Mathematik \textbf{182} (2015), 913 -- 939.

\bibitem[PZ24]{PZ2024}
Doron Puder and Tomer Zimhoni, \emph{Local statistics of random permutations from free products}, International Mathematics Research Notices \textbf{2024} (2024), no.~5, 4242--4300.

\bibitem[RS78]{RS1978}
Michael Reed and Barry Simon, \emph{{Analysis of Operators}}, Methods of Modern Mathematical Physics, vol.~IV, Academic Press [Harcourt Brace Jovanovich, Publishers], New York-London, 1978. \MR{493421}

\bibitem[Sar03]{Sarnak2003}
Peter Sarnak, \emph{Spectra of hyperbolic surfaces}, Bulletin of the American Mathematical Society \textbf{40} (2003), 441--478.

\bibitem[Sel65]{Selberg1965}
Atle Selberg, \emph{{On the estimation of Fourier coefficients of modular forms}}, Theory of Numbers (Providence, R.I.) (Albert~Leon Whiteman, ed.), Proceedings of Symposia in Pure Mathematics, vol. VIII, American Mathematical Society, 1965, pp.~1--15.

\bibitem[SW22]{SW2022}
Yang Shen and Yunhui Wu, \emph{Arbitrarily small spectral gaps for random hyperbolic surfaces with many cusps}, preprint (2022), arXiv:2203.15681.

\bibitem[WX22]{WX2022}
Yunhui Wu and Yuhao Xue, \emph{Random hyperbolic surfaces of large genus have first eigenvalues greater than {$\frac{3}{16}-\epsilon$}}, Geometric and Functional Analysis \textbf{32} (2022), 340--410.

\bibitem[WZZ24]{WZZ2024}
Yunhui Wu, Haohao Zhang, and Xuwen Zhu, \emph{Degenerating hyperbolic surfaces and spectral gaps for large genus}, Analysis \& PDE \textbf{17} (2024), no.~4, 1377--1395.

\bibitem[Zog87]{Zograf1987}
Peter Zograf, \emph{Small eigenvalues of automorphic laplacians in spaces of parabolic forms}, Journal of Soviet Mathematics \textbf{36} (1987), 106--114.

\end{thebibliography}
